\documentclass[12pt,reqno]{amsart}

\usepackage[colorlinks,linkcolor={blue},
citecolor={blue},urlcolor={blue}]{hyperref}

\usepackage{mathrsfs}
\usepackage{bbm}
\usepackage{amsfonts}
\usepackage{}
\usepackage{pifont}
\usepackage[shortlabels]{enumitem}
\usepackage{empheq}
\usepackage{amsfonts}
\usepackage{amsmath, amsrefs, amsthm, amssymb}
\usepackage[active]{srcltx}		
\usepackage{pdfsync}					
\usepackage{graphicx}
\usepackage [latin1]{inputenc}
\usepackage{bbm}
\usepackage{amsfonts}
\usepackage{amsthm}
\usepackage{graphicx}
\usepackage{framed}
\usepackage{multicol,multienum}
\usepackage{graphicx} 
\usepackage{booktabs} 
\usepackage{mathrsfs}
\usepackage{tcolorbox}
\usepackage{color}
\usepackage{xcolor}

\newcommand{\wuhao}{\fontsize{10.5pt}{\baselineskip}\selectfont}

\newtheorem{theorem}{Theorem}[section]
\newtheorem{proposition}[theorem]{Proposition}
\newtheorem{lemma}[theorem]{Lemma}
\newtheorem{corollary}[theorem]{Corollary}

\theoremstyle{definition}
\newtheorem{definition}[theorem]{Definition}

\newtheorem{remark}[theorem]{Remark}

\numberwithin{equation}{section}

\begin{document}

\title [The SLLBar equation with jump noise]{\wuhao Well-posedness and large deviations of  L\'{e}vy-driven Marcus stochastic Landau-Lifshitz-Baryakhtar equation}
\author{Fan Xu}
\address{School of Mathematics and Statistics, Hubei Key Laboratory of Engineering Modeling  and Scientific Computing, Huazhong University of Science and Technology,  Wuhan 430074, Hubei, P.R. China.}
\email{d202280019@hust.edu.cn (F. Xu)}

\author{Bin Liu}
\address{School of Mathematics and Statistics, Hubei Key Laboratory of Engineering Modeling  and Scientific Computing, Huazhong University of Science and Technology,  Wuhan 430074, Hubei, P.R. China.}
\email{binliu@mail.hust.edu.cn (B. Liu)}

\author{Lei Zhang}
\address{School of Mathematics and Statistics, Hubei Key Laboratory of Engineering Modeling  and Scientific Computing, Huazhong University of Science and Technology,  Wuhan 430074, Hubei, P.R. China.}
\email{lei\_zhang@hust.edu.cn (L. Zhang)}

\keywords{Stochastic Landau-Lifshitz-Baryakhtar equation; Well-posedness; Large deviations principle; L\'{e}vy noise.}


\begin{abstract}
This paper considers the stochastic Landau-Lifshitz-Baryakhtar (SLLBar) equation with pure jump noise in Marcus canonical form, which describes the dynamics of magnetic spin field in a ferromagnet at elevated temperatures with the effective field $\mathbf{H}_{\textrm{eff}}$ influenced by external random noise. Under the natural assumption that the magnetic body $\mathcal{O}\subset\mathbb{R}^d$ ($d=1,2,3$) is  bounded with smooth boundary, we shall prove that the initial-boundary value problem of SLLBar equation possesses a unique global probabilistically strong and analytically weak solution with initial data in the energy space $\mathbb{H}^1(\mathcal{O})$.
Then by employing the weak convergence method, we proceed to establish a  Freidlin-Wentzell type large deviation principle for pathwise solutions to the SLLBar equation.
\end{abstract}

\maketitle
\section{Introduction}\label{sec1}

The theory of ferromagnetism began with Weiss's work in 1907 \cite{brown1963micromagnetics,weiss1907hypothese}. Later, in 1935, Landau and Lifshitz \cite{landau1935theory} developed the dispersive theory of magnetization for ferromagnets and introduced the Landau-Lifshitz (LL) equation for ferromagnetic spin chains. Gilbert \cite{gilbert1955lagrangian} further advanced the theory in 1955, proposing the Landau-Lifshitz-Gilbert (LLG) equation to describe the evolution of the spin magnetic moment in magnetic systems, particularly addressing the precession and dissipation behavior under an external magnetic field. In 1997, Garanin \cite{garanin1997fokker} presented a thermodynamically consistent approach and derived the Landau-Lifshitz-Bloch (LLB) equation. In mathematics, the existence, uniqueness, and regularity of solutions to the LLG and LLB equations have been extensively discussed during the past decades, see for example \cite{alouges1992global,le2016weak,feischl2017existence,carbou2001regular,carbou2001regular1} and the reference therein. To explain certain experimental data and microscopic calculations, such as the non-local damping observed in magnetic metals and crystals or the higher-than-expected spin wave attenuation of short-wavelength magnons, Baryakhtar \cite{baryakhtar1984phenomenological,baryakhtar1997soliton,baryakhtar2013phenomenological} extended the LLG and LLB equations, introducing the so-called Landau-Lifshitz-Baryakhtar (LLBar) equation \cite{baryakhtar1984phenomenological,baryakhtar1997soliton,baryakhtar2013phenomenological}, which is a fourth-order nonlinear parabolic equation, and the dynamic behavior of solutions is described by
\begin{equation}\label{sys0}
\left\{
\begin{aligned}
\frac{\textrm{d}\mathbf{u}}{\textrm{d}t}&=\lambda_r\mathbf{H}_{\textrm{eff}}-\lambda_e\Delta\mathbf{H}_{\textrm{eff}}-\gamma\mathbf{u}\times \mathbf{H}_{\textrm{eff}},\\
\mathbf{H}_{\textrm{eff}}&=\Delta \mathbf{u}+\frac{1}{2\chi}(1-|\mathbf{u}|^2)\mathbf{u},
\end{aligned}
\right.
\end{equation}
where the unknown quantity $\mathbf{u}(t,x)\in\mathbb{R}^3$ denotes the magnetization vector of a magnetic body $\mathcal{O}\subset\mathbb{R}^d$, $d=1,2,3$. The positive constants $\lambda_r$, $\lambda_e$, and $\gamma$ are the relativistic damping constant, the exchange damping constant, and the electron gyromagnetic ratio, respectively. The positive constant $\chi$ is the magnetic susceptibility of the material. Without loss of generality, we assume that $\chi=\frac{1}{4}$, $\lambda_r=\lambda_e=\gamma=1$ in this paper. The notation $\mathbf{H}_{\textrm{eff}}$ denotes the effective field, which consists of the external magnetic field, the demagnetizing field and some quantum mechanical effects, etc.  For more details on the background of the LLG,  LLB and LLBar equations, we refer to \cite{feischl2017existence,guo2008landau,mayergoyz2009nonlinear,le2016weak,atxitia2016fundamentals,vogler2014landau,dvornik2013micromagnetic,dvornik2014thermodynamically,soenjaya2023global}.

Physical speaking, due to the inevitable stochastic perturbation stemming from the surrounding environment, it will be natural and important to studied the evolution of solutions to the associated stochastic counterparts. Indeed,  analyzing the noise-induced transitions was started by N\'{e}el \cite{neel1946bases}, and subsequent advancements were made in \cite{brown1963thermal,kamppeter1999stochastic}. In \cite{brzezniak2013weak,brzezniak2016weak}, Brze\'{z}niak and Goldys et al. successfully introduced proper Gaussian-type random noises into the LLG equation and studied quantitative property of solutions to the associated stochastic partial differential equation. Later, Brze\'{z}niak and Manna \cite{brzezniak2019weak} further discussed the stochastic LLG equation driven by L\'{e}vy noise. Recently, the large deviation principle for the LLG equation are established \cite{brzezniak2017large,gussetti2023pathwise}. Regarding references for the stochastic LLB equation, we refer to \cite{jiang2019martingale} for the existence and regularity of solutions, \cite{brzezniak2020existence} for the existence of invariant measures, and \cite{qiu2020asymptotic} for the establishment of large deviation principle.

Coming back to \eqref{sys0}, the effective field  $\mathbf{H}_{\textrm{eff}}$ will be inevitably influenced by uncertainties in practical applications, these factors can introduce randomness into the effective field, which in turn affects the evolution of the spin magnetic moment. One of most inspiration example stems from the theory of ferromagnetism, which says that describing the phase transitions between different equilibrium states induced by thermal fluctuations of the field $\mathbf{H}_{\textrm{eff}}$ is an important problem. Therefore, in order to
get a better understanding of the dynamics of the magnetic spin field in real world, it will be interesting to incorporate the random fluctuations of the effective field $\mathbf{H}_{\textrm{eff}}$ into the model \eqref{sys0} and then explore what will happen to the solutions due to the appearance of noise. By employing the similar ideas in \cite{brzezniak2013weak,brzezniak2019weak,jiang2019martingale}, one of the effective ways is to perturb the effective field $\mathbf{H}_{\textrm{eff}}$ by using an external stochastic forcing, that is, to make the substituting
$\mathbf{H}_{\textrm{eff}}\mapsto \mathbf{H}_{\textrm{eff}}+\xi(t)$ in $\eqref{sys0}$,
where $\xi(t)$ denotes a noise with respect to time variable. When there is no noise, i.e., $\xi(t)\equiv0$, the mathematical analysis for the LLBar equation started from the authors Soenjaya and Tran \cite{soenjaya2023global}, where they established the global well-posedness of weak (strong) solutions for the equation. Later, the authors of present paper \cite{xu2024wellposednessinvariantmeasuresstochastically} considered the stochastic LLBar equation driven by Gaussian noise, proving the existence and uniquenss of (local) global pathwise weak solutions and the existence of invariant measures. After finishing this work, we have learned a new result by Goldys et al. \cite{goldys2024stochasticlandaulifshitzbaryakhtarequationglobal}, where the authors demonstrated the existence and uniquenss of global pathwise strong solutions as well as the existence of invariant measures for this equation by establishing higher-order energy estimates.

In many physical systems, the state of the system can experience instantaneous and significant changes. The Marcus integral, utilizing the Marcus mapping, accounts for the second-order effects of jumps, thereby accurately capturing these jumps and considering their impact on the system. Given this physical significance, we would like to explore the effect of pure jump noise on the magnetization vector $\mathbf{u}(t,x)$. More precisely, let us consider the initial-boundary value problem (IBVP) for the SLLBar equation with pure jump noise in the Marcus canonical sense \cite{marcus1981modeling, kunita2004stochastic}:
\begin{equation}\label{sys1}
\left\{
\begin{aligned}
\mathrm{d}\mathbf{u}&= \left[\mathbf{H}_{\textrm{eff}}-\Delta \mathbf{H}_{\textrm{eff}}-\mathbf{u}\times \mathbf{H}_{\textrm{eff}}\right]\,\mathrm{d}t+\int_{\mathbb{B}}(\mathbf{u}\times \mathbf{h}+\mathbf{g})\diamond\,\mathrm{d}L(t),&&\textrm{in}~\mathbb{R}_+\times\mathcal{O},\\
\mathbf{H}_{\textrm{eff}}&=\Delta \mathbf{u}+2(1-|\mathbf{u}|^2)\mathbf{u},&&\textrm{in}~\mathbb{R}_+\times\mathcal{O},\\
\frac{\partial\mathbf{u}}{\partial\mathbf{n}}&=\frac{\partial\Delta \mathbf{u}}{\partial\mathbf{n}}=0,&&\textrm{on}~\mathbb{R}_+\times\partial\mathcal{O},\\
\mathbf{u}(0)&=\mathbf{u}_0,&&\textrm{in}~\mathcal{O}.
\end{aligned}
\right.
\end{equation}
In \eqref{sys1}, the magnetic body $\mathcal{O}\subset\mathbb{R}^d$, $d=1,2,3$, is a bounded domain with smooth boundary $\partial\mathcal{O}$, and $\mathbf{n}$ denotes the exterior unit normal vector on the boundary. Let $\mathbb{B}$ be a unit ball in $\mathbb{R}$ (excluding the center), the L\'{e}vy noise  $\{L(t)\}_{t\geq0}$ takes the form of
\begin{equation*}
\begin{split}
L(t)=\int_0^t\int_{\mathbb{B}}l\tilde{\eta}(\mathrm{d}s,\mathrm{d}l),\quad \textrm{for all}~t\geq0,
\end{split}
\end{equation*}
which is defined on a fixed probability space $\pi:=\left(\Omega,\mathcal{F},\mathbb{P}\right)$ with filtration $\mathbb{F}:=\{\mathcal{F}_t\}_{t\geq0}$.  Here $\tilde{\eta}=\eta-\textrm{Leb}\otimes\nu$ is the compensated time homogeneous Poisson random measure of $\eta$. The time-homogeneous Poisson random measure is denoted by $\eta$, and the associate intensity measure is given by $\textrm{Leb}\otimes\nu$ such that $\textrm{supp}(\nu)\subset B(0,1)/\{0\}\subset \mathbb{R}$. Moreover, $\mathbf{h}$ and $\mathbf{g}$ are space-dependent functions with suitable regularization conditions. We refer to section \ref{sec2} for more details about the definition of the Marcus integration with respect to $\diamond\,\mathrm{d}L(t)$. 

To the best of our knowledge, few results are available in regard to the mathematical analysis for the LLBar equation perturbed by pure jump noise, which arises naturally when one takes consideration of the ambient noises that influence the system intermittently. The  \emph{main purpose} of this paper is devoted to address the IBVP \eqref{sys1} in two folds:  First, we are going to investigate the existence and uniqueness of global pathwise solutions to \eqref{sys1}  under proper assumptions; Second, based on the previous well-posedness result, we shall further establish a Freidlin-Wentzell-type large deviation principle (LDP) for the pathwise solution.

The first main result can now be stated by the following theorem.

\begin{theorem}\label{the1} Assume that the initial data $\mathbf{u}_0\in\mathbb{H}^1$, the functions $
\mathbf{h}\in\mathbb{W}^{1,\infty}$ and $\mathbf{g}\in\mathbb{H}^1$. Then, the IBVP \eqref{sys1} has a unique global pathwise weak solution $\mathbf{u} $, such that
\begin{enumerate}
\item [$\bullet$] $\mathbf{u} $ is a c\`{a}dl\`{a}g $ \mathbb{F}$-predictable process in
$
L^p\left(\Omega;\mathbb{D}(0,T;\mathbb{H}^1)\cap L^2(0,T;\mathbb{H}^3)\right)$, $ p\geq1$;

\item [$\bullet$] there holds $\mathbb{P}$-a.s.
\begin{equation}\label{1def2}
\begin{split}
(\mathbf{u}(t),\phi)_{\mathbb{L}^2}&=(\mathbf{u}_0,\phi)_{\mathbb{L}^2}+\int_0^{t}\left(\nabla\mathbf{u},\nabla\phi\right)_{\mathbb{L}^2}\,\mathrm{d}s+\int_0^{t}\left(\nabla\Delta\mathbf{u},\nabla\phi\right)_{\mathbb{L}^2}\,\mathrm{d}s\\
&+2\int_0^{t}\left((1-|\mathbf{u}|^2)\mathbf{u},\phi\right)_{\mathbb{L}^2}\,\mathrm{d}s+\int_0^{t}\left(\mathbf{u}\times\nabla\mathbf{u},\nabla\phi\right)_{\mathbb{L}^2}\,\mathrm{d}s\\
&-2\int_0^{t}\left(\nabla(|\mathbf{u}|^2\mathbf{u}),\nabla\phi\right)_{\mathbb{L}^2}\,\mathrm{d}s+\int_0^t\int_{B}\left((\mathbf{u}\times \mathbf{h}+\mathbf{g})\diamond\,\mathrm{d}L(s),\phi\right)_{\mathbb{L}^2},
\end{split}
\end{equation}
for any $t\in[0,T]$ and $\phi\in \mathbb{H}^1$.
\end{enumerate}
\end{theorem}

\begin{remark}\label{222}
By using the similar argument with some key modifications, one can also establish the existence and uniqueness results for IBVP \eqref{sys1} with initial data belongs to $\mathbb{L}^2$ and $\mathbb{H}^2$, which corresponds to the global very weak solution $
\mathbf{u}\in L^p\left(\Omega;\mathbb{D}(0,T;\mathbb{L}^2)\cap L^2(0,T;\mathbb{H}^2)\right)$ and global strong solution $
\mathbf{u}\in L^p\left(\Omega;\mathbb{D}(0,T;\mathbb{H}^2)\cap L^2(0,T;\mathbb{H}^4)\right)$, respectively. To save the space, we would like to address this problem in a forthcoming  work.
\end{remark}

%
To establish the LDP, let us introduce the following basic definitions.

\begin{definition}\label{def6-1} Let $\mathbf{Z}$ be a Polish space. A function $\mathbf{I}:\mathbf{Z}\rightarrow[0,\infty]$ is called a \textsf{rate function} if $\mathbf{I}$ is lower semicontinuous. A rate function $\mathbf{I}$ is a good rate function if for arbitrary $M\in[0,\infty)$, the level set $K_M:=\{x:\mathbf{I}(x)\leq M\}$ is compact in $\mathbf{Z}$.
\end{definition}

\begin{definition}\label{def6-2} We say that a family of probability measures $\{\mathbb{P}_{\varepsilon}:\varepsilon>0\}$ satisfies the LDP on $\mathbf{Z}$ with a good rate function $\mathbf{I}:\mathbf{Z}\rightarrow[0,\infty]$, provided
\begin{enumerate}
\item [(1)] for each closed set $O_1\subset\mathbf{Z}$,
$$
\limsup_{\varepsilon\rightarrow0}\varepsilon\log\mathbb{P}_{\varepsilon}(O_1)\leq-\inf_{x\in O_1}\mathbf{I}(x);
$$
\item[(2)] for each open set $O_2\subset\mathbf{Z}$,
$$
\liminf_{\varepsilon\rightarrow0}\varepsilon\log\mathbb{P}_{\varepsilon}(O_2)\geq-\inf_{x\in O_2}\mathbf{I}(x).
$$
\end{enumerate}
\end{definition}

The LDP theory is closely related to the following random perturbations of the SLLBar equation for any $\epsilon>0$:
\begin{equation}\label{sys61}
\left\{
\begin{aligned}
\mathrm{d}\mathbf{u}^{\varepsilon}(t)&=\left[-\Delta\mathbf{u}^{\varepsilon}-\Delta^2\mathbf{u}^{\varepsilon}+2(1-|\mathbf{u}^{\varepsilon}|^2)\mathbf{u}^{\varepsilon}+2\Delta(|\mathbf{u}^{\varepsilon}|^2\mathbf{u}^{\varepsilon})-\mathbf{u}^{\varepsilon}\times\Delta \mathbf{u}^{\varepsilon}\right]\,\mathrm{d}t\\
&+\varepsilon\int_{B}(\mathbf{u}^{\varepsilon}\times \mathbf{h}+\mathbf{g})\diamond\,\mathrm{d}L^{\varepsilon^{-1}}(t),\\
\mathbf{u}^{\varepsilon}(0)&=\mathbf{u}_0\in \mathbb{H}^1,
\end{aligned}
\right.
\end{equation}
where the parameterized L\'{e}vy process $L^{\varepsilon^{-1}}(t)$ is defined in section \ref{sec6}.

Let us define the rate function $\mathbf{I}:\mathcal{Z}_T:=\mathbb{D}([0,T];\mathbb{H}^1)\cap L^2(0,T;\mathbb{H}^{3})\rightarrow[0,\infty]$ by
\begin{equation}\label{663-1}
\begin{split}
\mathbf{I}(f):=\inf\{\mathcal{L}_T(\theta):\theta\in \mathbb{S},\mathbf{u}^{\theta}=f\},~f\in \mathcal{Z}_T,
\end{split}
\end{equation}
where $\mathbf{u}^{\theta}$ satisfies \eqref{sys61-1} and the definitions of $\mathcal{L}_T(\cdot)$ and $\mathbb{S}$ are defined in \eqref{subsec6-1-5} and \eqref{ccc1}, respectively. Note that $\mathbf{I}(f)=\infty$ if the set $\{\theta\in\mathbb{S}:\mathbf{u}^{\theta}=f\}$ is empty.

Our second main result can now be stated as follows.

\begin{theorem}\label{the2} Under the same assumptions as in Theorem \ref{the1}, the family of laws $\{\mathscr{L}(\mathbf{u}^{\varepsilon}):\varepsilon\in(0,1]\}$ for the equation \eqref{sys61} on $\mathcal{Z}_T$ satisfies the large deviation principle with rate function $\mathbf{I}$ defined in \eqref{663-1}.
\end{theorem}
\begin{remark}
The proof of Theorem \ref{the2} is inspired by the weak convergence approach introduced by Budhiraja
et al. \cite{budhiraja2011variational,budhiraja2013large}, and the Girsanov-type theorem for Poisson random measures provided by Brze\'{z}niak, Peng and Zhai \cite{brzezniak2022well}. Moreover, it is further expected that the similar results may be established for the very weak solutions and strong solutions as we mentioned  in Remark \ref{222}, which will also be treated elsewhere.
\end{remark}

\subsection{Organization} This paper is organized as follows.  In section \ref{sec2}, we provide some basic results and introduce the Marcus mapping to rewrite equation \eqref{sys1}. In section \ref{sec3}, we establish uniform bounded estimates for solutions of the finite-dimensional system \eqref{sys2}. In section \ref{sec4}, we provide some compactness and tightness criteria. The main results on well-posedness and large deviation principle are presented in sections \ref{sec5} and \ref{sec6}, respectively. Some auxiliary results are given in Appendix.

\section{preparation}\label{sec2}
\subsection{Functional setting}

Let $\mathbb{L}^p:=\mathbb{L}^p(\mathcal{O};\mathbb{R}^3)$ be the space of $p$-th Lebesgue integrable functions, and $\mathbb{W}^{k,p}:=\mathbb{W}^{k,p}(\mathcal{O};\mathbb{R}^3)$ be the usual Sobolev space. For $k=2$, we set $\mathbb{H}^p:=\mathbb{W}^{2,p}$. Given a Banach space $ X$, the symbol $\langle\cdot,\cdot\rangle_{X^*,X}$ stands for the duality pairing between $X$ and its dual space $X^*:=\mathcal{L}(X;\mathbb{R})$. We denote by $L^p_w(0,T;Q)$ the space $L^p(0,T;Q)$ endowed with the weak topology.

Let $(\mathbb{S},\varrho)$ be a complete and separable metric space, we denote by  $\mathbb{D}([0,T];\mathbb{S})$ the space of all $\mathbb{S}$-valued c\`{a}dl\`{a}g functions defined on $[0,T]$, which are right continuous with left limits at every $t\in[0,T]$. In the following, the space $\mathbb{D}([0,T];\mathbb{S})$ will be endowed with the Skorokhod topology such that $(\mathbb{D}([0,T];\mathbb{S}),\delta_T)$ is a complete metric space \cite{joffe1986weak}, where the metric $\delta_T$ is given by
\begin{eqnarray*}
\delta_T(f,g):=\inf_{\lambda\in\Lambda_T}\bigg(\sup_{t\in[0,T]}\varrho(f(t),g\circ\lambda(t))+\sup_{t\in[0,T]}|t-\lambda(t)|+\sup_{s\neq t}\bigg|\log\frac{\lambda(t)-\lambda(s)}{t-s}\bigg|\bigg),
\end{eqnarray*}
and $\Lambda_T$ is the set of increasing homeomorphisms of $[0,T]$. Let $Q_w$ be a Banach space $Q$ endowed with the weak topology, we define $ \mathbb{D}([0,T];Q_w)$ to be the space of weakly c\`{a}dl\`{a}g functions $f:[0,T]\rightarrow Q$ with the weakest topology such that for all $f\in Q$ the mapping $\mathbb{D}([0,T];Q_w)\ni f\mapsto(f,h)_Q\in\mathbb{D}([0,T];\mathbb{R})$ are continuous.


\subsection{The Marcus mapping}\label{subsec2}
Let  $\mathbf{h}\in\mathbb{W}^{1,\infty}$ and $\mathbf{g}\in\mathbb{H}^1$ be the two functions appearing in the assumption of Theorem \ref{the1}, we define
$$
J:\mathbb{H}^1\ni f\mapsto f\times \mathbf{h}+\mathbf{g}\in \mathbb{H}^1.
$$
It is claer that $J$ is a bounded mapping from to $\mathbb{H}^1$ to $\mathbb{H}^1$. We now define a generalized Marcus mapping
$
\Phi:\mathbb{R}_{+}\times \mathbb{R}\times \mathbb{H}^1\rightarrow \mathbb{H}^1
$
such that for each fixed $l\in \mathbb{R}$, $\mathbf{u}_0\in\mathbb{H}^1$, the function
$
\mathbb{R}_{+}\ni t\mapsto\Phi(t,l,\mathbf{u}_0)\in\mathbb{H}^1
$
is a $C^1$ solution of the ordinary differential equation
\begin{equation*}\label{sub2-1-1}
\begin{split}
\frac{\mathrm{d}\mathbf{u}}{\mathrm{d}t}=lJ(\mathbf{u}(t)),~t\geq0,\quad \mathbf{u}(0)=\mathbf{u}_0\in\mathbb{H}^1,
\end{split}
\end{equation*}
which is equivalent to
\begin{equation*}\label{sub2-1-2}
\begin{split}
\Phi(t,l,\mathbf{u}_0)=\Phi(0,l,\mathbf{u}_0)+\int_0^tlJ(\Phi(s,l,\mathbf{u}_0))\,\mathrm{d}s,~t\geq0.
\end{split}
\end{equation*}
Since $J$ is bounded from $\mathbb{H}^1$ to $\mathbb{H}^1$, the mapping $\Phi:\mathbb{R}_{+}\times \mathbb{R}\times \mathbb{H}^1\rightarrow \mathbb{H}^1$ is well defined. In what follows, we fix $t=1$ and denote $\Phi(l,\cdot):=\Phi(1,l,\cdot)$, $l\in\mathbb{R}$. Based on the Marcus mapping $\Phi$, by calculating the increments, equation \ref{sys1} with notation $\diamond$ can be rewritten as the following standard stochastic integral form.
\begin{equation}\label{pp1}
\begin{split}
\mathbf{u}(t)&=\mathbf{u}_0+\int_0^t\left(-\Delta\mathbf{u}-\Delta^2\mathbf{u}+2(1-|\mathbf{u}|^2)\mathbf{u}+2\Delta(|\mathbf{u}|^2\mathbf{u})-\mathbf{u}\times\Delta \mathbf{u}\right) \mathrm{d}s\\
&+\int_0^t\int_{\mathbb{B}}\Phi(l,\mathbf{u})-\mathbf{u}-lJ(\mathbf{u})\nu(\mathrm{d}l)\,\mathrm{d}s+\int_0^t\int_{\mathbb{B}}[\Phi(l,\mathbf{u})-\mathbf{u}]\tilde{\eta}(\mathrm{d}s,\mathrm{d}l),
\end{split}
\end{equation}
For convenience, we define the following mapping  on $\mathbb{H}^1$.
\begin{equation*}\label{sub2-1-3}
\begin{split}
&\mathbf{G}(l,f):=\Phi(l,f)-f,~\mathbf{H}(l,f):=\Phi(l,f)-f-lJ(f),~\mathbf{b}(f):=\int_{\mathbb{B}}\mathbf{H}(l,f)\nu(\mathrm{d}l)~,f\in\mathbb{H}^1.
\end{split}
\end{equation*}
Equation\eqref{pp1} then can be rewritten as follows:
\begin{equation}\label{sub2-1-6}
\begin{split}
\mathbf{u}(t)&=\mathbf{u}_0+\int_0^t\left(-\Delta\mathbf{u}-\Delta^2\mathbf{u}+2(1-|\mathbf{u}|^2)\mathbf{u}+2\Delta(|\mathbf{u}|^2\mathbf{u})-\mathbf{u}\times\Delta \mathbf{u}\right) \mathrm{d}s\\
&+\int_0^t\mathbf{b}(\mathbf{u})\,\mathrm{d}s+\int_0^t\int_{\mathbb{B}}\mathbf{G}(l,\mathbf{u})\tilde{\eta}(\mathrm{d}s,\mathrm{d}l).
\end{split}
\end{equation}

In order to get a better understanding of the terms involved in \eqref{sub2-1-6}. Let us show some basic properties satisfied by the Marcus mapping.

\begin{lemma}\label{lem2-2-1} Let $\mathbf{h}\in\mathbb{W}^{1,\infty}$ and $\mathbf{g}\in\mathbb{H}^1$, then for every $l\in \mathbb{B}$, there exists a positive constant $C$ such that
\begin{equation*}\label{lem2-2-1-1}
\begin{split}
\|\Phi(l,\mathbf{u})\|_{\mathbb{H}^1}&\leq C(1+\|\mathbf{u}\|_{\mathbb{H}^1}),~\mathbf{u}\in \mathbb{H}^1,\\
\|\Phi(l,\mathbf{u})\|_{\mathbb{L}^4}&\leq C(1+\|\mathbf{u}\|_{\mathbb{L}^4}),~\mathbf{u}\in \mathbb{L}^4.
\end{split}
\end{equation*}
\end{lemma}
\begin{proof}[\emph{\textbf{Proof}}] According to the definition of $\Phi$, we have
\begin{equation*}
\begin{split}
\Phi(l,\mathbf{u})=\mathbf{u}+\int_0^1lJ(\Phi(s,l,\mathbf{u}))\,\mathrm{d}s.
\end{split}
\end{equation*}
Taking the $\mathbb{H}^1$-norm of both sides gives
\begin{equation*}
\begin{split}
\|\Phi(l,\mathbf{u})\|_{\mathbb{H}^1}
&\leq\|\mathbf{u}\|_{\mathbb{H}^1}+\int_0^1\|\Phi(s,l,\mathbf{u})\times \mathbf{h}+\mathbf{g}\|_{\mathbb{H}^1}\,\mathrm{d}s\\
&\leq\|\mathbf{u}\|_{\mathbb{H}^1}+C\int_0^1(1+\|\Phi(s,l,\mathbf{u})\|_{\mathbb{H}^1})\,\mathrm{d}s,
\end{split}
\end{equation*}
which combined with the Gronwall lemma implies the result \eqref{lem2-2-1-1}. Additionally, noting the conditions $\mathbf{h}\in \mathbb{W}^{1,\infty}$ and $\mathbf{g}\in \mathbb{H}^{1}$ and using the Sobolev embedding $\mathbb{H}^1\hookrightarrow \mathbb{L}^4$, we easily obtain that
$
\|\Phi(l,\mathbf{u})\|_{\mathbb{L}^4}\leq\|\mathbf{u}\|_{\mathbb{L}^4}+C\int_0^1(1+\|\Phi(s,l,\mathbf{u})\|_{\mathbb{L}^4})\,\mathrm{d}s.
$
Thus by using the Gronwall lemma, we obtain the result. The proof is complete.\end{proof}

\begin{corollary}\label{--cor2-2-2}  For every $\mathbf{u},~\mathbf{u}_1\in \mathbb{H}^1$ and $l\in \mathbb{B}$, there exists a constant $C>0$ such that
\begin{equation*}
\begin{split}
&\|\mathbf{G}(l,\mathbf{u})\|_{\mathbb{H}^1}\leq C|l|(1+\|\mathbf{u}\|_{\mathbb{H}^1}),~\|\mathbf{G}(l,\mathbf{u})-\mathbf{G}(l,\mathbf{u}_1)\|_{\mathbb{H}^1}\leq C|l|\|\mathbf{u}-\mathbf{u}_1\|_{\mathbb{H}^1};\\
&\|\mathbf{b}(\mathbf{u})\|_{\mathbb{H}^1}\leq C(1+\|\mathbf{u}\|_{\mathbb{H}^1}),~\|\mathbf{b}(\mathbf{u})-\mathbf{b}(\mathbf{u}_1)\|_{\mathbb{H}^1}\leq C\|\mathbf{u}-\mathbf{u}_1\|_{\mathbb{H}^1}.
\end{split}
\end{equation*}
\end{corollary}

\begin{proof}[\emph{\textbf{Proof}}] According to the definition of $\mathbf{G}$, we have
\begin{equation*}
\begin{split}
\|\mathbf{G}(l,\mathbf{u})-\mathbf{G}(l,\mathbf{u}_1)\|_{\mathbb{H}^1}&=\|\Phi(l,\mathbf{u})-\Phi(l,\mathbf{u}_1)-\mathbf{u}+\mathbf{u}_1\|_{\mathbb{H}^1}\\
&\leq\int_0^1|l|\|J(\Phi(s,l,\mathbf{u}))-J(\Phi(s,l,\mathbf{u}_1))\|_{\mathbb{H}^1}\,\mathrm{d}s\\
&\leq\int_0^1|l|\|\mathbf{h}\|_{\mathbb{L}^{\infty}}\|\Phi(s,l,\mathbf{u})-\Phi(s,l,\mathbf{u}_1)\|_{\mathbb{H}^1}\,\mathrm{d}s.
\end{split}
\end{equation*}
By the triangle inequality, it follows that
\begin{equation*}
\begin{split}
\|\Phi(l,\mathbf{u})-\Phi(l,\mathbf{u}_1)\|_{\mathbb{H}^1}&\leq\|\mathbf{G}(l,\mathbf{u})-\mathbf{G}(l,\mathbf{u}_1)\|_{\mathbb{H}^1}+\|\mathbf{u}-\mathbf{u}_1\|_{\mathbb{H}^1}\\
&\leq\|\mathbf{u}-\mathbf{u}_1\|_{\mathbb{H}^1}+C|l|\int_0^1\|\Phi(s,l,\mathbf{u})-\Phi(s,l,\mathbf{u}_1)\|_{\mathbb{H}^1}\,\mathrm{d}s.
\end{split}
\end{equation*}
Thus we use the Gronwall lemma to obtain that
\begin{equation*}
\begin{split}
\|\Phi(l,\mathbf{u})-\Phi(l,\mathbf{u}_1)\|_{\mathbb{H}^1}\leq e^{C|l|}\|\mathbf{u}-\mathbf{u}_1\|_{\mathbb{H}^1},
\end{split}
\end{equation*}
which implies that the mapping $\mathbf{G}$ is Lipschitz continuous. Similarly, it is easy to see that
\begin{equation*}
\begin{split}
\|\mathbf{H}(l,\mathbf{u})-\mathbf{H}(l,\mathbf{u}_1)\|_{\mathbb{H}^1}&=\|\Phi(l,\mathbf{u})-\Phi(l,\mathbf{u}_1)-\mathbf{u}+\mathbf{u}_1-lJ(\mathbf{u})+lJ(\mathbf{u}_1)\|_{\mathbb{H}^1}\\
&\leq\|\Phi(l,\mathbf{u})-\Phi(l,\mathbf{u}_1)\|_{\mathbb{H}^1}+\|\mathbf{u}-\mathbf{u}_1\|_{\mathbb{L}^2}+\|J(\mathbf{u})-J(\mathbf{u}_1)\|_{\mathbb{H}^1}\\
&\leq C\|\mathbf{u}-\mathbf{u}_1\|_{\mathbb{H}^1},
\end{split}
\end{equation*}
which means that the mapping $\mathbf{b}$ is Lipschitz continuous. Since $\mathbf{G}$ and $\mathbf{b}$ are Lipschitz, they obviously have linear growth.
\end{proof}

\section{Faedo-Galerkin apprpximation}\label{sec3}
The main aim of this section is first to introduce the approximation equation with solutions in the finite-dimensional space, and then derive some necessary uniform  a priori estimates for the approximation solutions. Let $\{\mathbf{e}_i\}_{i=1}^{\infty}$ denote an orthonormal basis of $\mathbb{L}^2$ consisting of eigenvectors for the Neumann Laplacian $A=-\Delta$ such that
\begin{equation*}
\begin{split}
A\mathbf{e}_i=\lambda_i\mathbf{e}_i~~\textrm{in}~~\mathcal{O},~~\textrm{and}~~\frac{\partial\mathbf{e}_i}{\partial\mathbf{n}}=0~~\textrm{on}~~\partial\mathcal{O},
\end{split}
\end{equation*}
where $\lambda_i>0$ are the eigenvalues of $A$, associated with $\mathbf{e}_i$. According to elliptic regularity results, $\mathbf{e}_i$ is smooth up to the boundary, and we have
\begin{equation*}
\begin{split}
A^2\mathbf{e}_i=\lambda_i^2\mathbf{e}_i~~\textrm{in}~~\mathcal{O},~~\textrm{and}~~\frac{\partial\mathbf{e}_i}{\partial\mathbf{n}}=\frac{\partial\Delta\mathbf{e}_i}{\partial\mathbf{n}}=0~~\textrm{on}~~\partial\mathcal{O}.
\end{split}
\end{equation*}
Let $S_n:=\textrm{span}\{\mathbf{e}_1...,\mathbf{e}_n\}$ and $\Pi_n:\mathbb{L}^2\rightarrow S_n$ be the orthogonal projection defined by
$$
(\Pi f,g)_{\mathbb{L}^2}=(f,g)_{\mathbb{L}^2},~g\in S_n,~f\in\mathbb{L}^2.
$$
Let us define a mapping $\Phi_n$ as the solution to the following ordinary differential equation
\begin{equation}\label{sec3-1}
\begin{split}
\frac{\mathrm{d}\mathbf{u}_n}{\mathrm{d}t}=lJ_n(\mathbf{u}_n(t)),~t\geq0,~\mathbf{u}_n(0)=\Pi_n\mathbf{u}_0\in S_n,
\end{split}
\end{equation}
where $J_n:=\Pi_nJ$.  Correspondingly , we also introduce the following notations:
$ \mathbf{G}_n(l,\mathbf{u}_n):=\Phi_n(l,\mathbf{u}_n)-\mathbf{u}_n$,
$\mathbf{H}_n(l,\mathbf{u}_n):=\Phi_n(l,\mathbf{u}_n)-\mathbf{u}_n-lJ_n(\mathbf{u}_n)$,
$\mathbf{b}_n(\mathbf{u}_n):=\int_{\mathbb{B}}\mathbf{H}_n(l,\mathbf{u}_n)\nu(\mathrm{d}l)$.

\begin{remark}\label{rem331}
Notice that if one replace the mapping $\Phi$ by $\Phi_n$ and the function $\textbf{u}$ by $\textbf{u}_n$ in finite-dimensional space $S_n$, then the mappings $\Phi_n$, $\mathbf{G}_n$, $\mathbf{H}_n$ and $\mathbf{b}_n$ inherit the corresponding properties from Lemma \ref{lem2-2-1} and Corollary \ref{--cor2-2-2}. As the proof of these properties are same as before, so we shall omit the details here.
\end{remark}

Now we consider the following Galerkin approximation scheme for \eqref{sub2-1-6}
\begin{equation}\label{sys2}
\left\{
\begin{aligned}
\mathrm{d}\mathbf{u}_n=&\Pi_n\Big[-\Delta\mathbf{u}_n-\Delta^2\mathbf{u}_n+2(1-|\mathbf{u}_n|^2)\mathbf{u}_n+2\Delta(|\mathbf{u}_n|^2\mathbf{u}_n)\\
&-\mathbf{u}_n\times\Delta \mathbf{u}_n\Big] \mathrm{d}t
+\mathbf{b}_n(\mathbf{u}_n)\,\mathrm{d}t+\int_{\mathbb{B}}\mathbf{G}_n(l,\mathbf{u}_n)\tilde{\eta}(\mathrm{d}t,\mathrm{d}l),&&\textrm{in}~(0,\infty)\times\mathcal{O},\\
\mathbf{u}_n(0)=&\Pi_n\mathbf{u}_0,&&\textrm{in}~\mathcal{O}.
\end{aligned}
\right.
\end{equation}
For the convenience of the subsequent discussion, we shall use the notations
\begin{equation}\label{nota1}
\begin{split}
&F_n^1(\mathbf{u}_n):=-\Delta \mathbf{u}_n+2\mathbf{u}_n,~F_n^2(\mathbf{u}_n):=-\Delta^2 \mathbf{u}_n,~F_n^3(\mathbf{u}_n):=-2\Pi_n\left(|\mathbf{u}_n|^2\mathbf{u}_n\right),\\
&F_n^4(\mathbf{u}_n):=-\Pi_n\left(\mathbf{u}_n\times\Delta \mathbf{u}_n\right),~F_n^5(\mathbf{u}_n):=2\Pi_n\Delta\left(|\mathbf{u}_n|^2\mathbf{u}_n\right),\\
&\mathbf{H}_{\textrm{eff}}^n:=\Delta \mathbf{u}_n+2\mathbf{u}_n-2\Pi_n(|\mathbf{u}_n|^2\mathbf{u}_n),~F_n(\mathbf{u}_n,\mathbf{H}_{\textrm{eff}}^n):=\mathbf{H}_{\textrm{eff}}^n-\Delta \mathbf{H}_{\textrm{eff}}^n-\Pi_n\left(\mathbf{u}_n\times \mathbf{H}_{\textrm{eff}}^n\right).
\end{split}
\end{equation}
Obviously, $F_n(\mathbf{u}_n,\mathbf{H}_{\textrm{eff}}^n)=\sum_{j=1}^5F_n^j(\mathbf{u}_n)$. Moreover, the existence of a unique local strong solution to the SDE \eqref{sys2} is a consequence of $F_n^1$-$F_n^5$ are locally Lipschitz \cite{applebaum2009levy,xu2024wellposednessinvariantmeasuresstochastically}.

We now proceed to prove uniform bounds for the approximate solutions. To begin with, we consider the case where the initial value belongs to $\mathbb{L}^2$.

\begin{lemma}\label{lem32} Let $\mathcal{O}\subset\mathbb{R}^d$,~$d=1,2,3$, be a bounded domain with $C^{1,1}$-boundary. Then for any $p\geq1$, $n\in\mathbb{N}$ and every $t\in[0,T]$, there exists a positive constant $C=C(\|\mathbf{u}_0\|_{\mathbb{L}^2},p,\mathbf{h},\mathbf{g},T)$ independent of $n$ such that
\begin{equation}\label{1lem32}
\begin{split}
&\mathbb{E}\sup_{s\in[0,t]}\|\mathbf{u}_n(s)\|_{\mathbb{L}^2}^{2p}+\mathbb{E}\left(\int_0^t\|\mathbf{u}_n(s)\|_{\mathbb{H}^2}^2\,\mathrm{d}s\right)^p+\mathbb{E}\left(\int_0^t\|\mathbf{u}_n(s)\|_{\mathbb{L}^4}^4\,\mathrm{d}s\right)^p\leq C.
\end{split}
\end{equation}
\end{lemma}
\begin{proof}[\emph{\textbf{Proof}}] Applying the It\^{o} formula to $\{\mathbf{F}:\mathbf{u}_n\mapsto\frac{1}{2}\|\mathbf{u}_n\|_{\mathbb{L}^2}^2\}$, we have
\begin{equation}\label{2lem32}
\begin{split}
&\frac{1}{2}\|\mathbf{u}_n(t)\|_{\mathbb{L}^2}^2-\int_0^t\|\nabla \mathbf{u}_n\|_{\mathbb{L}^2}^2\,\mathrm{d}s+\int_0^t\|\Delta \mathbf{u}_n\|_{\mathbb{L}^2}^2\,\mathrm{d}s+2\int_0^t\| \mathbf{u}_n\|_{\mathbb{L}^4}^4\,\mathrm{d}s\\
&+4\int_0^t\|\mathbf{u}_n\cdot\nabla\mathbf{u}_n\|_{\mathbb{L}^2}^2\,\mathrm{d}s+2\int_0^t\||\mathbf{u}_n||\nabla\mathbf{u}_n|\|_{\mathbb{L}^2}^2\,\mathrm{d}s\\
&=\frac{1}{2}\|\mathbf{u}_n(0)\|_{\mathbb{L}^2}^2+2\int_0^t\| \mathbf{u}_n\|_{\mathbb{L}^2}^2\,\mathrm{d}s+\int_0^t\int_{\mathbb{B}}\mathbf{F}(\Phi_n(l,\mathbf{u}_n))-\mathbf{F}(\mathbf{u}_n)\tilde{\eta}(\mathrm{d}s,\mathrm{d}l)\\
&+\int_0^t\int_{\mathbb{B}}\mathbf{F}(\Phi_n(l,\mathbf{u}_n))-\mathbf{F}(\mathbf{u}_n)-l\mathbf{F}'(\mathbf{u}_n)(J_n(\mathbf{u}_n))\nu(\mathrm{d}l)\mathrm{d}s.
\end{split}
\end{equation}
Using integration by parts, the H\"{o}lder inequality and Young's inequality, we have
\begin{equation}\label{3lem32}
\begin{split}
\|\nabla \mathbf{u}_n\|_{\mathbb{L}^2}^2=-(\mathbf{u}_n,\Delta\mathbf{u}_n)_{\mathbb{L}^2}\leq\varepsilon\|\Delta\mathbf{u}_n\|_{\mathbb{L}^2}^2+C_{\varepsilon}\|\mathbf{u}_n\|_{\mathbb{L}^2}^2.
\end{split}
\end{equation}
Applying the Young inequality and Jensen inequality, we see that for all $p\geq1$
\begin{equation}\label{4lem32}
\begin{split}
&\left|\int_0^t\int_{\mathbb{B}}\mathbf{F}(\Phi_n(l,\mathbf{u}_n))-\mathbf{F}(\mathbf{u}_n)-l\mathbf{F}'(\mathbf{u}_n)(J_n(\mathbf{u}_n))\nu(\mathrm{d}l)\mathrm{d}s\right|^p\\
&\leq\left|\int_0^t\int_{\mathbb{B}}C(1+\| \mathbf{u}_n\|_{\mathbb{L}^2}^2)+|l|\|\mathbf{u}_n\|_{\mathbb{L}^2}(1+\| \mathbf{u}_n\|_{\mathbb{L}^2})\nu(\mathrm{d}l)\mathrm{d}s\right|^p\\
&\leq C+C\int_0^t\| \mathbf{u}_n\|_{\mathbb{L}^2}^{2p}\,\mathrm{d}s.
\end{split}
\end{equation}
Now using the Burkh\"{o}lder-Davis-Gundy (BDG) inequality, H\"{o}lder's inequality and Young's inequality, we see that for any $p\geq1$
\begin{equation}\label{5lem32}
\begin{split}
&\mathbb{E}\left[\sup_{s\in[0,t]}\left|\int_0^{s}\int_{\mathbb{B}}\mathbf{F}(\Phi_n(l,\mathbf{u}_n))-\mathbf{F}(\mathbf{u}_n)\tilde{\eta}(\mathrm{d}s',\mathrm{d}l)\right|^p\right]\\
&\leq C\mathbb{E}\left|\int_0^t\int_{\mathbb{B}}\left[\mathbf{F}(\Phi_n(l,\mathbf{u}_n))-\mathbf{F}(\mathbf{u}_n)\right]^2\nu(\mathrm{d}l)\mathrm{d}s\right|^\frac{p}{2}\\
&\leq C+C\mathbb{E}\left(\sup_{s\in[0,t]}\|\mathbf{u}_n(s)\|_{\mathbb{L}^2}^{2}\int_0^t\| \mathbf{u}_n\|_{\mathbb{L}^2}^{2}\,\mathrm{d}s\right)^{\frac{p}{2}}\\
&\leq C+\varepsilon\mathbb{E}\sup_{s\in[0,t]}\|\mathbf{u}_n(s)\|_{\mathbb{L}^2}^{2p} +C_{\varepsilon}\mathbb{E}\int_0^t\|\mathbf{u}_n(s)\|_{\mathbb{L}^2}^{2p}\,\mathrm{d}s.
\end{split}
\end{equation}
Thus plugging \eqref{3lem32}-\eqref{5lem32} into \eqref{2lem32} and  choosing $\varepsilon$ small enough, we infer that
\begin{equation*}
\begin{split}
&\mathbb{E}\sup_{s\in[0,t]}\|\mathbf{u}_n(s)\|_{\mathbb{L}^2}^{2p}+\mathbb{E}\left(\int_0^t\|\nabla \mathbf{u}_n(s)\|_{\mathbb{L}^2}^2\,\mathrm{d}s\right)^p+\mathbb{E}\left(\int_0^t\|\Delta \mathbf{u}_n(s)\|_{\mathbb{L}^2}^2\,\mathrm{d}s\right)^p\\
&+\mathbb{E}\left(\int_0^t\|\mathbf{u}_n(s)\|_{\mathbb{L}^4}^4\,\mathrm{d}s\right)^p+\mathbb{E}\left(\int_0^t\||\mathbf{u}_n(s)||\nabla\mathbf{u}_n(s)|\|_{\mathbb{L}^2}^2\,\mathrm{d}s\right)^p\\
&\leq C+C\int_0^t\mathbb{E}\sup_{s'\in[0,s]}\|\mathbf{u}_n(s')\|_{\mathbb{L}^2}^{2p}\,\mathrm{d}s.
\end{split}
\end{equation*}
The estimate \eqref{1lem32} then follows from the Gronwall lemma. The proof is  completed.\end{proof}

If the initial value belongs to $\mathbb{H}^1$, then we further obtain the following estimates.
\begin{lemma}\label{lem35} Let $\mathcal{O}\subset\mathbb{R}^d$,~$d=1,2,3$, be a bounded domain with $C^{2,1}$-boundary. Then for any $p\geq1$, $n\in\mathbb{N}$ and every $t\in[0,T]$, there exists a positive constant $C=C(\|\mathbf{u}_0\|_{\mathbb{H}^1},p,\mathbf{h},\mathbf{g},T)$ independent of $n$ such that
\begin{equation}\label{1lem35}
\begin{split}
&\mathbb{E}\sup_{s\in[0,t]}\|\mathbf{u}_n(s)\|_{\mathbb{H}^1}^{2p}+\mathbb{E}\left(\int_0^t\|\mathbf{u}_n\|_{\mathbb{H}^3}^2\,\mathrm{d}s\right)^p\leq C,
\end{split}
\end{equation}
and
\begin{equation}\label{1lem36}
\begin{split}
&\sum_{j=1}^5\mathbb{E}\|F_n^j(\mathbf{u}_n)\|_{L^2(0,T;(\mathbb{H}^{1})^*)}^p\leq C.
\end{split}
\end{equation}
\end{lemma}
\begin{proof}[\emph{\textbf{Proof}}]Here we will make full use of the functional $\{\mathbf{\bar{F}}: \mathbf{u}_n\mapsto\frac{1}{2}\|\nabla\mathbf{u}_n\|_{\mathbb{L}^2}^2+\frac{1}{2}\|\mathbf{u}_n\|_{\mathbb{L}^4}^4-\|\mathbf{u}_n\|_{\mathbb{L}^2}^2\}$. Applying the It\^{o} formula to $\{\mathbf{F}_1: \mathbf{u}_n\mapsto\frac{1}{2}\|\nabla\mathbf{u}_n\|_{\mathbb{L}^2}^2\}$, $\{\mathbf{F}_2: \mathbf{u}_n\mapsto\frac{1}{2}\|\mathbf{u}_n\|_{\mathbb{L}^4}^4\}$ as well as $\{\mathbf{F}_3: \mathbf{u}_n\mapsto\|\mathbf{u}_n\|_{\mathbb{L}^2}^2\}$ respectively and recalling the definition of $\mathbf{H}_{\textrm{eff}}^n$ in \eqref{nota1}, we infer that
\begin{equation}\label{2lem35}
\begin{split}
&\mathbf{\bar{F}}(\mathbf{u}_n)=\frac{1}{2}\|\nabla\mathbf{u}_n(0)\|_{\mathbb{L}^2}^2+\frac{1}{2}\|\mathbf{u}_n(0)\|_{\mathbb{L}^4}^4-\|\mathbf{u}_n(0)\|_{\mathbb{L}^2}^2\\
&-\int_0^t(F_n(\mathbf{u}_n,\mathbf{H}_{\textrm{eff}}^n),\mathbf{H}_{\textrm{eff}}^n)_{\mathbb{L}^2}\,\mathrm{d}s+\sum_{j=1}^3\int_0^t\int_{\mathbb{B}}\mathbf{F}_j(\Phi_n(l,\mathbf{u}_n))-\mathbf{F}_j(\mathbf{u}_n)\tilde{\eta}(\mathrm{d}s,\mathrm{d}l)\\
&+\sum_{j=1}^3\int_0^t\int_{\mathbb{B}}\mathbf{F}_j(\Phi_n(l,\mathbf{u}_n))-\mathbf{F}_j(\mathbf{u}_n)\nu(\mathrm{d}l)\mathrm{d}s+\int_0^t\int_{\mathbb{B}}l(J_n(\mathbf{u}_n),\mathbf{H}_{\textrm{eff}}^n)_{\mathbb{L}^2}\nu(\mathrm{d}l)\mathrm{d}s.
\end{split}
\end{equation}
Noting that $(F _n(\mathbf{u}_n,\mathbf{H}_{\textrm{eff}}^n),\mathbf{H}_{\textrm{eff}}^n)_{\mathbb{L}^2}=\|\mathbf{H}_{\textrm{eff}}^n\|_{\mathbb{L}^2}^2+\|\nabla\mathbf{H}_{\textrm{eff}}^n\|_{\mathbb{L}^2}^2$,
we see from \eqref{2lem35} that
\begin{equation}\label{4lem35}
\begin{split}
&\mathbf{\bar{F}}(\mathbf{u}_n)+\int_0^t\|\mathbf{H}_{\textrm{eff}}^n\|_{\mathbb{L}^2}^2\,\mathrm{d}s+\int_0^t\|\nabla\mathbf{H}_{\textrm{eff}}^n\|_{\mathbb{L}^2}^2\,\mathrm{d}s\\
&\leq C+\sum_{j=1}^3\int_0^t\int_{\mathbb{B}}\mathbf{F}_j(\Phi_n(l,\mathbf{u}_n))-\mathbf{F}_j(\mathbf{u}_n)\tilde{\eta}(\mathrm{d}s,\mathrm{d}l)\\
&+\sum_{j=1}^3\int_0^t\int_{\mathbb{B}}\mathbf{F}_j(\Phi_n(l,\mathbf{u}_n))-\mathbf{F}_j(\mathbf{u}_n)\nu(\mathrm{d}l)\mathrm{d}s+\int_0^t\int_{\mathbb{B}}l(J_n(\mathbf{u}_n),\mathbf{H}_{\textrm{eff}}^n)_{\mathbb{L}^2}\nu(\mathrm{d}l)\mathrm{d}s.
\end{split}
\end{equation}
According to Lemma \ref{lem2-2-1} and Remark \ref{rem331}, we infer that
\begin{equation}\label{5lem35}
\begin{split}
&\sum_{j=1}^3\int_0^t\int_{\mathbb{B}}\mathbf{F}_j(\Phi_n(l,\mathbf{u}_n))-\mathbf{F}_j(\mathbf{u}_n)\nu(\mathrm{d}l)\mathrm{d}s\leq\int_0^tC(1+\|\mathbf{u}_n\|_{\mathbb{H}^1}^2+\|\mathbf{u}_n\|_{\mathbb{L}^4}^4)\,\mathrm{d}s.
\end{split}
\end{equation}
To estimate the last term of \eqref{4lem35}, by H\"{o}lder's inequality and Young's inequality, we have
\begin{equation}\label{6lem35}
\begin{split}
&\left|\int_0^t\int_{\mathbb{B}}l(J_n(\mathbf{u}_n),\mathbf{H}_{\textrm{eff}}^n)_{\mathbb{L}^2}\nu(\mathrm{d}l)\mathrm{d}s\right|\leq\varepsilon\int_0^t\|\mathbf{H}_{\textrm{eff}}^n\|_{\mathbb{L}^2}^2\,\mathrm{d}s+C_{\varepsilon}\int_0^t1+\|\mathbf{u}_n\|_{\mathbb{L}^2}^2\,\mathrm{d}s.
\end{split}
\end{equation}
Moreover,  using the BDG inequality, the H\"{o}lder inequality as well as Young's inequality, and then noting Remark \ref{rem331} and applying Lemma \ref{lem32}, we infer that for all $p\geq1$
\begin{equation}\label{7lem35}
\begin{split}
&\sum_{j=1}^3\mathbb{E}\left[\sup_{s\in[0,t]}\left|\int_0^s\int_{\mathbb{B}}\mathbf{F}_j(\Phi_n(l,\mathbf{u}_n))-\mathbf{F}_j(\mathbf{u}_n)\tilde{\eta}(\mathrm{d}s',\mathrm{d}l)\right|^p\right]\\
&\leq \sum_{j=1}^3C\mathbb{E}\left|\int_0^t\int_{\mathbb{B}}\left[\mathbf{F}_j(\Phi_n(l,\mathbf{u}_n))-\mathbf{F}_j(\mathbf{u}_n)\right]^2\nu(\mathrm{d}l)\mathrm{d}s\right|^\frac{p}{2}\\
&\leq C+C\mathbb{E}\left(\int_0^t\| \mathbf{u}_n\|_{\mathbb{L}^2}^{4}\,\mathrm{d}s\right)^{\frac{p}{2}}+C\mathbb{E}\left(\int_0^t\| \nabla\mathbf{u}_n\|_{\mathbb{L}^2}^{4}\,\mathrm{d}s\right)^{\frac{p}{2}}+C\mathbb{E}\left(\int_0^t\| \mathbf{u}_n\|_{\mathbb{L}^4}^{8}\,\mathrm{d}s\right)^{\frac{p}{2}}\\
&\leq C+C\mathbb{E}\left(\sup_{s\in[0,t]}\|\nabla\mathbf{u}_n(s)\|_{\mathbb{L}^2}^{2}\int_0^t\|\nabla \mathbf{u}_n\|_{\mathbb{L}^2}^{2}\,\mathrm{d}s\right)^{\frac{p}{2}}+C\mathbb{E}\left(\sup_{s\in[0,t]}\|\mathbf{u}_n(s)\|_{\mathbb{L}^4}^{4}\int_0^t\| \mathbf{u}_n\|_{\mathbb{L}^4}^{4}\,\mathrm{d}s\right)^{\frac{p}{2}}\\
&\leq C+\varepsilon\mathbb{E}\sup_{s\in[0,t]}\|\nabla\mathbf{u}_n(s)\|_{\mathbb{L}^2}^{2p} +C_{\varepsilon}\mathbb{E}\left(\int_0^t\|\nabla\mathbf{u}_n\|_{\mathbb{L}^2}^{2}\,\mathrm{d}s\right)^p+\varepsilon\mathbb{E}\sup_{s\in[0,t]}\|\mathbf{u}_n(s)\|_{\mathbb{L}^4}^{4p}\\
& +C_{\varepsilon}\mathbb{E}\left(\int_0^t\|\mathbf{u}_n\|_{\mathbb{L}^4}^{4}\,\mathrm{d}s\right)^p.
\end{split}
\end{equation}
Plugging \eqref{5lem35}-\eqref{7lem35} into \eqref{4lem35} and choosing $\varepsilon$ small enough, we derive that
\begin{equation*}
\begin{split}
&\mathbb{E}\sup_{s\in[0,t]}\left(\|\nabla\mathbf{u}_n(s)\|_{\mathbb{L}^2}^{2p}+\|\mathbf{u}_n(s)\|_{\mathbb{L}^4}^{4p}\right)+\mathbb{E}\left(\int_0^t\|\mathbf{H}_{\textrm{eff}}^n\|_{\mathbb{H}^1}^2\,\mathrm{d}s\right)^p\\
&\leq C+C\mathbb{E}\int_0^t\|\nabla\mathbf{u}_n\|_{\mathbb{L}^2}^{2p}+\|\mathbf{u}_n\|_{\mathbb{L}^4}^{4p}\,\mathrm{d}s,
\end{split}
\end{equation*}
which combined with the Gronwall lemma implies that
\begin{equation}\label{ss1}
\begin{split}
&\mathbb{E}\sup_{s\in[0,t]}\left(\|\nabla\mathbf{u}_n(s)\|_{\mathbb{L}^2}^{2p}+\|\mathbf{u}_n(s)\|_{\mathbb{L}^4}^{4p}\right)+\mathbb{E}\left(\int_0^t\|\mathbf{H}_{\textrm{eff}}^n\|_{\mathbb{H}^1}^2\,\mathrm{d}s\right)^p\leq C.
\end{split}
\end{equation}
Moreover, recalling the definition of $\mathbf{H}_{\textrm{eff}}^n$ and using the embedding $\mathbb{H}^1\hookrightarrow \mathbb{L}^{6}$, we see from \eqref{1lem32} and \eqref{ss1} that
\begin{equation*}
\begin{split}
&\mathbb{E}\left(\int_0^t\|\nabla\Delta\mathbf{u}_n\|_{\mathbb{L}^2}^2\,\mathrm{d}s\right)^p\\
&\leq C\mathbb{E}\left(\int_0^t\|\nabla\mathbf{u}_n\|_{\mathbb{L}^2}^2\,\mathrm{d}s\right)^p+C\mathbb{E}\left(\int_0^t\|\nabla\mathbf{H}_{\textrm{eff}}^n\|_{\mathbb{L}^2}^2\,\mathrm{d}s\right)^p+C\mathbb{E}\left(\int_0^t\|\nabla(|\mathbf{u}_n|^2\mathbf{u}_n)\|_{\mathbb{L}^2}^2\,\mathrm{d}s\right)^p\\
&\leq C+\mathbb{E}\left(\int_0^t\|\mathbf{u}_n\|_{\mathbb{L}^6}^4\|\nabla\mathbf{u}_n\|_{\mathbb{L}^6}^2\,\mathrm{d}s\right)^p\leq C+C\mathbb{E}\left(\sup_{s\in[0,t]}\|\mathbf{u}_n(s)\|_{\mathbb{H}^1}^4\int_0^t\|\mathbf{u}_n\|_{\mathbb{H}^2}^2\,\mathrm{d}s\right)^p\\
&\leq C+C\mathbb{E}\sup_{s\in[0,t]}\|\mathbf{u}_n(s)\|_{\mathbb{H}^1}^{8p}+C\mathbb{E}\left(\int_0^t\|\mathbf{u}_n\|_{\mathbb{H}^2}^2\,\mathrm{d}s\right)^{2p}\leq C.
\end{split}
\end{equation*}
Inequality \eqref{1lem35} then follows from the standard elliptic regularity result with Neumann boundary data \cite{grisvard2011elliptic}.

The following demonstrates the validity of inequality \eqref{1lem36}. Let $\phi\in \mathbb{H}^1$. By the H\"{o}lder inequality and $\mathbb{H}^1\hookrightarrow \mathbb{L}^6$, it follows that
\begin{equation*}
\begin{split}
|\left(F_n^1(\mathbf{u}_n),\phi\right)_{\mathbb{L}^2}|&\leq|\left(\nabla\mathbf{u}_n,\nabla\phi\right)_{\mathbb{L}^2}|+|\left(\mathbf{u}_n,\phi\right)_{\mathbb{L}^2}|\leq C\|\mathbf{u}_n\|_{\mathbb{H}^1}\|\phi\|_{\mathbb{H}^1},\\
|\left(F_n^2(\mathbf{u}_n),\phi\right)_{\mathbb{L}^2}|&=|\left(\nabla\Delta\mathbf{u}_n,\nabla\phi\right)_{\mathbb{L}^2}|\leq\|\nabla\Delta\mathbf{u}_n\|_{\mathbb{L}^2}\|\phi\|_{\mathbb{H}^1},\\
|\left(F_n^3(\mathbf{u}_n),\phi\right)_{\mathbb{L}^2}|&\leq C\|\mathbf{u}_n\|_{\mathbb{L}^2}\|\mathbf{u}_n\|_{\mathbb{L}^6}^2\|\phi\|_{\mathbb{L}^6}\leq C\|\mathbf{u}_n\|_{\mathbb{L}^2}\|\mathbf{u}_n\|_{\mathbb{H}^1}^2\|\phi\|_{\mathbb{H}^1},\\
|\left(F_n^4(\mathbf{u}_n),\phi\right)_{\mathbb{L}^2}|&\leq C|\left(\mathbf{u}_n\times\nabla\mathbf{u}_n,\nabla\phi\right)_{\mathbb{L}^2}|\leq C\|\mathbf{u}_n\|_{\mathbb{L}^4}\|\nabla\mathbf{u}_n\|_{\mathbb{L}^4}\|\nabla\phi\|_{\mathbb{L}^2}\\
&\leq C\|\mathbf{u}_n\|_{\mathbb{H}^1}\|\mathbf{u}_n\|_{\mathbb{H}^2}\|\phi\|_{\mathbb{H}^1},\\
|\left(F_n^5(\mathbf{u}_n),\phi\right)_{\mathbb{L}^2}|&\leq C\|\nabla(|\mathbf{u}_n|^2\mathbf{u}_n)\|_{\mathbb{L}^2}\|\nabla\phi\|_{\mathbb{L}^2}\leq C\||\mathbf{u}_n|^2|\nabla\mathbf{u}_n|\|_{\mathbb{L}^2}\|\nabla\phi\|_{\mathbb{L}^2}\\
&\leq C\|\mathbf{u}_n\|_{\mathbb{L}^6}^2\|\nabla\mathbf{u}_n\|_{\mathbb{L}^6}\|\nabla\phi\|_{\mathbb{L}^2}\leq C\|\mathbf{u}_n\|_{\mathbb{H}^1}^2\|\mathbf{u}_n\|_{\mathbb{H}^2}\|\phi\|_{\mathbb{H}^1}.
\end{split}
\end{equation*}
Thus by using the inequality \eqref{1lem35}, we derive that for all $p\geq1$
\begin{equation*}\label{3lem36}
\begin{split}
&\sum_{j=1}^5\mathbb{E}\|F_n^j(\mathbf{u}_n)\|_{L^2(0,T;(\mathbb{H}^{1})^*)}^p\\
&\leq C+C\mathbb{E}\left(\int_0^T\|\mathbf{u}_n\|_{\mathbb{H}^{1}}^2\,\mathrm{d}s\right)^{\frac{p}{2}}+C\mathbb{E}\left(\int_0^T\|\nabla\Delta\mathbf{u}_n\|_{\mathbb{L}^{2}}^2\,\mathrm{d}s\right)^{\frac{p}{2}}+C\mathbb{E}\left(\int_0^T\|\mathbf{u}_n\|_{\mathbb{L}^{2}}^2\|\mathbf{u}_n\|_{\mathbb{H}^{1}}^4\,\mathrm{d}s\right)^{\frac{p}{2}}\\
&+C\mathbb{E}\left(\int_0^T\|\mathbf{u}_n\|_{\mathbb{H}^{1}}^2\|\mathbf{u}_n\|_{\mathbb{H}^{2}}^2\,\mathrm{d}s\right)^{\frac{p}{2}}+C\mathbb{E}\left(\int_0^T\|\mathbf{u}_n\|_{\mathbb{H}^{1}}^4\|\mathbf{u}_n\|_{\mathbb{H}^{2}}^2\,\mathrm{d}s\right)^{\frac{p}{2}}\\
&\leq C+C\mathbb{E}\left[\sup_{t\in[0,T]}\|\mathbf{u}_n(t)\|_{\mathbb{H}^{1}}^{2p}\left(\int_0^T\|\mathbf{u}_n\|_{\mathbb{H}^{2}}^2\,\mathrm{d}s\right)^{\frac{p}{2}}\right]\\
&\leq C+C\mathbb{E}\sup_{t\in[0,T]}\|\mathbf{u}_n(t)\|_{\mathbb{H}^{1}}^{4p}+C\mathbb{E}\left(\int_0^T\|\mathbf{u}_n\|_{\mathbb{H}^{2}}^2\,\mathrm{d}s\right)^{p}\leq C.
\end{split}
\end{equation*}
The proof is completed.\end{proof}

\section{Compactness result}\label{sec4}
In this section we shall provide a compactness criterion in the following phase space.
\begin{equation*}\label{def4-4}
\begin{split}
\mathcal{Z}_T^1:=\mathbb{D}([0,T];(\mathbb{H}^{\beta_1})^*)\cap\mathbb{D}([0,T];\mathbb{H}^1_w)\cap L^2_w(0,T;\mathbb{H}^3) \cap L^2(0,T;\mathbb{W}^{2,4})\cap L^p(0,T;\mathbb{L}^4),
\end{split}
\end{equation*}
for some $\beta_1>1$ and $p>1$. Let $\mathcal{T}^1$ be the supremum of the corresponding topology.

At first we prove a strong convergence result that can be effectively used in stochastic parabolic partial differential equations driven by L\'{e}vy noise.

\begin{lemma}\label{adlem2} Let $B_0\subset B_1\subset B_2$ be three Banach spaces where $B_0$ is also a reflexive space. We assume that the embedding of $B_1$ in $B_2$ is continuous and the embedding of $B_0$ in $B_1$ is compact. Let $p\in(1,\infty)$. If there exists a sequence $\{f_n\}_{n\in\mathbb{N}}$ such that
\begin{enumerate}
\item[(1)] $f_n\rightarrow f$ in $L^p(0,T;B_2)$;
\item[(2)] $\sup_{n\in\mathbb{N}}\|f_n\|_{L^p(0,T;B_0)}<\infty$,
\end{enumerate}
then there exists a subsequence $\{f_{n_k}\}$ such that
$
f_{n_k}\rightarrow f~\textrm{in}~L^p(0,T;B_1).
$
\end{lemma}
\begin{proof}[\emph{\textbf{Proof}}] By the Lions lemma, for every $\varepsilon>0$ there exists a constant $C_{\varepsilon}$ such that
\begin{equation*}
\begin{split}
\|g\|_{B_1}^p\leq\varepsilon\|g\|_{B_0}^p+C_{\varepsilon}\|g\|_{B_2}^p,~g\in B_0.
\end{split}
\end{equation*}
Noting that $B_0$ is a reflexive Banach space and $\{f_n\}$ is bounded in $L^p(0,T;B_0)$, by the Banach-Alaoglu theorem there exist a subsequence $\{f_{n_k}\}$ and a point $\tilde{f}$ such that $f_{n_k}$ weakly converges $\tilde{f}$ in $L^p(0,T;B_0)$. Due to the uniqueness of the weak limit, we see from condition $(1)$ that $f_{n_k}$ weakly converges $f$ in $L^p(0,T;B_0)$. Thus for almost all $t\in[0,T]$
\begin{equation*}
\begin{split}
\|f_{n_k}(t)-f(t)\|_{B_1}^p\leq\varepsilon\|f_{n_k}(t)-f(t)\|_{B_0}^p+C_{\varepsilon}\|f_{n_k}(t)-f(t)\|_{B_2}^p,
\end{split}
\end{equation*}
which means that
\begin{equation*}
\begin{split}
\|f_{n_k}-f\|_{L^p(0,T;B_1)}^p\leq\varepsilon\|f_{n_k}-f\|_{L^p(0,T;B_0)}^p+C_{\varepsilon}\|f_{n_k}-f\|_{L^p(0,T;B_2)}^p.
\end{split}
\end{equation*}
Passing the upper limit as $k\rightarrow\infty$ and noting that
\begin{equation*}
\begin{split}
\|f_{n_k}-f\|_{L^p(0,T;B_0)}^p\leq C_p(\|f_{n_k}\|_{L^p(0,T;B_0)}^p+\|f\|_{L^p(0,T;B_0)}^p)\leq C,
\end{split}
\end{equation*}
we infer that
$
\limsup_{k\rightarrow\infty}\|f_{n_k}-f\|_{L^p(0,T;B_1)}^p\leq C\varepsilon.
$
Due to the arbitrariness of $\varepsilon$, we have
$
\lim_{k\rightarrow\infty}\|f_{n_k}-f\|_{L^p(0,T;B_1)}^p=0.
$
The proof is thus complete.\end{proof}

\begin{remark}\label{adrem2} If the condition (1) is replace by : $f_n\rightarrow f$ in $\mathbb{D}([0,T];B_2)$, then the result is also valid. Since $f_n\rightarrow f$ in $\mathbb{D}([0,T];B_2)$, $f_n(t)\rightarrow f(t)$ in $B_2$ for all continuity points of function $f$ \cite{billingsley2013convergence}. By condition (2) and Lebesgue dominated convergence theorem, it follows that for all $p\in(1,\infty)$, $f_n\rightarrow f~\textrm{in}~L^p(0,T;B_2)$, which means that Lemma \ref{adlem2} can be used.
\end{remark}

The key compactness result is as follows.
\begin{proposition}\label{pro4-2}
A set $\mathcal{K}\subset\mathcal{Z}_T^1$ is $\mathcal{T}^1$-relatively compact if the following conditions hold:
\begin{enumerate}
\item[(1)] $\sup_{f\in\mathcal{K}}\sup_{t\in[0,T]}\|f(t)\|_{\mathbb{H}^1}<\infty$;
\item[(2)] $\sup_{f\in\mathcal{K}}\int_0^T\|f(t)\|_{\mathbb{H}^3}^2\textrm{d}t<\infty$;
\item[(3)] $\lim_{\delta\rightarrow 0}\sup_{f\in\mathcal{K}}\sup_{s,t\in[0,T],|t-s|\leq\delta}\|f(t)-f(s)\|_{(\mathbb{H}^{\beta_1})^*}=0$.
\end{enumerate}
\end{proposition}
\begin{proof}[\emph{\textbf{Proof}}] Without loss of generality, we assume that $\mathcal{K}$ is a closed subset of $\mathcal{Z}_T^1$. Let $(f_m)$ be sequence in $\mathcal{K}$. We shall first prove that $(f_m)$ is compact in $\mathbb{D}([0,T];(\mathbb{H}^{\beta_1})^*)$. Indeed, from condition (1) we see that for every $t\in[0,T]$, the set $\{f_m(t)\}$ is bounded in $\mathbb{H}^1$. Since the embedding $\mathbb{H}^1\subset (\mathbb{H}^{\beta_1})^*$ is compact \cite{adams2003sobolev}, the set $\{f_m(t)\}$ is compact in $(\mathbb{H}^{\beta_1})^*$. Moreover by using the condition (3) and the result provided Lemma \ref{lem4-1}, we infer that there exists a subsequence $(f_{m_k})$ such that
\begin{equation*}
\begin{split}
f_{m_k}\rightarrow f~\textrm{in}~\mathbb{D}([0,T];(\mathbb{H}^{\beta_1})^*).
\end{split}
\end{equation*}
This together with condition (1) and Lemma \ref{adlem4-1} implies that there exists a subsequence of $(f_{m_k})$ (still denoted by $(f_{m_k})$) such that
\begin{equation*}
\begin{split}
f_{m_k}\rightarrow f~\textrm{in}~\mathbb{D}([0,T];(\mathbb{H}^{\beta_1})^*)\cap \mathbb{D}([0,T];\mathbb{B}_w^1).
\end{split}
\end{equation*}

In addition, by the Banach-Alaoglu theorem, condition (2) implies that the set $\mathcal{K}$ is compact in $L^2_w([0,T];\mathbb{H}^3)$. Thus there exists a subsequence of $(f_{m_k})$ (still denoted by $(f_{m_k})$) such that
\begin{equation*}
\begin{split}
f_{m_k}\rightarrow f~\textrm{in}~\mathbb{D}([0,T];(\mathbb{H}^{\beta_1})^*)\cap\mathbb{D}([0,T];\mathbb{B}^1_w)\cap L^2_w(0,T;\mathbb{H}^3).
\end{split}
\end{equation*}

Moreover, noting that $\mathbb{H}^3\hookrightarrow\hookrightarrow\mathbb{W}^{2,4}\hookrightarrow(\mathbb{H}^{\beta_1})^*$ and using the condition (2), we derive from Lemma \ref{adlem2} and Remark \ref{adrem2} that
there exists a subsequence of $(f_{m_k})$ (still denoted by $(f_{m_k})$) such that
$
f_{m_k}\rightarrow f~\textrm{in}~L^2(0,T;\mathbb{W}^{2,4})~\textrm{as}~k\rightarrow\infty.
$ Thus
\begin{equation*}
\begin{split}
f_{m_k}\rightarrow f~\textrm{in}~\mathbb{D}([0,T];(\mathbb{H}^{\beta_1})^*)\cap \mathbb{D}([0,T];\mathbb{B}^1_w)\cap L^2_w(0,T;\mathbb{H}^3) \cap L^2(0,T;\mathbb{W}^{2,4}).
\end{split}
\end{equation*}
Similarly, using the fact that $\mathbb{H}^1\hookrightarrow\hookrightarrow\mathbb{L}^4\hookrightarrow(\mathbb{H}^{\beta_1})^*$ and noting condition (1), we infer that there exists a subsequence of $(f_{m_k})$ (still denoted by $(f_{m_k})$) such that
$
f_{m_k}\rightarrow f~\textrm{in}~L^p(0,T;\mathbb{L}^{4})~\textrm{as}~k\rightarrow\infty,
$ which means that
\begin{equation*}
\begin{split}
f_{m_k}\rightarrow f~\textrm{in}~\mathbb{D}([0,T];(\mathbb{H}^{\beta_1})^*)\cap \mathbb{D}([0,T];\mathbb{B}^1_w)\cap L^2_w(0,T;\mathbb{H}^3) \cap L^2(0,T;\mathbb{W}^{2,4})\cap L^p(0,T;\mathbb{L}^4).
\end{split}
\end{equation*}
Noting that $\mathbb{D}([0,T];\mathbb{B}_w^1)$ is a metric subspace of $\mathbb{D}([0,T];\mathbb{H}_w^1)$, we naturally obtain the conclusion of Proposition \ref{pro4-2}.\end{proof}

Based on the aforementioned deterministic compactness result, we further obtain the following tightness criterion.
\begin{corollary}\label{cor4-1} Let $\mathbb{P}_n^1$ be the law of $\mathbf{u}_n$ on $\mathcal{Z}_T^1$. If there exists a positive constant $C_1>0$ such that
\begin{equation*}
\begin{split}
\sup_{n\in\mathbb{N}}\mathbb{E}(\sup_{t\in[0,T]}\|\mathbf{u}_n(t)\|_{\mathbb{H}^1}^2)+\sup_{n\in\mathbb{N}}\mathbb{E}\left(\int_0^T\|\mathbf{u}_n(t)\|_{\mathbb{H}^3}^2\,\mathrm{d}t\right)\leq C_1
\end{split}
\end{equation*}
and moreover $(\mathbf{u}_n)_{n\in\mathbb{N}}$ satisfies the Aldous condition \cite{aldous1978stopping} in $(\mathbb{H}^{\beta_1})^*$, then for every $\varepsilon>0$ there exists a compact subset $K_{\varepsilon}^1$ of $\mathcal{Z}_T^1$ such that $\mathbb{P}_n^1(K_{\varepsilon}^1)\geq1-\varepsilon$.
\end{corollary}
\begin{proof}[\emph{\textbf{Proof}}] According to the Chebyshev inequality, we see that for any $R>0$
\begin{equation*}
\begin{split}
\mathbb{P}\left(\sup_{t\in[0,T]}\|\mathbf{u}_n(t)\|_{\mathbb{H}^1}>R\right)\leq\frac{\mathbb{E}(\sup_{t\in[0,T]}\|\mathbf{u}_n(t)\|_{\mathbb{H}^1}^2)}{R^2}\leq\frac{C_1}{R^2}.
\end{split}
\end{equation*}
Choosing $R_1\geq\sqrt{\frac{3C_1}{\varepsilon}}$, then we have
$
\mathbb{P}\left(\sup_{t\in[0,T]}\|\mathbf{u}_n(t)\|_{\mathbb{H}^1}>R_1\right)\leq\frac{\varepsilon}{3}.
$ Let $S_1:=\{f\in\mathcal{Z}_T^1:\sup_{t\in[0,T]}\|f(t)\|_{\mathbb{H}^1}\leq R_1\}$. Similarly,  by the Chebyshev inequality we infer that
$
\mathbb{P}\left(\|\mathbf{u}_n\|_{L^2(0,T;\mathbb{H}^3)}>R_1\right)\leq\frac{\varepsilon}{3}.
$ Let $S_2:=\{f\in\mathcal{Z}_T^1:\|f\|_{L^2(0,T;\mathbb{H}^3)}\leq R_1\}$. Moreover, by Lemma \ref{lem4-2}, there exists a subset $A_{\frac{\varepsilon}{3}}\subset\mathbb{D}([0,T];\mathbb{H}^{-\beta_1})$ such that $\mathbb{P}_n^1(A_{\frac{\varepsilon}{3}})\geq1-\frac{\varepsilon}{3}$ and $\lim_{\delta\rightarrow0}\sup_{f\in A_{\frac{\varepsilon}{3}}}w_{[0,T],\mathbb{H}^{-\beta_1}}(f,\delta)=0$. Thus it is sufficient to define $K_{\varepsilon}^1$ as the closure of the set $S_1\cap S_2\cap A_{\frac{\varepsilon}{3}}$ in $\mathcal{Z}_T^1$. By using Proposition \ref{pro4-2}, we infer that $K_{\varepsilon}^1$ is compact in $\mathcal{Z}_T^1$. The proof is thus complete.\end{proof}

\section{Existence and uniqueness of pathwise solution}\label{sec5}
This section is devoted to provide a complete proof for Theorem \ref{the1}. Let us denote $\eta_n:=\eta,~n\in\mathbb{N}$ and $\bar{\mathbb{N}}:=\mathbb{N}\cup\{\infty\}$. Let $(\mathcal{S},\varrho)$ be a measurable space and let $M_{\bar{\mathbb{N}}}(\mathcal{S})$ be the set of all $\bar{\mathbb{N}}$-valued measures on $(\mathcal{S},\varrho)$.
\begin{proof}[\emph{\textbf{Proof of Theorem \ref{the1}}}] At first we shall prove that the sequence of laws $\mathscr{L}(\mathbf{u}_n)$ is tight on the space $\mathcal{Z}_T^1$.  Thanks to Corollary \ref{cor4-1} and Lemma \ref{lem35}, it is sufficient to show that the sequence $(\mathbf{u}_n)$ satisfies the Aldous condition on the space $(\mathbb{H}^{\beta_1})^*$. Assume that for every sequence $(\tau_n)_{n\in\mathbb{N}}$ of $\mathbb{F}$-stopping times with $\tau_n\leq T$ and $\theta>0$. Then we have
\begin{equation*}
\begin{split}
&\mathbf{u}_n(\tau_n+\theta)-\mathbf{u}_n(\tau_n)=\sum_{j=1}^5\int_{\tau_n}^{\tau_n+\theta}F_n^j(\mathbf{u}_n)\,\mathrm{d}t+\int_{\tau_n}^{\tau_n+\theta}\mathbf{b}_n(\mathbf{u}_n)\,\mathrm{d}t+\int_{\tau_n}^{\tau_n+\theta}\int_{\mathbb{B}}\mathbf{G}_n(l,\mathbf{u}_n)\tilde{\eta}(\mathrm{d}t,\mathrm{d}l).
\end{split}
\end{equation*}
Since $\mathbb{L}^2\hookrightarrow(\mathbb{H}^{1})^*\hookrightarrow(\mathbb{H}^{\beta_1})^*$ for $\beta_1>1$, we see from Remark \ref{rem331} and inequality \eqref{1lem36} that
\begin{equation*}
\begin{split}
&\sum_{j=1}^5\mathbb{E}\left\|\int_{\tau_n}^{\tau_n+\theta}F_n^j(\mathbf{u}_n)\,\mathrm{d}t\right\|_{(\mathbb{H}^{\beta_1})^*}+\mathbb{E}\left\|\int_{\tau_n}^{\tau_n+\theta}\mathbf{b}_n(\mathbf{u}_n)\,\mathrm{d}t\right\|_{(\mathbb{H}^{\beta_1})^*}\\
&\leq\sum_{j=1}^5C\mathbb{E}\int_{\tau_n}^{\tau_n+\theta}\|F_n^j(\mathbf{u}_n)\|_{(\mathbb{H}^{1})^*}\,\mathrm{d}t+C\mathbb{E}\int_{\tau_n}^{\tau_n+\theta}\|\mathbf{b}_n(\mathbf{u}_n)\|_{\mathbb{L}^{2}}\,\mathrm{d}t\\
&\leq\sum_{j=1}^5C\mathbb{E}\left(\int_{0}^{T}\|F_n^j(\mathbf{u}_n)\|_{(\mathbb{H}^{1})^*}^2\,\mathrm{d}t\right)^{\frac{1}{2}}\theta^{\frac{1}{2}}+C\mathbb{E}\left(\int_{0}^{T}\|\mathbf{b}_n(\mathbf{u}_n)\|_{\mathbb{L}^{2}}^2\,\mathrm{d}t\right)^{\frac{1}{2}}\theta^{\frac{1}{2}}\\
&\leq C\theta^{\frac{1}{2}}.
\end{split}
\end{equation*}
Moreover, by the It\^{o}-L\'{e}vy Isometry, we have
\begin{equation*}
\begin{split}
&\mathbb{E}\left\|\int_{\tau_n}^{\tau_n+\theta}\int_{\mathbb{B}}\mathbf{G}_n(l,\mathbf{u}_n)\tilde{\eta}(\mathrm{d}t,\mathrm{d}l)\right\|_{(\mathbb{H}^{\beta_1})^*}^2\leq C\mathbb{E}\left\|\int_{\tau_n}^{\tau_n+\theta}\int_{\mathbb{B}}\mathbf{G}_n(l,\mathbf{u}_n)\tilde{\eta}(\mathrm{d}t,\mathrm{d}l)\right\|_{\mathbb{L}^{2}}^2\\
&=C\mathbb{E}\int_{\tau_n}^{\tau_n+\theta}\int_{\mathbb{B}}\left\|\mathbf{G}_n(l,\mathbf{u}_n)\right\|_{\mathbb{L}^{2}}^2\nu(\mathrm{d}l)\mathrm{d}t\leq C\mathbb{E}\int_{\tau_n}^{\tau_n+\theta}1+\|\mathbf{u}_n\|_{\mathbb{L}^2}^2\,\mathrm{d}t\\
&\leq C\theta.
\end{split}
\end{equation*}
Thus the sequence of laws $\mathscr{L}(\mathbf{u}_n)$ is tight on the space $\mathcal{Z}_T^1$, which allows us to apply the generalised Jakubowski-Skorokhod embedding theorem \cite{brzezniak2018stochastic,motyl2013stochastic}. Hence there exists a subsequence $(\mathbf{u}_{n_k},\eta_{n_k})_{k\in\mathbb{N}}$, a probability space $(\Omega',\mathcal{F}',\mathbb{F}',\mathbb{P}')$, and, on this space, $\mathcal{Z}_T^1\times M_{\bar{\mathbb{N}}}([0,T]\times \mathbb{B})$-valued variables $(\mathbf{u}',\eta')$, $(\mathbf{u}_k',\eta_k'),~k\in\mathbb{N}$, such that
\begin{enumerate}
\item[(1)] $\mathscr{L}((\mathbf{u}_k',\eta_k'))=\mathscr{L}((\mathbf{u}_{n_k},\eta_{n_k}))$ for all $k\in\mathbb{N}$;
\item[(2)] $(\mathbf{u}_k',\eta_k')\rightarrow(\mathbf{u}',\eta')$ in $\mathcal{Z}_T^1\times M_{\bar{\mathbb{N}}}([0,T]\times \mathbb{B})$, $\mathbb{P}'$-a.s. as $k\rightarrow\infty$;
\item[(3)] $\eta'_n(\omega')=\eta'(\omega')$, for all $\omega'\in\Omega'$.
\end{enumerate}
We will denote above sequences again by $(\mathbf{u}_{n},\eta_{n})_{n\in\mathbb{N}}$ and $(\mathbf{u}_{n}',\eta_{n}')_{n\in\mathbb{N}}$. In particular, there exists a constant $C>0$ independent of $n$ such that for every $p\geq1$
\begin{equation}\label{1lem55-4}
\begin{split}
&\mathbb{E}'\sup_{s\in[0,t]}\|\mathbf{u}_n'(s)\|_{\mathbb{H}^1}^{2p}+\mathbb{E}'\left(\int_0^t\| \mathbf{u}_n'(s)\|_{\mathbb{H}^3}^2\,\mathrm{d}s\right)^p+\sum_{j=1}^5\mathbb{E}'\|F_n^j(\mathbf{u}_n')\|_{L^2(0,T;(\mathbb{H}^{1})^*)}^p\leq C.
\end{split}
\end{equation}
Furthermore, we have the following weak convergence result.
\begin{lemma}\label{lem5-5} For any $p\geq1$, there holds
\begin{equation*}\label{1lem5-5}
\begin{split}
\mathbf{u}_n'\rightarrow\mathbf{u}'~\textrm{weakly in}~L^{2p}(\Omega';L^{\infty}(0,T;\mathbb{H}^1)\cap L^{2}(0,T;\mathbb{H}^3).
\end{split}
\end{equation*}
\end{lemma}
\begin{proof}[\emph{\textbf{Proof}}] We shall fist prove that
\begin{equation}\label{2lem5-5}
\begin{split}
\mathbf{u}_n'\rightarrow\mathbf{u}'~\textrm{weakly in}~L^{\frac{4}{3}}(\Omega';L^{4}(0,T;\mathbb{L}^2)).
\end{split}
\end{equation}
Since $\mathbf{u}_n'\rightarrow\mathbf{u}'~\textrm{weakly in}~L^{4}(0,T;\mathbb{L}^4)\hookrightarrow L^{4}(0,T;\mathbb{L}^2)$, $\mathbb{P}'$-a.s., for any $\phi\in L^{4}(\Omega';L^{\frac{4}{3}}(0,T;\mathbb{L}^2))$ there holds
$
\int_0^T(\mathbf{u}_n',\phi)_{\mathbb{L}^2}\,\mathrm{d}t\rightarrow\int_0^T(\mathbf{u}',\phi)_{\mathbb{L}^2}\,\mathrm{d}t.
$
Moreover, by using inequality \eqref{1lem55-4}, we have
\begin{equation*}
\begin{split}
&\sup_{n\in\mathbb{N}}\mathbb{E}'\left|\int_0^T(\mathbf{u}_n',\phi)_{\mathbb{L}^2}\,\mathrm{d}t\right|^2\leq\sup_{n\in\mathbb{N}}\mathbb{E}'\left(\|\mathbf{u}_n'\|_{L^{\infty}(0,T;\mathbb{L}^2)}^2\|\phi\|_{L^1(0,T;\mathbb{L}^2)}^2\right)\\
&\leq\sup_{n\in\mathbb{N}}\|\mathbf{u}_n'\|_{L^{4}(\Omega';L^{\infty}(0,T;\mathbb{L}^2))}^2\|\phi\|_{L^{4}(\Omega';L^{1}(0,T;\mathbb{L}^2))}^2<\infty.
\end{split}
\end{equation*}
Thus by using the Vitali convergence theorem we have
$
\mathbb{E}'\int_0^T(\mathbf{u}_n',\phi)_{\mathbb{L}^2}\,\mathrm{d}t\rightarrow\mathbb{E}'\int_0^T(\mathbf{u}',\phi)_{\mathbb{L}^2}\,\mathrm{d}t,
$
which means the result \eqref{2lem5-5}. On the other hand, by using the Banach-Alaoglu theorem we infer from \eqref{1lem55-4} that there exists a subsequence of $\{\mathbf{u}_n'\}$ (still denoted by $\{\mathbf{u}_n'\}$) and $\mathbf{v}\in L^{2p}(\Omega';L^{\infty}(0,T;\mathbb{H}^1)\cap L^{2}(0,T;\mathbb{H}^3))$ such that
\begin{equation*}
\begin{split}
\mathbf{u}_n'\rightarrow\mathbf{v}~\textrm{weakly in}~L^{2p}(\Omega';L^{\infty}(0,T;\mathbb{H}^1)\cap L^{2}(0,T;\mathbb{H}^3))\subset L^{\frac{4}{3}}(\Omega';L^{4}(0,T;\mathbb{L}^2)).
\end{split}
\end{equation*}
By the uniqueness of weak limit, we infer that
\begin{equation*}
\begin{split}
\mathbf{u}'=\mathbf{v}~\textrm{in}~L^{2p}(\Omega';L^{\infty}(0,T;\mathbb{H}^1)\cap L^{2}(0,T;\mathbb{H}^3)).
\end{split}
\end{equation*}
The proof is thus complete.\end{proof}

Let us define
\begin{equation}\label{def111}
\begin{split}
\sum_{j=1}^5F^j(\mathbf{u}'):=(-\Delta\mathbf{u}'+2\mathbf{u}')-\Delta^2\mathbf{u}'-2|\mathbf{u}'|^2\mathbf{u}'-\mathbf{u}'\times\Delta \mathbf{u}'+2\Delta(|\mathbf{u}'|^2\mathbf{u}').
\end{split}
\end{equation}
Let $\phi_n$ be  sequence in $S_n$ such that $\phi_n\rightarrow\phi~\text{in}~\mathbb{H}^1$. We obtain the following convergence result, which allows us to construct the desired martingale solution.
\begin{proposition}\label{cor5-1}  For any $p\geq1$, there holds
\begin{equation}\label{1cor5-1}
\begin{split}
\lim_{n\rightarrow\infty}\sum_{j=1}^5\mathbb{E}'\left|\int_0^T\left\langle F^j_n(\mathbf{u}'_n),\phi_n\right\rangle_{(\mathbb{H}^{1})^*,\mathbb{H}^1}-\left\langle F^j(\mathbf{u}'),\phi\right\rangle_{(\mathbb{H}^{1})^*,\mathbb{H}^1}\,\mathrm{d}s\right|^p=0,
\end{split}
\end{equation}
\begin{equation}\label{1lem5-7}
\begin{split}
\lim_{n\rightarrow\infty}\mathbb{E}'\left|\int_0^T\left(\mathbf{b}_n(\mathbf{u}_n'),\phi_n\right)_{\mathbb{L}^2}-\left( \mathbf{b}(\mathbf{u}'),\phi\right)_{\mathbb{L}^2}\,\mathrm{d}s\right|^2=0,
\end{split}
\end{equation}
\begin{equation}\label{2lem5-7}
\begin{split}
\lim_{n\rightarrow\infty}\mathbb{E}'\left|\int_0^T\int_{\mathbb{B}}\left(\mathbf{G}_n(l,\mathbf{u}_n'),\phi_n\right)_{\mathbb{L}^2}-\left(\mathbf{G}(l,\mathbf{u}'),\phi\right)_{\mathbb{L}^2}\tilde{\eta}(\mathrm{d}s,\mathrm{d}l)\right|^2=0.
\end{split}
\end{equation}
\end{proposition}
\begin{proof}[\emph{\textbf{Proof}}] According to \eqref{1lem55-4}, we infer that for $p\geq1$ and $j\in\{1,2,3,4,5\}$,
\begin{equation}\label{4cor5-1}
\begin{split}
&\sup_{n\in\mathbb{N}}\mathbb{E}'\left|\int_0^T\left\langle F^j_n(\mathbf{u}'_n),\phi_n\right\rangle_{(\mathbb{H}^{1})^*,\mathbb{H}^1}\,\mathrm{d}s\right|^p\leq\sup_{n\in\mathbb{N}}\|\phi_n\|_{\mathbb{H}^1}^p\|F^j_n(\mathbf{u}'_n)\|_{L^p(\Omega';L^2(0,T;(\mathbb{H}^{1})^*))}<\infty.
\end{split}
\end{equation}
Thus by the Vitali convergence theorem, to prove \eqref{1cor5-1}, it is sufficient to show that $\mathbb{P}'$-a.s.
\begin{equation}\label{3cor5-1}
\begin{split}
\left\langle F^j_n(\mathbf{u}'_n),\phi_n\right\rangle_{(\mathbb{H}^{1})^*,\mathbb{H}^1}=(F^j_n(\mathbf{u}'_n),\phi_n)_{\mathbb{L}^2}\rightarrow\left\langle F^j(\mathbf{u}'),\phi\right\rangle_{(\mathbb{H}^{1})^*,\mathbb{H}^1},~j=1,2,3,4,5.
\end{split}
\end{equation}
To show $(F^1_n(\mathbf{u}'_n),\phi_n)_{\mathbb{L}^2}\rightarrow\left\langle F^1(\mathbf{u}'),\phi\right\rangle_{(\mathbb{H}^{1})^*,\mathbb{H}^1}$ a.s, it is sufficient to prove that
$$
\lim_{n\rightarrow\infty}\int_0^t\left(\nabla \mathbf{u}_n',\nabla\phi_n\right)_{\mathbb{L}^2}\,\mathrm{d}s=\int_0^t\left(\nabla\mathbf{u}',\nabla\phi\right)_{\mathbb{L}^2}\,\mathrm{d}s,~\textrm{a.s.}
$$
Since $\mathbf{u}_n'\rightarrow \mathbf{u}'$ in $L^2(0,T;\mathbb{H}^2)$ a.s., the above limiting process is clearly valid. Similarly, by the fact that $\mathbf{u}_n'\rightarrow\mathbf{u}'$ weakly in $L^2(0,T;\mathbb{H}^3)$ a.s., it is easy to prove that $(F^2_n(\mathbf{u}'_n),\phi_n)_{\mathbb{L}^2}\rightarrow\left\langle F^2(\mathbf{u}'),\phi\right\rangle_{(\mathbb{H}^{1})^*,\mathbb{H}^1}$ a.s. To show $(F^3_n(\mathbf{u}'_n),\phi_n)_{\mathbb{L}^2}\rightarrow\left\langle F^3(\mathbf{u}'),\phi\right\rangle_{\mathbb{H}^{1})^*,\mathbb{H}^1}$ a.s., it is sufficient to prove that
$$
\lim_{n\rightarrow\infty}\int_0^T\left(|\mathbf{u}_n'|^2\mathbf{u}_n',\phi_n\right)_{\mathbb{L}^2}\,\mathrm{d}s=\int_0^T\left(|\mathbf{u}'|^2\mathbf{u}',\phi\right)_{\mathbb{L}^2}\,\mathrm{d}s,~\textrm{a.s.}
$$
By the H\"{o}lder inequality and the fact that $\mathbf{u}_n'\rightarrow\mathbf{u}'$ in $\mathcal{Z}_T^1$ a.s, we have
\begin{equation*}
\begin{split}
&\left|\int_0^T\left(|\mathbf{u}_n'|^2\mathbf{u}_n',\phi_n\right)_{\mathbb{L}^2}\,\mathrm{d}s-\int_0^T\left(|\mathbf{u}'|^2\mathbf{u}',\phi\right)_{\mathbb{L}^2}\,\mathrm{d}s\right|\\
&\leq\left|\int_0^T\left(|\mathbf{u}_n'|^2\mathbf{u}_n',\phi_n-\phi\right)_{\mathbb{L}^2}\,\mathrm{d}s\right|+\left|\int_0^T\left(|\mathbf{u}_n'|^2\left(\mathbf{u}_n'-\mathbf{u}'\right),\phi\right)_{\mathbb{L}^2}\,\mathrm{d}s\right|+\left|\int_0^T\left(\left(|\mathbf{u}_n'|^2-|\mathbf{u}'|^2\mathbf{u}'\right),\phi\right)_{\mathbb{L}^2}\,\mathrm{d}s\right|\\
&\leq\|\phi_n-\phi\|_{\mathbb{L}^{2}}\|\mathbf{u}_n'\|_{L^3(0,T;\mathbb{L}^6)}^3+\|\phi\|_{\mathbb{L}^{2}}\|\mathbf{u}_n'\|_{L^4(0,T;\mathbb{L}^6)}^2\|\mathbf{u}_n'-\mathbf{u}'\|_{L^2(0,T;\mathbb{L}^6)}\\
&+\|\phi\|_{\mathbb{L}^{2}}\|\mathbf{u}_n'-\mathbf{u}'\|_{L^2(0,T;\mathbb{L}^6)}\|\mathbf{u}_n'+\mathbf{u}'\|_{L^4(0,T;\mathbb{L}^6)}\|\mathbf{u}'\|_{L^4(0,T;\mathbb{L}^6)}\rightarrow0~\textrm{as}~n\rightarrow\infty.
\end{split}
\end{equation*}
To show $(F^4_n(\mathbf{u}'_n),\phi_n)_{\mathbb{L}^2}\rightarrow\left\langle F^4(\mathbf{u}'),\phi\right\rangle_{(\mathbb{H}^{1})^*,\mathbb{H}^1}$ a.s., it is sufficient to prove that
$$
\lim_{n\rightarrow\infty}\int_0^t\left(\mathbf{u}'_n\times\nabla\mathbf{u}'_n,\nabla\phi_n\right)_{\mathbb{L}^2}\,\mathrm{d}s=\int_0^t\left(\mathbf{u}'\times\nabla\mathbf{u}',\nabla\phi\right)_{\mathbb{L}^2}\,\mathrm{d}s,~\textrm{a.s.}
$$
Since $\mathbf{u}_n'\rightarrow\mathbf{u}'$ in $\mathcal{Z}_T^1$, it follows that
\begin{equation*}
\begin{split}
&\left|\int_0^T\left(\mathbf{u}'_n\times\nabla\mathbf{u}'_n,\nabla\phi_n\right)_{\mathbb{L}^2}\,\mathrm{d}s-\int_0^T\left(\mathbf{u}'\times\nabla\mathbf{u}',\nabla\phi\right)_{\mathbb{L}^2}\,\mathrm{d}s\right|\\
&\leq\left|\int_0^T\left(\mathbf{u}'_n\times\nabla\mathbf{u}'_n,\nabla\phi_n-\nabla\phi\right)_{\mathbb{L}^2}\,\mathrm{d}s\right|+\left|\int_0^T\left(\mathbf{u}'_n\times\nabla\left(\mathbf{u}'_n-\mathbf{u}'\right),\nabla\phi\right)_{\mathbb{L}^2}\,\mathrm{d}s\right|\\
&+\left|\int_0^T\left(\left(\mathbf{u}'_n-\mathbf{u}'\right)\times\nabla\mathbf{u}',\nabla\phi\right)_{\mathbb{L}^2}\,\mathrm{d}s\right|\\
&\leq\|\nabla(\phi_n-\phi)\|_{\mathbb{L}^2}\|\mathbf{u}_n'\|_{L^2(0,T;\mathbb{H}^2)}^2+\|\nabla\phi\|_{\mathbb{L}^2}\|\mathbf{u}_n'\|_{L^2(0,T;\mathbb{L}^4)}\|\mathbf{u}_n'-\mathbf{u}'\|_{L^2(0,T;\mathbb{H}^2)}\\
&+\|\nabla\phi\|_{\mathbb{L}^2}\|\nabla\mathbf{u}'\|_{L^2(0,T;\mathbb{L}^4)}\|\mathbf{u}_n'-\mathbf{u}'\|_{L^2(0,T;\mathbb{L}^4)}\rightarrow0~\textrm{as}~n\rightarrow\infty.
\end{split}
\end{equation*}
To show $(F^5_n(\mathbf{u}'_n),\phi_n)_{\mathbb{L}^2}\rightarrow\left\langle F^5(\mathbf{u}'),\phi\right\rangle_{(\mathbb{H}^{1})^*,\mathbb{H}^1}$ a.s., it is sufficient to prove that
$$
\lim_{n\rightarrow\infty}\int_0^t\left(\nabla(|\mathbf{u}'_n|^2\mathbf{u}'_n),\nabla\phi_n\right)_{\mathbb{L}^2}\,\mathrm{d}s=\int_0^t\left(\nabla(|\mathbf{u}'|^2\mathbf{u}'),\nabla\phi\right)_{\mathbb{L}^2}\,\mathrm{d}s,~\textrm{a.s.}
$$
By using the triangle inequality, the H\"{o}lder inequality, the embedding $\mathbb{H}^1\hookrightarrow\mathbb{L}^6$ as well as the fact $\mathbf{u}_n'\rightarrow\mathbf{u}'$ in $\mathcal{Z}_T^1$, we infer that
\begin{equation*}
\begin{split}
&\left|\int_0^T\left(\nabla\left(|\mathbf{u}_n'|^2\mathbf{u}_n'\right),\nabla\phi_n\right)_{\mathbb{L}^2}\,\mathrm{d}s-\int_0^T\left(\nabla\left(|\mathbf{u}'|^2\mathbf{u}'\right),\nabla\phi\right)_{\mathbb{L}^2}\,\mathrm{d}s\right|\\
&\leq\left|\int_0^T\left(\nabla\left(|\mathbf{u}_n'|^2\mathbf{u}_n'\right),\nabla(\phi_n-\phi)\right)_{\mathbb{L}^2}\,\mathrm{d}s\right|+\left|\int_0^T\left(\nabla\left(|\mathbf{u}_n'|^2\mathbf{u}_n'\right)-\nabla\left(|\mathbf{u}_n'|^2\mathbf{u}'\right),\nabla\phi\right)_{\mathbb{L}^2}\,\mathrm{d}s\right|\\
&+\left|\int_0^T\left(\nabla\left(|\mathbf{u}_n'|^2\mathbf{u}'\right)-\nabla\left(|\mathbf{u}'|^2\mathbf{u}'\right),\nabla\phi\right)_{\mathbb{L}^2}\,\mathrm{d}s\right|\\
&\leq C\|\nabla(\phi-\phi_n)|_{\mathbb{L}^2}\int_0^T\|\mathbf{u}_n'\|_{\mathbb{L}^6}^2\|\nabla\mathbf{u}_n'\|_{\mathbb{L}^6}\,\mathrm{d}s+C\|\nabla\phi\|_{\mathbb{L}^2}\int_0^T\|\mathbf{u}_n'\|_{\mathbb{L}^6}\|\nabla\mathbf{u}_n'\|_{\mathbb{L}^6}\|\mathbf{u}_n'-\mathbf{u}'\|_{\mathbb{L}^6}\,\mathrm{d}s\\
&+C\|\nabla\phi\|_{\mathbb{L}^2}\int_0^T\|\mathbf{u}_n'\|_{\mathbb{L}^6}^2\|\nabla\left(\mathbf{u}_n'-\mathbf{u}'\right)\|_{\mathbb{L}^6}\,\mathrm{d}s+C\|\nabla\phi\|_{\mathbb{L}^2}\int_0^T\|\mathbf{u}_n'\|_{\mathbb{L}^6}\|\mathbf{u}'\|_{\mathbb{L}^6}\|\nabla(\mathbf{u}_n'-\mathbf{u}')\|_{\mathbb{L}^6}\,\mathrm{d}s\\
&+C\|\nabla\phi\|_{\mathbb{L}^2}\int_0^T\|\mathbf{u}'\|_{\mathbb{L}^6}\|\nabla\mathbf{u}'\|_{\mathbb{L}^6}\|\mathbf{u}_n'-\mathbf{u}'\|_{\mathbb{L}^6}\,\mathrm{d}s\\
&\leq C\|\phi-\phi_n\|_{\mathbb{H}^1}\|\mathbf{u}_n'\|_{L^{\infty}(0,T;\mathbb{H}^1)}^2\|\mathbf{u}_n'\|_{L^{2}(0,T;\mathbb{H}^2)}\\ &+C\|\phi\|_{\mathbb{H}^1}\|\mathbf{u}_n'\|_{L^{\infty}(0,T;\mathbb{H}^1)}\|\mathbf{u}_n'\|_{L^{2}(0,T;\mathbb{H}^2)}\|\mathbf{u}_n'-\mathbf{u}'\|_{L^{2}(0,T;\mathbb{H}^1)}\\
&+C\|\phi\|_{\mathbb{H}^1}\|\mathbf{u}_n'\|_{L^{\infty}(0,T;\mathbb{H}^1)}^2\|\mathbf{u}_n'-\mathbf{u}'\|_{L^{2}(0,T;\mathbb{H}^2)}\\
&+C\|\phi\|_{\mathbb{H}^1}\|\mathbf{u}_n'\|_{L^{\infty}(0,T;\mathbb{H}^1)}\|\mathbf{u}'\|_{L^{\infty}(0,T;\mathbb{H}^1)}\|\mathbf{u}_n'-\mathbf{u}'\|_{L^{2}(0,T;\mathbb{H}^2)}\\
&+C\|\phi\|_{\mathbb{H}^1}\|\mathbf{u}'\|_{L^{\infty}(0,T;\mathbb{H}^1)}\|\mathbf{u}'\|_{L^{2}(0,T;\mathbb{H}^2)}\|\mathbf{u}_n'-\mathbf{u}'\|_{L^{2}(0,T;\mathbb{H}^1)}\rightarrow0~\textrm{as}~n\rightarrow\infty.
\end{split}
\end{equation*}
Therefore \eqref{1cor5-1} is true.

To show \eqref{1lem5-7}, it is sufficient to prove that
\begin{equation*}
\begin{split}
\lim_{n\rightarrow\infty}\mathbb{E}'\left|\int_0^T\left(\mathbf{b}_n(\mathbf{u}_n')-\mathbf{b}(\mathbf{u}'),\phi\right)_{\mathbb{L}^2}\,\mathrm{d}s\right|^2=0,
\end{split}
\end{equation*}
By the H\"{o}lder inequality, it follows from Remark \ref{rem331} that
\begin{equation*}
\begin{split}
&\mathbb{E}'\left|\int_0^T\left(\mathbf{b}_n(\mathbf{u}_n')-\mathbf{b}(\mathbf{u}'),\phi\right)_{\mathbb{L}^2}\,\mathrm{d}s\right|^2\leq C\|\phi\|_{\mathbb{L}^2}^2\mathbb{E}'\int_0^T\|\mathbf{b}_n(\mathbf{u}_n')-\mathbf{b}(\mathbf{u}')\|_{\mathbb{L}^2}^2\,\mathrm{d}s\\
&\leq C\mathbb{E}'\int_0^T\|\mathbf{u}_n'-\mathbf{u}'\|_{\mathbb{L}^2}^2\,\mathrm{d}s\rightarrow 0~\textrm{as}~n\rightarrow\infty.
\end{split}
\end{equation*}

To show \eqref{2lem5-7}, by the It\^{o}-L\'{e}vy Isometry, it is sufficient to prove that
\begin{equation*}
\begin{split}
\lim_{n\rightarrow\infty}\mathbb{E}'\int_0^T\int_{\mathbb{B}}\left|\left(\mathbf{G}(l,\mathbf{u}_n')-\mathbf{G}(l,\mathbf{u}'),\phi\right)_{\mathbb{L}^2}\right|^2\nu(\mathrm{d}l)\mathrm{d}s=0.
\end{split}
\end{equation*}
Thanks to Remark \ref{rem331}, we have
\begin{equation*}
\begin{split}
&\mathbb{E}'\int_0^T\int_{\mathbb{B}}\left|\left(\mathbf{G}(l,\mathbf{u}_n')-\mathbf{G}(l,\mathbf{u}'),\phi\right)_{\mathbb{L}^2}\right|^2\nu(\mathrm{d}l)\mathrm{d}s\leq\|\phi\|_{\mathbb{L}^2}^2\mathbb{E}'\left(\int_{\mathbb{B}}C(l)\nu(\mathrm{d}l)\int_0^T\|\mathbf{u}_n'-\mathbf{u}'\|_{\mathbb{L}^2}^2\,\mathrm{d}s\right)\\
&\leq C\mathbb{E}'\int_0^T\|\mathbf{u}_n'-\mathbf{u}'\|_{\mathbb{L}^2}^2\,\mathrm{d}s\rightarrow 0~\textrm{as}~n\rightarrow\infty.
\end{split}
\end{equation*}
The proof is thus completed.
\end{proof}

Now let us define
\begin{equation*}\label{1pro55-1}
\begin{split}
&\mathbf{M}_n^1(\mathbf{u}_n',\tilde{\eta}_n',\phi_n)(t):=(\mathbf{u}'_n(0),\phi_n)_{\mathbb{L}^2}+\int_0^t\left\langle\sum_{j=1}^5F^j_n(\mathbf{u}'_n),\phi_n \right\rangle_{(\mathbb{H}^{1})^*,\mathbb{H}^1}\,\mathrm{d}s\\
&+\int_0^t\left\langle\mathbf{b}_n(\mathbf{u}_n'),\phi_n \right\rangle_{(\mathbb{H}^{1})^*,\mathbb{H}^1}\,\mathrm{d}s+\int_0^t\int_{\mathbb{B}}\left\langle\mathbf{G}_n(l,\mathbf{u}_n'),\phi_n\right\rangle_{(\mathbb{H}^{1})^*,\mathbb{H}^1}\tilde{\eta}'_n(\mathrm{d}s,\mathrm{d}l),
\end{split}
\end{equation*}
\begin{equation*}\label{2pro55-1}
\begin{split}
&\mathbf{M}^1(\mathbf{u}',\tilde{\eta}',\phi)(t):=(\mathbf{u}'(0),\phi)_{\mathbb{L}^2}+\int_0^t\left\langle\sum_{j=1}^5F^j(\mathbf{u}'),\phi \right\rangle_{(\mathbb{H}^{1})^*,\mathbb{H}^1}\,\mathrm{d}s\\
&+\int_0^t\left\langle\mathbf{b}(\mathbf{u}'),\phi \right\rangle_{(\mathbb{H}^{1})^*,\mathbb{H}^1}\,\mathrm{d}s+\int_0^t\int_{\mathbb{B}}\left\langle\mathbf{G}(l,\mathbf{u}'),\phi\right\rangle_{(\mathbb{H}^{1})^*,\mathbb{H}^1}\tilde{\eta}'(\mathrm{d}s,\mathrm{d}l).
\end{split}
\end{equation*}
Since $\mathbf{u}'_n\rightarrow\mathbf{u}'$ in $L^2(0,T;\mathbb{L}^2)$ a.s., we easily prove that
$
\int_0^T\left|\left(\mathbf{u}_n',\phi_n\right)_{\mathbb{L}^2}-\left( \mathbf{u}',\phi\right)_{\mathbb{L}^2}\right|^2\,\mathrm{d}s=0,
$ a.s. Moreover, it is easy to check that for $p\geq1$,
$
\sup_{n\in\mathbb{N}}\mathbb{E}'\left|\int_0^T\left|\left(\mathbf{u}_n',\phi_n\right)_{\mathbb{L}^2}-\left( \mathbf{u}',\phi\right)_{\mathbb{L}^2}\right|^2\,\mathrm{d}s\right|^p<\infty.
$
Thus by the Vitali theorem we infer that
\begin{equation}\label{3pro5-1}
\begin{split}
\lim_{n\rightarrow\infty}\|(\mathbf{u}_n',\phi_n)_{\mathbb{L}^2}-(\mathbf{u}',\phi)_{\mathbb{L}^2}\|_{L^2(\Omega'\times[0,T])}=0.
\end{split}
\end{equation}
Moreover thanks to Proposition \ref{cor5-1}, we have
$$
\lim_{n\rightarrow\infty}\mathbb{E}'\left|\mathbf{M}_n^1(\mathbf{u}_n',\tilde{\eta}_n',\phi_n)(t)-\mathbf{M}^1(\mathbf{u}',\tilde{\eta}',\phi)(t)\right|^2=0.
$$
And it is not hard to check that
$
\sup_{n\in{\mathbb{N}}}\mathbb{E}'\left|\mathbf{M}_n^1(\mathbf{u}_n',\tilde{\eta}_n',\phi_n)(t)-\mathbf{M}^1(\mathbf{u}',\tilde{\eta}',\phi)(t)\right|^2<\infty.
$ Thus by the Dominated convergence theorem, we obtain that
\begin{equation}\label{4pro5-1}
\begin{split}
\lim_{n\rightarrow\infty}\|\mathbf{M}_n^1(\mathbf{u}_n',\tilde{\eta}_n',\phi_n)(t)-\mathbf{M}^1(\mathbf{u}',\tilde{\eta}',\phi)(t)\|_{L^2(\Omega'\times[0,T])}=0.
\end{split}
\end{equation}
Since $\mathbf{u}_n$ is a solution of the Galerkin equation, for all $t\in[0,T]$ and $\mathbb{P}$-a.s.
\begin{equation*}
\begin{split}
(\mathbf{u}_n(t),\phi_n)_{\mathbb{L}^2}=\mathbf{M}_n^1(\mathbf{u}_n,\tilde{\eta}_n,\phi_n)(t).
\end{split}
\end{equation*}
Moreover, since $\mathscr{L}((\mathbf{u}_n,\eta_n))=\mathscr{L}((\mathbf{u}_{n}',\eta_{n}'))$ for all $n\in\mathbb{N}$,
\begin{equation*}
\begin{split}
\int_0^T\mathbb{E}'\left|(\mathbf{u}_n'(t),\phi_n)_{\mathbb{L}^2}-\mathbf{M}_n^1(\mathbf{u}_n',\tilde{\eta}_n',\phi_n)(t)\right|^2\,\mathrm{d}s=\int_0^T\mathbb{E}\left|(\mathbf{u}_n(t),\phi_n)_{\mathbb{L}^2}-\mathbf{M}_n^1(\mathbf{u}_n,\tilde{\eta}_n,\phi_n)(t)\right|^2\,\mathrm{d}s=0
\end{split}
\end{equation*}
Thus by \eqref{3pro5-1} and \eqref{4pro5-1}, we have
\begin{equation*}
\begin{split}
\int_0^T\mathbb{E}'\left|(\mathbf{u}'(t),\phi)_{\mathbb{L}^2}-\mathbf{M}^1(\mathbf{u}',\tilde{\eta}',\phi)(t)\right|^2\,\mathrm{d}s=0,
\end{split}
\end{equation*}
which means that for leb-almost all $t\in[0,T]$ and $\mathbb{P}'$-a.s.
$
\left(\mathbf{u}'(t),\phi \right)_{\mathbb{L}^2}-\mathbf{M}^1(\mathbf{u}',\tilde{\eta}',\phi)(t)=0.
$ Since $\mathbf{u}'$ is a $\mathcal{Z}_T^1$-valued random variable, in particular $\mathbf{u}'\in\mathbb{D}([0,T];\mathbb{H}^1_w)$, i.e. $\mathbf{u}'$ is weakly c\`{a}dl\`{a}g. Moreover, since two c\`{a}dl\`{a}g functions equal for leb-almost all $t\in[0,T]$ must be equal for all $t\in[0,T]$, we derive that for all $\phi\in \mathbb{H}^1$ and $t\in[0,T]$, the equality
\begin{equation}\label{5pro5-1}
\begin{split}
\left(\mathbf{u}',\phi \right)_{\mathbb{L}^2}=\mathbf{M}^1(\mathbf{u}',\tilde{\eta}',\phi)(t)
\end{split}
\end{equation}
is valid, $\mathbb{P}'$-a.s. Therefore $(\Omega',\mathcal{F}',\mathbb{F}',\mathbb{P}',\mathbf{u}',\eta')$ is a martingale weak solution of \eqref{sys1}.

The pathwise uniqueness result comes from the following proposition.
\begin{proposition}\label{pro5-2} Let $\mathcal{O}\subset\mathbb{R}^d,~d=1,2,3$ be a bounded domain with $C^{2,1}$-boundary and let $\mathbf{u}_0\in \mathbb{H}^1$ be fixed. Assume that $(\Omega,\mathcal{F},\mathbb{F},\mathbb{P},\mathbf{u}_1,\eta')$ and $(\Omega,\mathcal{F},\mathbb{F},\mathbb{P},\mathbf{u}_2,\eta')$ are two martingale weak solution of \eqref{sys1} such that for $i=1,2$,
\begin{equation*}
\begin{split}
\mathbf{u}_i(0)=\mathbf{u}(0);~\mathbf{u}_i\in L^{\infty}(0,T;\mathbb{H}^1)\cap L^2(0,T;\mathbb{H}^3)\textrm{~a.s.};~\mathbf{u}_i~\textrm{satisfies the equation}~\eqref{5pro5-1}.
\end{split}
\end{equation*}
Then
$
\mathbf{u}_1(\cdot,\omega)=\mathbf{u}_2(\cdot,\omega)
$, $\mathbb{P}$-a.s.
\end{proposition}
\begin{proof}[\emph{\textbf{Proof}}] Let $\mathbf{u}^*:=\mathbf{u}_1-\mathbf{u}_2$. Then $\mathbf{u}^*$ satisfies the following equation
\begin{equation*}
\begin{split}
\mathbf{u}^*(t)=\sum_{j=1}^5\int_0^tF^j(\mathbf{u}_1)-F^j(\mathbf{u}_2)\,\mathrm{d}s+\int_0^t\mathbf{b}(\mathbf{u}_1)-\mathbf{b}(\mathbf{u}_2)\,\mathrm{d}s+\int_0^t\int_{\mathbb{B}}\mathbf{G}(l,\mathbf{u}_1)-\mathbf{G}(l,\mathbf{u}_2)\tilde{\eta}(\mathrm{d}s,\mathrm{d}l)
\end{split}
\end{equation*}
in $(\mathbb{H}^{1})^*$ with $\mathbf{u}^*_0=0$. Let
\begin{equation*}
\begin{split}
\xi^K:=\inf\left\{t\geq0:\|\mathbf{u}_1(t)\|_{\mathbb{H}^1}^2+\|\mathbf{u}_2(t)\|_{\mathbb{H}^1}^2+\int_0^t\|\mathbf{u}_1\|_{\mathbb{H}^3}^2\,\mathrm{d}s+\int_0^t\|\mathbf{u}_2\|_{\mathbb{H}^3}^2\,\mathrm{d}s>K\right\}\wedge T,~K>0.
\end{split}
\end{equation*}
According to \eqref{1lem5-5}, it follows that
$
\xi^K\nearrow T~\textrm{as}~K\rightarrow\infty,~\mathbb{P}\textrm{-a.s.}
$ Since $\mathbb{H}^1\subset \mathbb{L}^2\subset (\mathbb{H}^{1})^*$ is a Gelfand triple for Hilbert space, we can use the It\^{o} formula (cf. Gy\"{o}ngi and Krylov \cite{gyongy1982stochastics}) to $\|\mathbf{u}^*\|_{\mathbb{L}^2}^2$. Thus we have
\begin{equation}\label{2pro5-2}
\begin{split}
&\|\mathbf{u}^*(t\wedge\xi^K)\|_{\mathbb{L}^2}^2+2\int_0^{t\wedge\xi^K}\|\Delta\mathbf{u}^*\|_{\mathbb{L}^2}^2\,\mathrm{d}s\\
&=2\int_0^{t\wedge\xi^K}\|\nabla\mathbf{u}^*\|_{\mathbb{L}^2}^2\,\mathrm{d}s+4\int_0^{t\wedge\xi^K}\left(\mathbf{u}^*-|\mathbf{u}_1|^2\mathbf{u}^*+\mathbf{u}_2\left(|\mathbf{u}_2|^2-|\mathbf{u}_1|^2\right),\mathbf{u}^*\right)_{\mathbb{L}^2}\,\mathrm{d}s\\
&+4\int_0^{t\wedge\xi^K}\left(|\mathbf{u}_1|^2\mathbf{u}_1-|\mathbf{u}_2|^2\mathbf{u}_2,\Delta\mathbf{u}^*\right)_{\mathbb{L}^2}\,\mathrm{d}s+2\int_0^{t\wedge\xi^K}\left(\mathbf{u}_1\times\nabla\mathbf{u}_1-\mathbf{u}_2\times\nabla\mathbf{u}_2,\nabla\mathbf{u}^*\right)_{\mathbb{L}^2}\,\mathrm{d}s\\
&+2\int_0^{t\wedge\xi^K}(\mathbf{b}(\mathbf{u}_1)-\mathbf{b}(\mathbf{u}_2),\mathbf{u}^*)_{\mathbb{L}^2}\,\mathrm{d}s+2\int_0^{t\wedge\xi^K}\int_{\mathbb{B}}\left(\mathbf{G}(l,\mathbf{u}_1)-\mathbf{G}(l,\mathbf{u}_2),\mathbf{u}^*\right)_{\mathbb{L}^2}\tilde{\eta}(\mathrm{d}s,\mathrm{d}l)\\
&+\int_0^{t\wedge\xi^K}\int_{\mathbb{B}}\|\mathbf{G}(l,\mathbf{u}_1)-\mathbf{G}(l,\mathbf{u}_2)\|_{L^2}^2\tilde{\eta}(\mathrm{d}s,\mathrm{d}l)+\int_0^{t\wedge\xi^K}\int_{\mathbb{B}}\|\mathbf{G}(l,\mathbf{u}_1)-\mathbf{G}(l,\mathbf{u}_2)\|_{L^2}^2\nu(\mathrm{d}l)\mathrm{d}s\\
&:=\sum_{j=1}^8I_j(t\wedge\xi^K).
\end{split}
\end{equation}
By using integration by parts, H\"{o}lder's inequality and Young's inequality, we have
\begin{equation}\label{3pro5-2}
\begin{split}
|I_1(t\wedge\xi^K)|\leq\varepsilon\int_0^{t\wedge\xi^K}\|\Delta\mathbf{u}^*\|_{\mathbb{L}^2}^2\,\mathrm{d}s+C_{\varepsilon}\int_0^{t\wedge\xi^K}\|\mathbf{u}^*\|_{\mathbb{L}^2}^2\,\mathrm{d}s.
\end{split}
\end{equation}
By using the triangle inequality, we see that
\begin{equation}\label{4pro5-2}
\begin{split}
I_2(t\wedge\xi^K)\leq C\int_0^{t\wedge\xi^K}\|\mathbf{u}^*\|_{\mathbb{L}^2}^2\,\mathrm{d}s+C\int_0^{t\wedge\xi^K}\|\mathbf{u}_2\|_{\mathbb{L}^{\infty}}\left(\|\mathbf{u}_1\|_{\mathbb{L}^{\infty}}+\|\mathbf{u}_2\|_{\mathbb{L}^{\infty}}\right)\|\mathbf{u}^*\|_{\mathbb{L}^2}^2\,\mathrm{d}s.
\end{split}
\end{equation}
Similarly, by the H\"{o}lder inequality and Young's inequality, it follows that
\begin{equation}\label{5pro5-2}
\begin{split}
&|I_3(t\wedge\xi^K)|\leq C\left|\int_0^{t\wedge\xi^K}\left(|\mathbf{u}_1|^2\mathbf{u}^*+\left(|\mathbf{u}_1|^2-|\mathbf{u}_2|^2\right)\mathbf{u}_2,\Delta\mathbf{u}^*\right)_{\mathbb{L}^2}\,\mathrm{d}s\right|\\
&\leq C\int_0^{t\wedge\xi^K}\|\mathbf{u}_1\|_{\mathbb{L}^{\infty}}^2\|\mathbf{u}^*\|_{\mathbb{L}^2}\|\Delta\mathbf{u}^*\|_{\mathbb{L}^2}\,\mathrm{d}s\\
&+C\int_0^{t\wedge\xi^K}(\|\mathbf{u}_1\|_{\mathbb{L}^{\infty}}\|\mathbf{u}_2\|_{\mathbb{L}^{\infty}}+\|\mathbf{u}_2\|_{\mathbb{L}^{\infty}}^2)\|\mathbf{u}^*\|_{\mathbb{L}^2}\|\Delta\mathbf{u}^*\|_{\mathbb{L}^2}\,\mathrm{d}s\\
&\leq \varepsilon\int_0^{t\wedge\xi^K}\|\Delta\mathbf{u}^*\|_{\mathbb{L}^{2}}^{2}\,\mathrm{d}s+C_{\varepsilon}\int_0^{t\wedge\xi^K}(\|\mathbf{u}_1\|_{\mathbb{L}^{\infty}}^4+\|\mathbf{u}_2\|_{\mathbb{L}^{\infty}}^4)\|\mathbf{u}^*\|_{\mathbb{L}^2}^2\,\mathrm{d}s.
\end{split}
\end{equation}
For the fourth term, it follows that
\begin{equation}\label{6pro5-2}
\begin{split}
&|I_4(t\wedge\xi^K)|=2\left|\int_0^{t\wedge\xi^K}\left(\nabla\mathbf{u}_1\times\mathbf{u}^*,\nabla\mathbf{u}^*\right)_{\mathbb{L}^2}\,\mathrm{d}s\right|= 2\left|\int_0^{t\wedge\xi^K}\left(\mathbf{u}_1\times\mathbf{u}^*,\Delta\mathbf{u}^*\right)_{\mathbb{L}^2}\,\mathrm{d}s\right|\\
&\leq C\int_0^{t\wedge\xi^K}\|\mathbf{u}_1\|_{\mathbb{L}^{\infty}}\|\mathbf{u}^*\|_{\mathbb{L}^{2}}\|\Delta\mathbf{u}^*\|_{\mathbb{L}^{2}}\,\mathrm{d}s\\
&\leq\varepsilon\int_0^{t\wedge\xi^K}\|\Delta\mathbf{u}^*\|_{\mathbb{L}^{2}}^{2}\,\mathrm{d}s+C_{\varepsilon}\int_0^{t\wedge\xi^K}\|\mathbf{u}_1\|_{\mathbb{L}^{\infty}}^2\|\mathbf{u}^*\|_{\mathbb{L}^{2}}^2\,\mathrm{d}s.
\end{split}
\end{equation}
Moreover by Corollary \ref{--cor2-2-2}, we infer that
\begin{equation}\label{7pro5-2}
\begin{split}
&|I_5+I_8|(t\wedge\xi^K)\leq C\int_0^{t\wedge\xi^K}\|\mathbf{u}^*\|_{\mathbb{L}^2}^2\,\mathrm{d}s.
\end{split}
\end{equation}
Thus plugging \eqref{3pro5-2}-\eqref{7pro5-2} into \eqref{2pro5-2} and choosing $\varepsilon$ small enough, we infer that
\begin{equation*}
\begin{split}
\|\mathbf{u}^*(t\wedge\xi^K)\|_{\mathbb{L}^2}^2\leq C\int_0^{t\wedge\xi^K}\mathbf{F}(s)\|\mathbf{u}^*(s)\|_{\mathbb{L}^{2}}^2\,\mathrm{d}s+I_6(t\wedge\xi^K)+I_7(t\wedge\xi^K),
\end{split}
\end{equation*}
where
$
\mathbf{F}:=1+\|\mathbf{u}_1\|_{\mathbb{L}^{\infty}}^4+\|\mathbf{u}_2\|_{\mathbb{L}^{\infty}}^4.
$
According to the GN inequality,
\begin{equation*}
\begin{split}
\|f\|_{\mathbb{L}^{\infty}}\leq C\|f\|_{\mathbb{H}^{1}}^{\frac{1}{2}}\|f\|_{\mathbb{H}^{2}}^{\frac{1}{2}},~f\in\mathbb{H}^2,~d=1,2,3.
\end{split}
\end{equation*}
Thus
\begin{equation*}
\begin{split}
\int_0^{t\wedge\xi^K}\mathbf{F}(s)\,\mathrm{d}s&\leq t+C\sup_{s\in[0,t\wedge\xi^K]}\|\mathbf{u}_1(s)\|_{\mathbb{H}^{1}}^2\int_0^{t\wedge\xi^K}\|\mathbf{u}_1(s)\|_{\mathbb{H}^{2}}^2\,\mathrm{d}s\\
&+C\sup_{s\in[0,t\wedge\xi^K]}\|\mathbf{u}_2(s)\|_{\mathbb{H}^{1}}^2\int_0^{t\wedge\xi^K}\|\mathbf{u}_2(s)\|_{\mathbb{H}^{2}}^2\,\mathrm{d}s\leq C_k.
\end{split}
\end{equation*}
Thus by using the Gronwall lemma and BDG inequality, we have
\begin{equation*}
\begin{split}
&\mathbb{E}\sup_{s\in[0,t]}\|\mathbf{u}^*(s\wedge\xi^K)\|_{\mathbb{L}^2}^2\leq C\mathbb{E}\left[\exp\left(\int_0^{t\wedge\xi^K}\mathbf{F}(s)\,\mathrm{d}s\right)\left(\sup_{s\in[0,t]}|I_6(s\wedge\xi^K)|+\sup_{s\in[0,t]}|I_7(s\wedge\xi^K)|\right)\right]\\
&\leq C_K\mathbb{E}\left(\int_0^{t\wedge\xi^K}\|\mathbf{u}^*(s)\|_{\mathbb{L}^{2}}^4\,\mathrm{d}s\right)^{\frac{1}{2}}\leq\frac{1}{2}\mathbb{E}\sup_{s\in[0,t]}\|\mathbf{u}^*(s\wedge\xi^K)\|_{\mathbb{L}^2}^2+C_K\mathbb{E}\int_0^{t\wedge\xi^K}\|\mathbf{u}^*(s)\|_{\mathbb{L}^{2}}^2\,\mathrm{d}s.
\end{split}
\end{equation*}
Using the Gronwall lemma again we infer that $\sup_{s\in[0,t\wedge\xi^K]}\|\mathbf{u}^*(s)\|_{\mathbb{L}^2}^2=0$, $\mathbb{P}$-a.s. By the monotone convergence theorem and the fact that $\xi^K\nearrow T$ as $K\rightarrow\infty$, it follows that $\mathbb{P}$-a.s.,
$
\sup_{s\in[0,T]}\|\mathbf{u}^*(s)\|_{\mathbb{L}^2}^2=0,
$
which implies the uniqueness.\end{proof}
Theorem \ref{the1} then follows from the the Yamada-Watanabe theorem.\end{proof}

\section{Large deviations principle}\label{sec6}
In this section, we shall establish a Freidlin-Wentzell type LDP for pathwise weak solutions of the SLLBar equation \eqref{sys1}.

Let $\varepsilon>0$ be fixed. Define a time scaling of the L\'evy process by
$
L^{\varepsilon^{-1}}(t):=L(\varepsilon^{-1}t),~t\geq0.
$
Let $\eta^{\varepsilon^{-1}}$ be the time-homogeneous Poisson random measure (PRM) of the L\'evy process $L^{\varepsilon^{-1}}$ and $\nu^{\varepsilon^{-1}}$ be its compensator. Then we have
\begin{equation*}
\begin{split}
\eta^{\varepsilon^{-1}}([0,t]\times\Gamma)&:=\#\left\{s\in[0,t]:L^{\varepsilon^{-1}}(s)-L^{\varepsilon^{-1}}(s-)\in\Gamma\right\}\\
&=\#\left\{s\in[0,\varepsilon^{-1}t]:L(s)-L(s-)\in\Gamma\right\}=\eta([0,\varepsilon^{-1}t]\times\Gamma),
\end{split}
\end{equation*}
which implies that
$
t\nu^{\varepsilon^{-1}}(\Gamma)=\mathbb{E}[\eta^{\varepsilon^{-1}}([0,t]\times\Gamma)]=\mathbb{E}[\eta([0,\varepsilon^{-1}t]\times\Gamma)]=\varepsilon^{-1}t\nu(\Gamma),
$
and
$
\nu^{\varepsilon^{-1}}(\Gamma)=\varepsilon^{-1}\nu(\Gamma).
$
If we denote by $\tilde{\eta}^{\varepsilon^{-1}}$  the time-homogeneous compensated PRM of the L\'evy process $L^{\varepsilon^{-1}}$, then
$
\tilde{\eta}^{\varepsilon^{-1}}([0,t]\times\Gamma)=\eta^{\varepsilon^{-1}}([0,t]\times\Gamma)-\varepsilon^{-1}t\nu(\Gamma),
$
and so
\begin{equation}\label{63}
\begin{split}
\tilde{\eta}^{\varepsilon^{-1}}(\mathrm{d}t,\mathrm{d}l)=\eta^{\varepsilon^{-1}}(\mathrm{d}t,\mathrm{d}l)-\varepsilon^{-1}\nu(\mathrm{d}l)\mathrm{d}t.
\end{split}
\end{equation}

By using the representation \eqref{63}, equation \eqref{sys61} can be rewritten in the   form of
\begin{equation}\label{sys62}
\begin{split}
\mathbf{u}^{\varepsilon}(t)&=\mathbf{u}_0^{\varepsilon}+\int_0^t-\Delta\mathbf{u}^{\varepsilon}-\Delta^2\mathbf{u}^{\varepsilon}+2(1-|\mathbf{u}^{\varepsilon}|^2)\mathbf{u}^{\varepsilon}+2\Delta(|\mathbf{u}^{\varepsilon}|^2\mathbf{u}^{\varepsilon})-\mathbf{u}^{\varepsilon}\times\Delta \mathbf{u}^{\varepsilon}\,\mathrm{d}s\\
&+\varepsilon\int_0^t\Phi(l,\mathbf{u}^{\varepsilon})-\mathbf{u}^{\varepsilon}\tilde{\eta}^{\varepsilon^{-1}}(\mathrm{d}s,\mathrm{d}l)+\varepsilon\int_0^t\int_{\mathbb{B}}\Phi(l,\mathbf{u}^{\varepsilon})-\mathbf{u}^{\varepsilon}-lJ(\mathbf{u}^{\varepsilon})\nu^{\varepsilon^{-1}}(\mathrm{d}l)\mathrm{d}s\\
&:=\mathbf{u}_0^{\varepsilon}+\sum_{j=1}^5\int_0^tF^j(\mathbf{u}^{\varepsilon})\,\mathrm{d}s+\varepsilon\int_0^t\mathbf{G}(l,\mathbf{u}^{\varepsilon})\tilde{\eta}^{\varepsilon^{-1}}(\mathrm{d}s,\mathrm{d}l)+\int_0^t\mathbf{b}(\mathbf{u}^{\varepsilon})\nu(\mathrm{d}l)\mathrm{d}s.
\end{split}
\end{equation}

Similar to the proof of Theorem \ref{the1}, we have the following result.
\begin{lemma}\label{the61}
Let $\mathcal{O}\subset\mathbb{R}^d$, $d=1,2,3$, be a bounded domain with $C^{2,1}$-boundary.   Assume that $\mathbf{u}_0^{\varepsilon}\in\mathbb{H}^1$, the functions $
\mathbf{h}\in\mathbb{W}^{1,\infty}$ and $\mathbf{g}\in\mathbb{H}^1$. Then there exists a unique pathwise weak solution $\mathbf{u}^{\varepsilon}(t)$ to \eqref{sys62}.
\end{lemma}

\begin{remark}\label{the62} The uniqueness in law and the existence of a weak solution hold for equation \eqref{sys62} in the following sense:
\begin{enumerate}
\item [(1)] If $(\Omega,\mathcal{F},\mathbb{F},\mathbb{P},\mathbf{u}_1^{\varepsilon},\eta^{\varepsilon^{-1}})$ and $(\Omega',\mathcal{F}',\mathbb{F}',\mathbb{P}',\mathbf{u}_2^{\varepsilon},(\eta^{\varepsilon^{-1}})')$ are two martingale weak solutions to the problem \eqref{sys62}, such that both $\mathbf{u}_1^{\varepsilon}$ and $\mathbf{u}_2^{\varepsilon}$ are $\mathcal{Z}_T$-valued random variables, then $\mathbf{u}_1^{\varepsilon}$ and $\mathbf{u}_2^{\varepsilon}$ have the same laws on $\mathcal{Z}_T$.
\item[(2)] For every $\varepsilon>0$ there exists a Borel measurable function $\mathcal{J}^{\varepsilon}:\bar{\mathbb{M}}_T\rightarrow \mathcal{Z}_T$ (see subsection \ref{subsec6-1} for the definition of $\bar{\mathbb{M}}_T$) such that the following statement holds: If $(\Omega,\mathcal{F},\mathbb{F},\mathbb{P})$ is an arbitrary filtered probability space, $\eta$ is a arbitrary $\mathbb{R}$-valued time homogeneous PRM defined on $(\Omega,\mathcal{F},\mathbb{F},\mathbb{P})$, and $
X^{\varepsilon}:\Omega\ni\omega\mapsto\mathcal{J}^{\varepsilon}(\varepsilon\eta^{\varepsilon^{-1}}(\omega))\in\mathcal{Z}_T$, then $(\Omega,\mathcal{F},\mathbb{F},\mathbb{P},X^{\varepsilon},\eta^{\varepsilon^{-1}})$ is a martingale weak solution to the problem \eqref{sys62}.
\end{enumerate}
\end{remark}

Before proving the second main result of this paper, we first provide some necessary definitions and framework.

\subsection{The framework}\label{subsec6-1}
Let $\mathbb{B}_T:=[0,T]\times\mathbb{B}$, $\mathbb{X}:=\mathbb{B}\times[0,\infty)$ and $\mathbb{X}_T:=[0,T]\times\mathbb{B}\times[0,\infty)$. Let $\mathbb{M}_T:=\mathcal{M}(\mathbb{B}_T)$ be the space of all nonnegative measures $\vartheta$ on $(\mathbb{B}_T,\mathcal{B}(\mathbb{B}_T))$ such that $\vartheta(K)<\infty$ for every compact set $K$ of $\mathbb{B}_T$. We endow the set $\mathbb{M}_T$ with the weakest topology, denoted by $\mathcal{T}(\mathbb{M}_T)$, such that for every $g\in C_c(\mathbb{B}_T)$, the map
\begin{equation*}
\begin{split}
\vartheta\mapsto(g,\vartheta):=\int_{\mathbb{B}_T}g(l,s)\vartheta(\mathrm{d}l,\mathrm{d}s)\in \mathbb{R}
\end{split}
\end{equation*}
is continuous. This topology can be metrized such that $\mathbb{M}_T$ is a Polish space \cite[section 2]{budhiraja2011variational}. Analogously we define $\bar{\mathbb{M}}_T:=\mathcal{M}(\mathbb{X}_T)$ and $\mathcal{T}(\bar{\mathbb{M}}_T)$. Then there exists a unique probability measure $\bar{\mathbb{P}}$ on $(\bar{\mathbb{M}}_T,\mathcal{T}(\bar{\mathbb{M}}_T))$ \cite[section I.8]{ikeda2014stochastic}, under which the canonical map $\bar{\eta}(\bar{m})=\bar{m}$ is a PRM with intensity measure $\nu(\mathrm{d}l)\mathrm{d}t\mathrm{d}r$. The corresponding compensated PRM is denoted by $\tilde{\bar{\eta}}$ and is defined by
$
\tilde{\bar{\eta}}(\mathrm{d}t\mathrm{d}l\mathrm{d}r):=\bar{\eta}(\mathrm{d}t\mathrm{d}l\mathrm{d}r)-\nu(\mathrm{d}l)\mathrm{d}t\mathrm{d}r.
$
Denote $\mathcal{F}_t:=\sigma\{\bar{\eta}((0,s]\times D):s\in[0,t],~D\in \mathcal{B}(\mathbb{X})\}$ and let $\bar{\mathcal{F}}_t$ the completion under $\bar{\mathbb{P}}$. Let $\bar{\mathbb{F}}:=(\bar{\mathcal{F}}_t)_{t\in[0,T]}$. Set $\bar{\mathcal{P}}$ be the $\bar{\mathbb{F}}$-predictable $\sigma$-field on $[0,T]\times\bar{\mathbb{M}}_T$, with the filtration $\{\bar{\mathcal{F}}_t:t\in[0,T]\}$ on $(\bar{\mathbb{M}}_T,\mathcal{B}(\bar{\mathbb{M}}_T))$. Let $\bar{\mathcal{A}}$ be the class of all $(\bar{\mathcal{P}}\otimes\mathcal{B}(\mathbb{B}))\backslash\mathcal{B}[0,\infty)$-measurable maps $\varphi:\mathbb{B}_T\times \bar{\mathbb{M}}_T\rightarrow[0,\infty)$. For $\varphi\in\bar{\mathcal{A}}$, define counting process $\eta_c^{\varphi}$ on $\mathbb{B}$ by
\begin{equation*}\label{subsec6-1-1}
\begin{split}
\eta_c^{\varphi}((0,t]\times D):=\int_{(0,t]\times D\times(0,\infty)}1_{[0,\varphi(s,l)]}(r)\bar{\eta}(\mathrm{d}s\mathrm{d}l\mathrm{d}r),~t\in[0,T],~D\in \mathcal{B}(\mathbb{B}).
\end{split}
\end{equation*}
Similarly, we define
\begin{equation*}\label{subsec6-1-2}
\begin{split}
\tilde{\eta}_c^{\varphi}((0,t]\times D):=\int_{(0,t]\times D\times(0,\infty)}1_{[0,\varphi(s,l)]}(r)\tilde{\bar{\eta}}(\mathrm{d}s\mathrm{d}l\mathrm{d}r),~t\in[0,T],~D\in \mathcal{B}(\mathbb{B}).
\end{split}
\end{equation*}
Clearly, $\eta^{\varphi}$ is the controlled PRM, and
$
\tilde{\eta}_c^{\varphi}((0,t]\times D)=\eta_c^{\varphi}((0,t]\times D)-\int_{(0,t]\times D}\varphi(s,l)\nu(\mathrm{d}l)\mathrm{d}s.
$
For $K\in \mathbb{N}$, we denote
$
S^K:=\{\theta:\mathbb{B}_T\rightarrow[0,\infty):\mathcal{L}_T\leq K\},
$
where
\begin{equation}\label{subsec6-1-5}
\begin{split}
\mathcal{L}_T(\theta):=\int_0^T\int_{\mathbb{B}}\theta(t,l)\log\theta(t,l)-\theta(t,l)+1\nu(\mathrm{d}l)\mathrm{d}t.
\end{split}
\end{equation}
A function $\theta\in S^K$ can be identified with  a measure $\nu^{\theta}\in \mathbb{M}_T$, defined by
\begin{equation*}
\begin{split}
\nu^{\theta}(D_T):=\int_{D_T}\theta(t,l)\nu(\mathrm{d}l)\mathrm{d}t,~D_T\in \mathcal{B}(\mathbb{B}_T).
\end{split}
\end{equation*}
That is,
$
\frac{\nu^{\theta}(\mathrm{d}l,\mathrm{d}t)}{\nu(\mathrm{d}l)\mathrm{d}t}=\theta.
$
This identification induces a topology on $S^K$, under which $S^K$ is a compact space (see Appendix of \cite{budhiraja2013large}). Let us denote
\begin{equation}\label{ccc1}
\begin{split}
\mathbb{S}=\bigcup_{K\in\mathbb{N}} S^K.
\end{split}
\end{equation}

\subsection{Two auxiliary equations}\label{subsec6-2}
This subsection is devoted to introduce two important equations associated with \eqref{sys61} which will play major role in proving the sufficient conditions for the LDP.

\subsubsection{Deterministic control equation}\label{subsubsec6-1}
For any $\theta\in \mathbb{S}$, we consider the following skeleton equation:
\begin{equation} \label{sys61-1}
\left\{
\begin{aligned}
\mathrm{d}\mathbf{u}^{\theta}(t)&=\left[-\Delta\mathbf{u}^{\theta}-\Delta^2\mathbf{u}^{\theta}+2(1-|\mathbf{u}^{\theta}|^2)\mathbf{u}^{\theta}+2\Delta(|\mathbf{u}^{\theta}|^2\mathbf{u}^{\theta})-\mathbf{u}^{\theta}\times\Delta \mathbf{u}^{\theta}\right]\,\mathrm{d}t\\
&+\mathbf{b}(\mathbf{u}^{\theta})+\int_{\mathbb{B}}\mathbf{G}(l,\mathbf{u}^{\theta})(\theta(t,l)-1)\nu(\mathrm{d}l)\mathrm{d}t,~t\in[0,T],\\
\mathbf{u}^{\theta}(0)&=\mathbf{u}_0\in \mathbb{H}^1.
\end{aligned}
\right.
\end{equation}

The following result determines the solvability of equation \eqref{sys61-1}.
\begin{lemma}\label{the61-1}
Let $\mathcal{O}\subset\mathbb{R}^d$, $d=1,2,3$, be a bounded domain with $C^{2,1}$-boundary. Let $\theta\in\mathbb{S}$ and $\mathbf{u}_0\in\mathbb{H}^1$. Then the equation \eqref{sys61-1} admits a unique weak solution
$
\mathbf{u}^{\theta}\in C([0,T];\mathbb{H}^1)\cap L^2(0,T;\mathbb{H}^3)
$
such that for all $\phi\in\mathbb{H}^1$
\begin{equation}\label{1the61-1}
\begin{split}
(\mathbf{u}^{\theta}(t),\phi)_{\mathbb{L}^2}=&(\mathbf{u}_0^{\theta},\phi)_{\mathbb{L}^2}+\int_0^{t}\left(\nabla\mathbf{u}^{\theta},\nabla\phi\right)_{\mathbb{L}^2}\,\mathrm{d}s+\int_0^{t}\left(\nabla\Delta\mathbf{u}^{\theta},\nabla\phi\right)_{\mathbb{L}^2}\,\mathrm{d}s\\
&+2\int_0^{t}\left((1-|\mathbf{u}^{\theta}|^2)\mathbf{u}^{\theta},\phi\right)_{\mathbb{L}^2}\,\mathrm{d}s+\int_0^{t}\left(\mathbf{u}^{\theta}\times\nabla\mathbf{u}^{\theta},\nabla\phi\right)_{\mathbb{L}^2}\,\mathrm{d}s\\
&-2\int_0^{t}\left(\nabla(|\mathbf{u}^{\theta}|^2\mathbf{u}^{\theta}),\nabla\phi\right)_{\mathbb{L}^2}\,\mathrm{d}s+\int_0^{t}\left(\mathbf{b}(\mathbf{u}^{\theta}),\phi\right)_{\mathbb{L}^2}\,\mathrm{d}s\\
&+\int_0^t\int_{\mathbb{B}}\left(\mathbf{G}(l,\mathbf{u}^{\theta})(\theta(t,l)-1),\phi\right)_{\mathbb{L}^2}\nu(\mathrm{d}l)\mathrm{d}s.
\end{split}
\end{equation}
Moreover, for every $K\in \mathbb{N}$, there exists $C_K>0$ such that
\begin{equation*}\label{2the61-1}
\begin{split}
\sup_{\theta\in S^K}\left(\sup_{t\in[0,T]}\|\mathbf{u}^{\theta}\|_{\mathbb{H}^1}^2+\int_0^T\|\mathbf{u}^{\theta}\|_{\mathbb{H}^3}^2\,\mathrm{d}s\right)\leq C_K.
\end{split}
\end{equation*}
\end{lemma}

\begin{remark}\label{rem6-1} For $\theta\in \mathbb{S}$, $\mathcal{J}^0(\theta):=\mathbf{u}^{\theta}$ denotes the unique solution to \eqref{sys61-1}. Thus Lemma \ref{the61-1} shows that the map $\mathcal{J}^0$ is well defined.
\end{remark}

The proof of Lemma \ref{the61-1} utilizes the classical Faedo-Galerkin approximation scheme combined with compactness methods. Let $\{\mathbf{e}_i\}_{i=1}^{\infty}$ denote an orthonormal basis of $\mathbb{L}^2$ consisting of eigenvectors for the Neumann Laplacian $A=-\Delta$. Let $S_n:=\textrm{span}\{\mathbf{e}_1...,\mathbf{e}_n\}$ and $\Pi_n:\mathbb{L}^2\rightarrow S_n$ be the orthogonal projection. Let us consider the following Galerkin approximation of \eqref{sys61-1}:
\begin{equation}\label{app1}
\begin{split}
\mathrm{d}\mathbf{u}^{\theta}_n(t)&=\Pi_n\left[-\Delta\mathbf{u}^{\theta}_n-\Delta^2\mathbf{u}^{\theta}_n+2(1-|\mathbf{u}^{\theta}_n|^2)\mathbf{u}^{\theta}_n+2\Delta(|\mathbf{u}^{\theta}_n|^2\mathbf{u}^{\theta}_n)-\mathbf{u}^{\theta}_n\times\Delta \mathbf{u}^{\theta}_n\right]\,\mathrm{d}t\\
&+\mathbf{b}_n(\mathbf{u}_n^{\theta})\,\mathrm{d}t+\int_{\mathbb{B}}\mathbf{G}_n(l,\mathbf{u}_n^{\theta})(\theta(t,l)-1)\nu(\mathrm{d}l)\mathrm{d}t,~t\in[0,T],
\end{split}
\end{equation}
with $\mathbf{u}^{\theta}_n(0)=\Pi_n\mathbf{u}_0$. We use the notations $\sum_{j=1}^5F_n^j$ defined in \eqref{nota1}, then the equation \eqref{app1} can be written in the integral form as
\begin{equation}\label{app2}
\begin{split}
\mathbf{u}^{\theta}_n(t)&=\mathbf{u}^{\theta}_n(0)+\sum_{j=1}^5\int_0^tF_n^j(\mathbf{u}^{\theta}_n)\,\mathrm{d}s+\int_0^t\mathbf{b}_n(\mathbf{u}^{\theta}_n)\,\mathrm{d}s+\int_0^t\int_{\mathbb{B}}\mathbf{G}_n(l,\mathbf{u}_n^{\theta})(\theta(s,l)-1)\nu(\mathrm{d}l)\mathrm{d}s.
\end{split}
\end{equation}
It is easy to check that $F_n^1$ and $F_n^2$ are globally Lipschitz and $F_n^3$-$F_n^5$ are locally Lipschitz. Additionally, it is easy to observe that the mapping
\begin{equation*}
\begin{split}
&F_n^6:S_n\ni f\mapsto\mathbf{b}_n(f)\in S_n,\\
&F_n^7:S_n\ni f\mapsto\int_{\mathbb{B}}\mathbf{G}_n(l,f)(\theta(\cdot,l)-1)\nu(\mathrm{d}l)\in S_n
\end{split}
\end{equation*}
are Lipschitz continuous. Thus the problem \eqref{app2} admits a unique solution in $S_n$. Next we establish some uniform energy estimates for the approximate solutions. For convenience, we denote the solution $\mathbf{u}^{\theta}_n$ by $\mathbf{u}_n$ in the proof of the following two lemmas.

\begin{lemma}\label{lemapp-1} Let $T>0$ and $\mathbf{u}_n$ be the solution of \eqref{app2}. Then under the same assumption as of Lemma \ref{the61-1}, there exists a positive constant $C=C(\|\mathbf{u}_0\|_{\mathbb{H}^1},\theta,\mathbf{h},\mathbf{g},T)$ independent of $n$ such that for any $n\in\mathbb{N}$,
\begin{equation*}\label{lemapp-1-1}
\begin{split}
\sup_{t\in[0,T]}\|\mathbf{u}_n(t)\|_{\mathbb{H}^1}^2+\int_0^T\|\mathbf{u}_n(t)\|_{\mathbb{H}^3}^2\,\mathrm{d}t\leq C.
\end{split}
\end{equation*}
\end{lemma}
\begin{proof}[\emph{\textbf{Proof}}]  By directly calculating $\|\mathbf{u}_n\|_{\mathbb{L}^2}^2$, we have
\begin{equation*}\label{lemapp-1-2}
\begin{split}
&\frac{1}{2}\|\mathbf{u}_n(t)\|_{\mathbb{L}^2}^2-\int_0^t\|\nabla \mathbf{u}_n\|_{\mathbb{L}^2}^2\,\mathrm{d}s+\int_0^t\|\Delta \mathbf{u}_n\|_{\mathbb{L}^2}^2\,\mathrm{d}s+2\int_0^t\| \mathbf{u}_n\|_{\mathbb{L}^4}^4\,\mathrm{d}s\\
&+4\int_0^t\|\mathbf{u}_n\cdot\nabla\mathbf{u}_n\|_{\mathbb{L}^2}^2\,\mathrm{d}s+2\int_0^t\||\mathbf{u}_n||\nabla\mathbf{u}_n|\|_{\mathbb{L}^2}^2\,\mathrm{d}s\\
&=\frac{1}{2}\|\mathbf{u}_n(0)\|_{\mathbb{L}^2}^2+2\int_0^t\| \mathbf{u}_n\|_{\mathbb{L}^2}^2\,\mathrm{d}s+\int_0^t\left(\mathbf{b}_n(\mathbf{u}_n),\mathbf{u}_n\right)_{\mathbb{L}^2}\mathrm{d}s\\
&+\int_0^t\left(\int_{\mathbb{B}}\mathbf{G}_n(l,\mathbf{u}_n)(\theta(s,l)-1)\nu(\mathrm{d}l),\mathbf{u}_n\right)_{\mathbb{L}^2}\mathrm{d}s.
\end{split}
\end{equation*}
Through a series of similar calculations as demonstrated in Lemma \ref{lem32}, we easily obtain
\begin{equation}\label{lemapp-1-3}
\begin{split}
&\sup_{s\in[0,t]}\|\mathbf{u}_n(s)\|_{\mathbb{L}^2}^2+\int_0^t\|\Delta \mathbf{u}_n\|_{\mathbb{L}^2}^2\,\mathrm{d}s+\int_0^t\| \mathbf{u}_n\|_{\mathbb{L}^4}^4\,\mathrm{d}s+\int_0^t\||\mathbf{u}_n||\nabla\mathbf{u}_n|\|_{\mathbb{L}^2}^2\,\mathrm{d}s\\
&\leq C+C\int_0^t\|\mathbf{b}_n(\mathbf{u}_n)\|_{\mathbb{L}^2}\|\mathbf{u}_n\|_{\mathbb{L}^2}\,\mathrm{d}s+C\left|\int_0^t\left(\int_{\mathbb{B}}\mathbf{G}_n(l,\mathbf{u}_n)(\theta(s,l)-1)\nu(\mathrm{d}l),\mathbf{u}_n\right)_{\mathbb{L}^2}\mathrm{d}s\right|\\
&\leq C+C\int_0^t(1+\sup_{r\in[0,s]}\|\mathbf{u}_n(s)\|_{\mathbb{L}^2}^2)\left(1+\int_{\mathbb{B}}|l||\theta(s,l)-1)|\nu(\mathrm{d}l)\right)\mathrm{d}s
\end{split}
\end{equation}
Noting that
$
\int_0^T\int_{\mathbb{B}}|l||\theta(s,l)-1)|\nu(\mathrm{d}l)\mathrm{d}s<\infty,
$
we can use the Gronwall lemma to \eqref{lemapp-1-3} to derive that
\begin{equation}\label{lemapp-1-5}
\begin{split}
&\sup_{s\in[0,t]}\|\mathbf{u}_n(s)\|_{\mathbb{L}^2}^2+\int_0^t\|\mathbf{u}_n\|_{\mathbb{H}^2}^2\,\mathrm{d}s+\int_0^t\| \mathbf{u}_n\|_{\mathbb{L}^4}^4\,\mathrm{d}s\leq C.
\end{split}
\end{equation}

To further obtain higher-order estimates, we consider the following functional $\{\mathbf{\bar{F}}: \mathbf{u}_n\mapsto\frac{1}{2}\|\nabla\mathbf{u}_n\|_{\mathbb{L}^2}^2+\frac{1}{2}\|\mathbf{u}_n\|_{\mathbb{L}^4}^4-\|\mathbf{u}_n\|_{\mathbb{L}^2}^2\}$. We will use some notations defined in section \ref{sec2} and section \ref{sec3}. By a direct calculation, we have
\begin{equation}\label{lemapp-1-6}
\begin{split}
&\mathbf{\bar{F}}(\mathbf{u}_n)=\frac{1}{2}\|\nabla\mathbf{u}_n(0)\|_{\mathbb{L}^2}^2+\frac{1}{2}\|\mathbf{u}_n(0)\|_{\mathbb{L}^4}^4-\|\mathbf{u}_n(0)\|_{\mathbb{L}^2}^2\\
&-\int_0^t(F_n(\mathbf{u}_n,\mathbf{H}_{\textrm{eff}}^n),\mathbf{H}_{\textrm{eff}}^n)_{\mathbb{L}^2}\,\mathrm{d}s-\int_0^t(F_n^6(\mathbf{u}_n),\mathbf{H}_{\textrm{eff}}^n)_{\mathbb{L}^2}\,\mathrm{d}s-\int_0^t(F_n^7(\mathbf{u}_n),\mathbf{H}_{\textrm{eff}}^n)_{\mathbb{L}^2}\,\mathrm{d}s.
\end{split}
\end{equation}
Recalling the definition of $F_n(\mathbf{u}_n,\mathbf{H}_{\textrm{eff}}^n)$, it follows that
\begin{equation*}\label{lemapp-1-7}
\begin{split}
(F _n(\mathbf{u}_n,\mathbf{H}_{\textrm{eff}}^n),\mathbf{H}_{\textrm{eff}}^n)_{\mathbb{L}^2}=\|\mathbf{H}_{\textrm{eff}}^n\|_{\mathbb{L}^2}^2+\|\nabla\mathbf{H}_{\textrm{eff}}^n\|_{\mathbb{L}^2}^2,
\end{split}
\end{equation*}
which combined with \eqref{lemapp-1-6} implies that
\begin{equation}\label{lemapp-1-8}
\begin{split}
&\mathbf{\bar{F}}(\mathbf{u}_n)+\int_0^t\|\mathbf{H}_{\textrm{eff}}^n\|_{\mathbb{L}^2}^2\,\mathrm{d}s+\int_0^t\|\nabla\mathbf{H}_{\textrm{eff}}^n\|_{\mathbb{L}^2}^2\,\mathrm{d}s\\
&\leq C+C\left|\int_0^t\left(\mathbf{b}_n(\mathbf{u}_n),\mathbf{H}_{\textrm{eff}}^n\right)_{\mathbb{L}^2}\mathrm{d}s\right|+C\left|\int_0^t\left(\int_{\mathbb{B}}\mathbf{G}_n(l,\mathbf{u}_n)(\theta(s,l)-1)\nu(\mathrm{d}l),\mathbf{H}_{\textrm{eff}}^n\right)_{\mathbb{L}^2}\mathrm{d}s\right|\\
&\leq C+\frac{1}{2}\int_0^t\|\mathbf{H}_{\textrm{eff}}^n\|_{\mathbb{L}^2}^2\,\mathrm{d}s+C\int_0^t(1+\sup_{r\in[0,s]}\|\mathbf{u}_n(s)\|_{\mathbb{L}^2}^2)\left(1+\int_{\mathbb{B}}|l||\theta(s,l)-1)|\nu(\mathrm{d}l)\right)\mathrm{d}s\\
&\leq C.
\end{split}
\end{equation}
Thus by the elliptic regularity result, we derive from \eqref{lemapp-1-8} and \eqref{lemapp-1-5} that
\begin{equation*}
\begin{split}
\sup_{s\in[0,t]}\|\mathbf{u}_n(s)\|_{\mathbb{H}^1}^2+\int_0^t\|\mathbf{u}_n\|_{\mathbb{H}^3}^2\,\mathrm{d}s\leq C.
\end{split}
\end{equation*}
The proof is thus completed.
\end{proof}

\begin{lemma}\label{lemapp-2} Let $T>0$ and $\mathbf{u}_n$ be the solution of \eqref{app2}. Let $\alpha\in(0,\frac{1}{2})$ and $p\geq2$. Then under the same assumption as of Lemma \ref{the61-1}, there exists a positive constant $C=C(\|\mathbf{u}_0\|_{\mathbb{H}^1},\theta,\mathbf{h},\mathbf{g},T)$ independent of $n$ such that for any $n\in\mathbb{N}$,
\begin{equation}\label{lemapp-2-1}
\begin{split}
\|\mathbf{u}_n\|_{W^{\alpha,p}(0,T;(\mathbb{H}^{1})^*)}^2\leq C.
\end{split}
\end{equation}
\end{lemma}
\begin{proof}[\emph{\textbf{Proof}}]  Thanks to Lemma \ref{lem35}, we have
\begin{equation*}
\begin{split}
&\sum_{j=1}^5\|F_n^j(\mathbf{u}_n)\|_{L^2(0,T;(\mathbb{H}^{1})^*)}^2\leq C,
\end{split}
\end{equation*}
which implies that
\begin{equation*}
\begin{split}
&\sum_{j=1}^5\left\|\int_0^{\cdot}F_n^j(\mathbf{u}_n)\right\|_{W^{1,2}(0,T;(\mathbb{H}^{1})^*)}^2\leq C.
\end{split}
\end{equation*}
Since $\mathbb{L}^2\hookrightarrow(\mathbb{H}^{1})^*$, we have
\begin{equation*}
\begin{split}
&\|F_n^6(\mathbf{u}_n)\|_{L^2(0,T;(\mathbb{H}^{1})^*)}^2\leq\|F_n^6(\mathbf{u}_n)\|_{L^2(0,T;\mathbb{L}^{2})}^2\leq C.
\end{split}
\end{equation*}
Since
\begin{equation*}
\begin{split}
&\left\|\int_0^t\int_{B}\mathbf{G}_n(l,\mathbf{u}_n)(\theta(s,l)-1)\nu(\mathrm{d}l)\mathrm{d}s\right\|_{\mathbb{L}^2}^p\\
&\leq C\left[(1+\sup_{t\in[0,T]}\|\mathbf{u}_n(t)\|_{\mathbb{L}^2})\int_0^T\int_{\mathbb{B}}|l||\theta(s,l)-1|\nu(\mathrm{d}l)\mathrm{d}s\right]^p\leq C,
\end{split}
\end{equation*}
 by Fubini's theorem,  we infer that
\begin{equation*}
\begin{split}
&\left\|\int_0^{\cdot}F_n^7(\mathbf{u}_n)\right\|_{W^{\alpha,p}(0,T;(\mathbb{H}^{1})^*)}^p\\
&\lesssim\left\|\int_0^{\cdot}F_n^7(\mathbf{u}_n)\right\|_{W^{\alpha,p}(0,T;\mathbb{L}^{2})}^p=\int_0^T\left\|\int_0^t\int_{\mathbb{B}}\mathbf{G}_n(l,\mathbf{u}_n)(\theta(s,l)-1)\nu(\mathrm{d}l)\mathrm{d}s\right\|_{\mathbb{L}^2}^p\,\mathrm{d}t\\
&+\int_0^T\int_0^T\frac{\|\int_s^t\int_{\mathbb{B}}\mathbf{G}_n(l,\mathbf{u}_n)(\theta(r,l)-1)\nu(\mathrm{d}l)\mathrm{d}r\|_{\mathbb{L}^2}^p}{|t-s|^{1+\alpha p}}\,\mathrm{d}t\mathrm{d}s\\
&\leq C+C\int_0^T\int_0^T\frac{1}{|t-s|^{1+\alpha p}}\,\mathrm{d}t\mathrm{d}s\leq C.
\end{split}
\end{equation*}
Since $W^{1,2}(0,T;(\mathbb{H}^{1})^*)\hookrightarrow W^{\alpha,p}(0,T;(\mathbb{H}^{1})^*)$ for $\frac{1}{2}+\frac{1}{p}>\alpha$ \cite{flandoli1995martingale} and
\begin{equation*}
\begin{split}
&\mathbf{u}_n(t)=\mathbf{u}_n(0)+\sum_{j=1}^7\int_0^tF_n^j(s)\mathrm{d}s,
\end{split}
\end{equation*}
the inequality \eqref{lemapp-2-1} is valid.\end{proof}

\begin{proof}[\emph{\textbf{Proof of Lemma \ref{the61-1}}}] Thanks to Lemma \ref{lemapp-1}, there exists a subsequence of $\mathbf{u}_n$ (still denoted by $\mathbf{u}_n$) such that
\begin{equation*}
\begin{split}
&\mathbf{u}_n\rightarrow \mathbf{u}~\textrm{weak-star in}~L^{\infty}(0,T;\mathbb{H}^1),\\
&\mathbf{u}_n\rightarrow \mathbf{u}~\textrm{weakly in}~L^2(0,T;\mathbb{H}^3).
\end{split}
\end{equation*}
Since the embeddings
\begin{equation*}
\begin{split}
&L^2(0,T;\mathbb{H}^3)\cap W^{\alpha,2}(0,T;(\mathbb{H}^{1})^*)\hookrightarrow L^2(0,T;\mathbb{W}^{2,4}),\\
&L^p(0,T;\mathbb{H}^1)\cap W^{\alpha,p}(0,T;(\mathbb{H}^{1})^*)\hookrightarrow L^p(0,T;\mathbb{L}^{4})
\end{split}
\end{equation*}
are compact \cite{flandoli1995martingale}, by Lemmas \ref{lemapp-1} and \ref{lemapp-2}, we infer that for $p\geq2$
\begin{equation*}
\begin{split}
&\mathbf{u}_n\rightarrow \mathbf{u}~\textrm{strongly in}~L^{2}(0,T;\mathbb{W}^{2,4})\cap L^p(0,T;\mathbb{L}^{4}).
\end{split}
\end{equation*}
Thus by a similar argument to the proof of \eqref{3cor5-1}, it is not difficult to prove that $\mathbf{u}$ satisfies the equality \eqref{1the61-1}. Moreover, it is easy to check that
\begin{equation*}
\begin{split}
&\left\|\mathbf{b}(\mathbf{u})\right\|_{L^2(0,T;\mathbb{L}^2)}<\infty,
\end{split}
\end{equation*}
and
\begin{equation*}
\begin{split}
&\left\|\int_{\mathbb{B}}\mathbf{G}(l,\mathbf{u})(\theta(\cdot,l)-1)\nu(\mathrm{d}l)\right\|_{L^1(0,T;\mathbb{H}^1)}\\
&\leq(1+\sup_{t\in[0,T]}\|\mathbf{u}(t)\|_{\mathbb{H}^1})\int_0^T\int_{\mathbb{B}}|l||\theta(s,l)-1|\nu(\mathrm{d}l)\mathrm{d}s<\infty.
\end{split}
\end{equation*}
Therefore, we have
\begin{equation}\label{1app}
\begin{split}
&\frac{\mathrm{d}\mathbf{u}}{\mathrm{d}t}\in L^2(0,T;(\mathbb{H}^{3})^*)+L^1(0,T;\mathbb{H}^{1}),\\
&\mathbf{u}\in L^2(0,T;\mathbb{H}^{3})\cap L^{\infty}(0,T;\mathbb{H}^{1}).
\end{split}
\end{equation}
Due to Lemma 1.2, (1.84) and (1.85) in Temam \cite{temam2024navier}, we derive from \eqref{1app} that
\begin{equation*}
\begin{split}
\mathbf{u}\in C([0,T];\mathbb{H}^{1})\cap L^2(0,T;\mathbb{H}^{3}).
\end{split}
\end{equation*}
The proof of uniqueness is standard, so we omit the details. The proof is completed.\end{proof}

\subsubsection{Stochastic control equation}\label{subsubsec6-2} Let $\varepsilon>0$ and $\varphi\in\bar{\mathcal{A}}$. We consider the following stochastic partial differential equation.
\begin{equation}\label{sys61-2}
\begin{split}
\mathrm{d}\mathbf{u}^{\varepsilon,\varphi}(t)&=\Big(-\Delta\mathbf{u}^{\varepsilon,\varphi}
-\Delta^2\mathbf{u}^{\varepsilon,\varphi}+2(1-|\mathbf{u}^{\varepsilon,\varphi}|^2)\mathbf{u}^{\varepsilon,\varphi}
+2\Delta(|\mathbf{u}^{\varepsilon,\varphi}|^2\mathbf{u}^{\varepsilon,\varphi})\\
&-\mathbf{u}^{\varepsilon,\varphi}\times\Delta \mathbf{u}^{\varepsilon,\varphi}\Big)\mathrm{d}t+\mathbf{b}(\mathbf{u}^{\varepsilon,\varphi})\,\mathrm{d}t\\
&+\varepsilon\int_{\mathbb{B}}\mathbf{G}(l,\mathbf{u}^{\varepsilon,\varphi})\left(\eta^{\varepsilon^{-1}\varphi}(\mathrm{d}t,\mathrm{d}l)-\varepsilon^{-1}\nu(\mathrm{d}l)\mathrm{d}s\right)
\end{split}
\end{equation}
For convenience, we denote $\mathbf{u}^{\varepsilon,\varphi}(t)$ by $U(t)$. According to \eqref{63}, the equation \eqref{sys61-2} can be written in the following integral form.
\begin{equation}\label{sys61-3}
\begin{split}
U(t)&=U_0+\int_0^t-\Delta U-\Delta^2U+2(1-|U|^2)U+2\Delta(|U|^2U)-U\times\Delta U\,\mathrm{d}s+\int_0^t\mathbf{b}(U)\,\mathrm{d}s\\
&+\varepsilon\int_0^t\int_{\mathbb{B}}\mathbf{G}(l,U)\tilde{\eta}^{\varepsilon^{-1}\varphi}(\mathrm{d}s,\mathrm{d}l)+\int_0^t\int_{\mathbb{B}}\mathbf{G}(l,U)(\varphi(s,l)-1)\nu(\mathrm{d}l)\mathrm{d}s.
\end{split}
\end{equation}
Let $\{K_n\}_{n\in\mathbb{N}}$ be a sequence of compact sets such that $\bigcup_{n\in\mathbb{N}}K_n=\mathbb{B}$. Let
$
\bar{\Omega}=\bar{\mathbb{M}}_T$ and $\bar{\mathcal{F}}=\mathcal{T}(\bar{\mathbb{M}}_T).
$
Let us denote
\begin{equation*}\label{6-1-1}
\begin{split}
\bar{\mathcal{A}}_b:=\bigcup_{n=1}^{\infty}\Bigl\{\varphi\in\bar{\mathcal{A}}:\varphi(t,x,\omega)\in\left[\frac{1}{n},n\right]~\textrm{if}~(t,x,\omega)\in[0,T]\times K_n\times\bar{\Omega}\\
\textrm{and}~\varphi(t,x,\omega)=1~\textrm{if}~(t,x,\omega)\in[0,T]\times K_n^c\times\bar{\Omega}\Bigl\}.
\end{split}
\end{equation*}
Now we state the following fundamental result. The key of proving this result is a Girsanov-type theorem for PRM. We provide a brief outline of the proof, with more detailed steps available in section 7 and Theorem 6.1 of \cite{brzezniak2022well}.
\begin{lemma}\label{the66-1} Assume $\varepsilon>0$. Let $\varphi\in \bar{\mathcal{A}}_b$ defined on $\left(\bar{\Omega},\bar{\mathcal{F}},\bar{\mathbb{F}},\bar{\mathbb{P}}\right)$. Then the process defined by
\begin{equation*}\label{6666-1}
\begin{split}
X^{\varepsilon}:=\mathcal{J}^{\varepsilon}(\varepsilon\eta^{\varepsilon^{-1}\varphi})
\end{split}
\end{equation*}
is the unique solution of problem \eqref{sys61-3}.
\end{lemma}
\begin{proof}[\emph{\textbf{Proof}}]  Let $\varphi\in \bar{\mathcal{A}}_b$ and $\psi:=\varphi^{-1}$. Then $\psi\in \bar{\mathcal{A}}$. Thus there exists $n\in\mathbb{N}$ and a compact set $K_n$ such that
\begin{equation*}\label{the66-1-1}
\begin{split}
&\psi(t,l,\omega)\in\left[\frac{1}{n},n\right]~\textrm{if}~(t,l,\omega)\in[0,T]\times K_n\times\bar{\Omega},\\
&\psi(t,l,\omega)=1~\textrm{if}~(t,l,\omega)\in[0,T]\times K_n^c\times\bar{\Omega}.
\end{split}
\end{equation*}
Let
\begin{equation*}\label{the66-1-2}
\begin{split}
E_t^{\varepsilon}(\psi)&:=\exp\Bigl\{\int_{(0,t]\times\mathbb{B}\times [0,\varepsilon^{-1}\varphi(s,l)]}\log(\psi(s,l))\bar{\eta}(\mathrm{d}l,\mathrm{d}s,\mathrm{d}r)\\
&+\int_{(0,t]\times\mathbb{B}\times [0,\varepsilon^{-1}\varphi(s,l)]}-(\psi(s,l)+1)\nu(\mathrm{d}l)\mathrm{d}s\mathrm{d}r\Bigl\}\\
&=\exp\Bigl\{\int_{(0,t]\times K_n\times [0,\varepsilon^{-1}\varphi(s,l)]}\log(\psi(s,l))\bar{\eta}(\mathrm{d}l,\mathrm{d}s,\mathrm{d}r)\\
&+\int_{(0,t]\times K_n\times [0,\varepsilon^{-1}\varphi(s,l)]}-(\psi(s,l)+1)\nu(\mathrm{d}l)\mathrm{d}s\mathrm{d}r\Bigl\}.
\end{split}
\end{equation*}
Then the following statements are true.
\begin{enumerate}
\item [(1)] According to Lemma 2.3 in \cite{budhiraja2011variational}, $E_t^{\varepsilon}(\psi)$ is a $\bar{\mathbb{F}}$-martingale on $(\bar{\Omega},\bar{\mathcal{F}},\bar{\mathbb{F}},\mathbb{\bar{P}})$.
\item[(2)] By Lemma \ref{the61} and Remark \ref{the62}, the problem \eqref{sys62} admits a uniqueness solution
$
\mathbf{u}^{\varepsilon}=\mathcal{J}^{\varepsilon}(\varepsilon\eta^{\varepsilon^{-1}}),
$
defined on $\left(\bar{\Omega},\bar{\mathcal{F}},\bar{\mathbb{F}},\bar{\mathbb{P}}\right)$.
\item[(3)] The formula
$
\mathbb{P}_T^{\varepsilon}(A):=\int_{A}E_T^{\varepsilon}(\psi)\mathrm{d}\bar{\mathbb{P}},~A\in\bar{\mathcal{F}},
$
defines a probability measure on $(\bar{\Omega},\bar{\mathcal{F}})$. $\mathbb{P}_T^{\varepsilon}$ and $\bar{\mathbb{P}}$ are equivalent on $(\bar{\Omega},\bar{\mathcal{F}})$.
\item[(4)] On $(\bar{\Omega},\bar{\mathcal{F}},\bar{\mathbb{F}},\mathbb{P}_T^{\varepsilon})$, $\varepsilon\eta^{\varepsilon^{-1}\varphi}$ has the same law as that of $\varepsilon\eta^{\varepsilon^{-1}}$ on $(\bar{\Omega},\bar{\mathcal{F}},\bar{\mathbb{F}},\bar{\mathbb{P}})$.
\end{enumerate}
Thus by (2) and (4), the process $\mathcal{J}^{\varepsilon}(\varepsilon\eta^{\varepsilon^{-1}\varphi})$ defined on $(\bar{\Omega},\bar{\mathcal{F}},\bar{\mathbb{F}},\mathbb{P}_T^{\varepsilon})$ is the unique solution of problem \eqref{sys61-3}. Moreover by a standard argument as shown in Lemma 7.1 of \cite{brzezniak2022well}, it is not hard to derive that $\mathcal{J}^{\varepsilon}(\varepsilon\eta^{\varepsilon^{-1}\varphi})$ defined on $(\bar{\Omega},\bar{\mathcal{F}},\bar{\mathbb{F}},\bar{\mathbb{P}})$ is the unique solution of \eqref{sys61-3}.\end{proof}

\subsection{Proof of Theorem \ref{the2}}\label{subsec6-3} Now we state the sufficient conditions for establishing an LDP for the family $\mathcal{J}^{\varepsilon}(\varepsilon\eta^{\varepsilon^{-1}})$.

\textsf{Condition 1}:  For all $K\in\mathbb{N}$, let $\theta_n,~\theta\in S^K$ and $\theta_n\rightarrow\theta$ in $S^K$ (that is $\nu^{\theta_n}\rightarrow\nu^{\theta}$) as $n\rightarrow\infty$, then
\begin{equation*}\label{663-2}
\begin{split}
\mathcal{J}^0(\theta_n)\rightarrow\mathcal{J}^0(\theta),~\textrm{i.e.},~\mathbf{u}^{\theta_n}\rightarrow\mathbf{u}^{\theta}~\textrm{in}~\mathcal{Z}_T.
\end{split}
\end{equation*}

\textsf{Condition 2}: Let $K_n,~n\in\mathbb{N}$ be an increasing sequence of compact subsets of $\mathbb{B}$ such that $\bigcup_{n\in\mathbb{N}}K_n=\mathbb{B}$. Let us denote
$
\mathcal{U}^K:=\{\varphi\in\bar{\mathcal{A}}_b:\varphi\in S^K,~\bar{\mathbb{P}}\textrm{-a.s.}\}.
$
Let $\{\varepsilon_n\}_{n\in\mathbb{N}}$ be a $(0,1]$-valued sequence converging to $0$. For all $K\in\mathbb{N}$, let $\varphi_{\varepsilon_n},~\varphi\in S^K$ be such that $\varphi_{\varepsilon_n}$ converges in law to $\varphi$ as $\varepsilon_n\rightarrow 0$. Then
\begin{equation*}\label{663-3}
\begin{split}
\mathcal{J}^{\varepsilon_n}(\varepsilon_n\eta^{\varepsilon_n^{-1}\varphi_{\varepsilon_n}})~\textrm{converges in law to}~\mathcal{J}^0(\varphi)~\textrm{in}~\mathcal{Z}_T.
\end{split}
\end{equation*}

\subsubsection{Verification of Condition 1}\label{subsubsec6-4-1} Condition 1 is a consequence of the following Lemma.
\begin{lemma}\label{lem666-1}   Let $K\in\mathbb{N}$ and let $\theta_n,~\theta\in S^K$, be such that
\begin{equation*}
\begin{split}
\theta_n\rightarrow\theta~\textrm{in}~S^K~\textrm{as}~n\rightarrow\infty.
\end{split}
\end{equation*}
Then
$
\mathcal{J}^0(\theta_n)\rightarrow\mathcal{J}^0(\theta)~\textrm{in}~\mathcal{Z}_T~\textrm{as}~n\rightarrow\infty.
$
In particular if $\theta$ and $\tilde{\theta}\in S^K$, possibly defined on different probability spaces $\Omega,~\tilde{\Omega}$, with the same laws, then the laws of the random variables
\begin{equation*}
\begin{split}
\Omega\ni\omega\mapsto\mathcal{J}^0(\theta_n)\in \mathcal{Z}_T~\textrm{and}~\tilde{\Omega}\ni\tilde{\omega}\mapsto\mathcal{J}^0(\tilde{\theta}_n)\in \mathcal{Z}_T
\end{split}
\end{equation*}
are equal.
\end{lemma}
\begin{proof}[\emph{\textbf{Proof}}] Assume $\theta\in \mathbb{S}$. Set $\mathbf{u}^{\theta}$ be the solution of the skeleton equation given by \begin{equation}\label{sys666-1}
\begin{split}
\mathrm{d}\mathbf{u}^{\theta}(t)&=\left[-\Delta\mathbf{u}^{\theta}-\Delta^2\mathbf{u}^{\theta}+2(1-|\mathbf{u}^{\theta}|^2)\mathbf{u}^{\theta}+2\Delta(|\mathbf{u}^{\theta}|^2\mathbf{u}^{\theta})-\mathbf{u}^{\theta}\times\Delta \mathbf{u}^{\theta}\right]\,\mathrm{d}t\\
&+\mathbf{b}(\mathbf{u}^{\theta})\,\mathrm{d}t+\int_{\mathbb{B}}\mathbf{G}(l,\mathbf{u}^{\theta})(\theta(t,l)-1)\nu(\mathrm{d}l)\mathrm{d}t.
\end{split}
\end{equation}
Similarly $\mathbf{u}^{\theta_n}$ is a solution of the above with $\theta$ replaced by $\theta_n$. For convenience, we define the solution of the skeleton equation \eqref{sys666-1} by
$
\mathbf{u}=\mathbf{u}^{\theta}:=\mathcal{J}^0(\theta)~\textrm{and}~\mathbf{u}_n=\mathbf{u}^{\theta_n}:=\mathcal{J}^0(\theta_n).
$

We shall prove that $\mathbf{u}_n\rightarrow\mathbf{u}$ in $\mathcal{Z}_T$. The equation \eqref{sys666-1} can be written in the terms of $\mathbf{u}_n$ in the integral form as follows:
\begin{equation}\label{sys666-2}
\begin{split}
\mathbf{u}_n(t)&=\mathbf{u}_0+\int_0^t\left[-\Delta\mathbf{u}_n-\Delta^2\mathbf{u}_n+2(1-|\mathbf{u}_n|^2)\mathbf{u}_n+2\Delta(|\mathbf{u}_n|^2\mathbf{u}_n)-\mathbf{u}_n\times\Delta \mathbf{u}_n\right]\,\mathrm{d}s\\
&+\int_0^t\mathbf{b}(\mathbf{u}_n)\,\mathrm{d}s+\int_0^t\int_{\mathbb{B}}\mathbf{G}(l,\mathbf{u}_n)(\theta_n(s,l)-1)\nu(\mathrm{d}l)\mathrm{d}s.
\end{split}
\end{equation}
Using the energy estimates provided by Lemmas \ref{lemapp-1} and \ref{lemapp-2}, we infer that
\begin{equation}\label{sys666-3}
\begin{split}
&\mathbf{u}_n\rightarrow \mathbf{u}^*~\textrm{weak-star in}~L^{\infty}(0,T;\mathbb{H}^1),\\
&\mathbf{u}_n\rightarrow \mathbf{u}^*~\textrm{weakly in}~L^2(0,T;\mathbb{H}^3),\\
&\mathbf{u}_n\rightarrow \mathbf{u}^*~\textrm{strongly in}~L^{2}(0,T;\mathbb{W}^{2,4})\cap L^p(0,T;\mathbb{L}^{4}).
\end{split}
\end{equation}
In particular, let
\begin{equation}\label{sys666-31}
\begin{split}
\sup_{n\in\mathbb{N}}\sup_{t\in[0,T]}\|\mathbf{u}_n\|_{\mathbb{H}^1}+\sup_{t\in[0,T]}\|\mathbf{u}^*\|_{\mathbb{H}^1}=:\mathbf{c}_0<\infty.
\end{split}
\end{equation}

At first we will prove that $\mathbf{u}^*$ is a solution of \eqref{sys666-1}, which means that $\mathbf{u}^*=\mathbf{u}$. To complete this goal, we use similar arguments as in the proof of Proposition $V.1.3$ in \cite{boyer2012mathematical} (see also Temam \cite{temam2024navier}).  Let $h(t)$ be a continuously differentiable function on $[0,T]$ with $h(T)=0$. Let $\{\mathbf{e}_i\}_{i=1}^{\infty}$ denote an orthonormal basis of $\mathbb{L}^2$ consisting of eigenvectors for the Neumann Laplacian. According to Lemma \ref{the61-1}, we see that
\begin{equation}\label{sys666-4}
\begin{split}
&-\int_0^T(\mathbf{u}_n(t),h'(t)\mathbf{e}_i)_{\mathbb{L}^2}\,\mathrm{d}t=(\mathbf{u}_0,h(0)\mathbf{e}_i)_{\mathbb{L}^2}-\int_0^{T}\left\langle\Delta\mathbf{u}_n,h(t)\mathbf{e}_i\right\rangle_{\mathbb{H}^{-1},\mathbb{H}^1}\,\mathrm{d}t\\
&-\int_0^{T}\left\langle\Delta^2\mathbf{u}_n,h(t)\mathbf{e}_i\right\rangle_{\mathbb{H}^{-1},\mathbb{H}^1}\,\mathrm{d}t+2\int_0^{T}\left\langle(1-|\mathbf{u}_n|^2)\mathbf{u}_n,h(t)\mathbf{e}_i\right\rangle_{\mathbb{H}^{-1},\mathbb{H}^1}\,\mathrm{d}t\\
&-\int_0^{T}\left\langle\mathbf{u}_n\times\Delta\mathbf{u}_n,h(t)\mathbf{e}_i\right\rangle_{\mathbb{H}^{-1},\mathbb{H}^1}\,\mathrm{d}t+2\int_0^{T}\left\langle\Delta(|\mathbf{u}_n|^2\mathbf{u}_n),h(t)\mathbf{e}_i\right\rangle_{\mathbb{H}^{-1},\mathbb{H}^1}\,\mathrm{d}t\\
&+\int_0^T\left(\mathbf{b}(\mathbf{u}_n),h(t)\mathbf{e}_i\right)_{\mathbb{L}^2}\,\mathrm{d}t+\int_0^T\int_{\mathbb{B}}\left(\mathbf{G}(l,\mathbf{u}_n)(\theta_n(t,l)-1),h(t)\mathbf{e}_i\right)_{\mathbb{L}^2}\nu(\mathrm{d}l)\mathrm{d}t\\
&:=\sum_{j=1}^7Z_j(\mathbf{u}_n)+Z_8(\mathbf{u}_n,\theta_n).
\end{split}
\end{equation}
Using the result \eqref{sys666-3} and the standard arguments to the proof of Theorem 3.1 in \cite{temam2024navier}, we infer that
 \begin{equation}\label{sys666-4-1}
\begin{split}
&\lim_{n\rightarrow\infty}-\int_0^T(\mathbf{u}_n(t),h'(t)\mathbf{e}_i)_{\mathbb{L}^2}\,\mathrm{d}t-\sum_{j=1}^7Z_j(\mathbf{u}_n)=-\int_0^T(\mathbf{u}^*(t),h'(t)\mathbf{e}_i)_{\mathbb{L}^2}\,\mathrm{d}t-\sum_{j=1}^7Z_j(\mathbf{u}^*)
\end{split}
\end{equation}
Moreover, as $\theta_n\rightarrow\theta$ in $\mathbb{S}$, by Lemma 3.11 in \cite{budhiraja2013large}, we have
 \begin{equation}\label{sys666-5}
\begin{split}
&\lim_{n\rightarrow\infty}Z_8(\mathbf{u}^*,\theta_n)=Z_8(\mathbf{u}^*,\theta).
\end{split}
\end{equation}
Set for $\delta>0$, $A_{n,\delta}:=\{t\in[0,T],\|\mathbf{u}_n(t)-\mathbf{u}^*\|_{\mathbb{H}^1}\geq\delta\}$. Let $\lambda_T$ denotes the Lebesgue measure on $[0,T]$. Since $\mathbf{u}_n\rightarrow\mathbf{u}^*$ strongly in $L^2(0,T;\mathbb{H}^1)$,
\begin{equation}\label{sys666-6}
\begin{split}
&\lim_{n\rightarrow\infty}\lambda_T(A_{n,\delta})\leq\lim_{n\rightarrow\infty}\frac{C}{\delta^2}\int_0^T\|\mathbf{u}_n-\mathbf{u}^*\|_{\mathbb{H}^1}^2\,\mathrm{d}t=0.
\end{split}
\end{equation}
Thus for any $\delta>0$, we have
\begin{equation*}
\begin{split}
&|Z_8(\mathbf{u}_n,\theta_n)-Z_8(\mathbf{u}^*,\theta_n)|=\left|\int_0^T\int_{\mathbb{B}}\left((\mathbf{G}(l,\mathbf{u}_n)-\mathbf{G}(l,\mathbf{u}^*))(\theta_n(t,l)-1),h(t)\mathbf{e}_i\right)_{\mathbb{L}^2}\nu(\mathrm{d}l)\mathrm{d}t\right|\\
&\leq C_{h}\int_0^T\int_{\mathbb{B}}|l|\|\mathbf{u}_n-\mathbf{u}^*\|_{\mathbb{H}^1}|\theta_n(t,l)-1|\nu(\mathrm{d}l)\mathrm{d}t\\
&\leq 2\mathbf{c}_0C_{h}\int_{A_{n,\delta}}\int_{\mathbb{B}}|l||\theta_n(t,l)-1|\nu(\mathrm{d}l)\mathrm{d}t+\delta C_h\int_{A_{n,\delta}^c}\int_{\mathbb{B}}|l||\theta_n(t,l)-1|\nu(\mathrm{d}l)\mathrm{d}t
\end{split}
\end{equation*}
Taking the limit, using \eqref{sys666-6} and noting the arbitrariness of $\delta$, we infer that
\begin{equation}\label{sys666-7}
\begin{split}
&\lim_{n\rightarrow\infty}|Z_8(\mathbf{u}_n,\theta_n)-Z_8(\mathbf{u}^*,\theta_n)|=0.
\end{split}
\end{equation}
Now combining \eqref{sys666-5} and \eqref{sys666-7}, we infer that
\begin{equation}\label{sys666-8}
\begin{split}
&\lim_{n\rightarrow\infty}|Z_8(\mathbf{u}_n,\theta_n)-Z_8(\mathbf{u}^*,\theta)|=0.
\end{split}
\end{equation}
Thus we see from \eqref{sys666-4}, \eqref{sys666-4-1} and \eqref{sys666-8} that
\begin{equation*}\label{sys666-9}
\begin{split}
&\lim_{n\rightarrow\infty}-\int_0^T(\mathbf{u}_n(t),h'(t)\mathbf{e}_i)_{\mathbb{L}^2}\,\mathrm{d}t=-\int_0^T(\mathbf{u}^*(t),h'(t)\mathbf{e}_i)_{\mathbb{L}^2}\,\mathrm{d}t\\
&=\lim_{n\rightarrow\infty}\sum_{j=1}^7Z_j(\mathbf{u}_n)+Z_8(\mathbf{u}_n,\theta_n)=\sum_{j=1}^7Z_j(\mathbf{u}^*)+Z_8(\mathbf{u}^*,\theta^*).
\end{split}
\end{equation*}
Using similar arguments as in the proof of Theorem 3.1, Chapter 3, Temam \cite{temam2024navier}, we conclude that $\mathbf{u}^*$ is the desired solution of problem \eqref{sys666-1}, that is $\mathbf{u}^*=\mathbf{u}=\mathbf{u}^{\theta}$.

Next we shall show that $\mathbf{u}_n\rightarrow\mathbf{u}$ in  $C([0,T];\mathbb{H}^1)\cap L^2(0,T;\mathbb{H}^3)$. We observe that the above proof of \eqref{sys666-7} yields the following result, which will be used later on:
\begin{equation}\label{sys666-10}
\begin{split}
&\lim_{n\rightarrow\infty}\sup_{f\in S^K}\int_0^T\|\mathbf{u}_n(t)-\mathbf{u}(t)\|_{\mathbb{H}^1}\int_{\mathbb{B}}|l||f(t,l)-1|\nu(\mathrm{d}l)\mathrm{d}t=0.
\end{split}
\end{equation}
We use some notations defined in \eqref{def111}. Let $X_n:=\mathbf{u}_n-\mathbf{u}$. Then,
\begin{equation}\label{sys666-11}
\begin{split}
&X_n(t)=\sum_{j=1}^5\int_0^tF^j(\mathbf{u}_n)-F^j(\mathbf{u})\,\mathrm{d}s+\int_0^t\mathbf{b}(\mathbf{u}_n)-\mathbf{b}(\mathbf{u})\,\mathrm{d}s\\
&+\int_0^t\int_{\mathbb{B}}\mathbf{G}(l,\mathbf{u}_n)(\theta_n(s,l)-1)-\mathbf{G}(l,\mathbf{u})(\theta(s,l)-1)\nu(\mathrm{d}l)\mathrm{d}s.
\end{split}
\end{equation}
By directly calculating $\|X_n(t)\|_{\mathbb{L}^2}^2$ and $\|\nabla X_n(t)\|_{\mathbb{L}^2}^2$ and using integration by parts, we obtain
\begin{equation}\label{sys666-12}
\begin{split}
&\|X_n(t)\|_{\mathbb{L}^2}^2+2\int_0^{t}\|\Delta X_n\|_{\mathbb{L}^2}^2\,\mathrm{d}s\\
&=2\int_0^{t}\|\nabla X_n\|_{\mathbb{L}^2}^2\,\mathrm{d}s+4\int_0^{t}\left(X_n-|\mathbf{u}_n|^2X_n+\mathbf{u}\left(|\mathbf{u}|^2-|\mathbf{u}_n|^2\right),X_n\right)_{\mathbb{L}^2}\,\mathrm{d}s\\
&+4\int_0^{t}\left(|\mathbf{u}_n|^2\mathbf{u}_n-|\mathbf{u}|^2\mathbf{u},\Delta X_n\right)_{\mathbb{L}^2}\,\mathrm{d}s+2\int_0^{t}\left(\mathbf{u}_n\times\nabla\mathbf{u}_n-\mathbf{u}\times\nabla\mathbf{u},\nabla X_n\right)_{\mathbb{L}^2}\,\mathrm{d}s\\
&+2\int_0^{t}\left(\mathbf{b}(\mathbf{u}_n)-\mathbf{b}(\mathbf{u}),X_n\right)_{\mathbb{L}^2}\,\mathrm{d}s\\
&+2\int_0^t\int_{\mathbb{B}}\left(\mathbf{G}(l,\mathbf{u}_n)(\theta_n(s,l)-1)-\mathbf{G}(l,\mathbf{u})(\theta(s,l)-1),X_n\right)_{\mathbb{L}^2}\nu(\mathrm{d}l)\mathrm{d}s\\
&:=\sum_{j=1}^6M_j(t),
\end{split}
\end{equation}
and
\begin{equation}\label{sys666-13}
\begin{split}
&\|\nabla X_n(t)\|_{\mathbb{L}^2}^2+2\int_0^{t}\|\nabla\Delta X_n\|_{\mathbb{L}^2}^2\,\mathrm{d}s\\
&=2\int_0^{t}\|\Delta X_n\|_{\mathbb{L}^2}^2\,\mathrm{d}s+4\int_0^{t}\left(\nabla((1-|\mathbf{u}_n|^2)\mathbf{u}_n-(1-|\mathbf{u}|^2)\mathbf{u}),\nabla X_n\right)_{\mathbb{L}^2}\,\mathrm{d}s\\
&+4\int_0^{t}\left(\nabla(|\mathbf{u}_n|^2\mathbf{u}_n-|\mathbf{u}|^2\mathbf{u}),\nabla\Delta X_n\right)_{\mathbb{L}^2}\,\mathrm{d}s+2\int_0^{t}\left(\mathbf{u}_n\times\Delta\mathbf{u}_n-\mathbf{u}\times\Delta\mathbf{u},\Delta X_n\right)_{\mathbb{L}^2}\,\mathrm{d}s\\
&+2\int_0^{t}\left(\nabla\left(\mathbf{b}(\mathbf{u}_n)-\mathbf{b}(\mathbf{u})\right),\nabla X_n\right)_{\mathbb{L}^2}\,\mathrm{d}s\\
&+2\int_0^t\int_{\mathbb{B}}\left(\nabla\mathbf{G}(l,\mathbf{u}_n)(\theta_n(s,l)-1)-\nabla\mathbf{G}(l,\mathbf{u})(\theta(s,l)-1),\nabla X_n\right)_{\mathbb{L}^2}\nu(\mathrm{d}l)\mathrm{d}s\\
&:=\sum_{j=1}^6N_j(t).
\end{split}
\end{equation}
By the H\"{o}lder inequality and Young's inequality, we have
\begin{equation}\label{sys666-14}
\begin{split}
&|M_1(t)|\leq \varepsilon\int_0^{t}\|\Delta X_n\|_{\mathbb{L}^2}^2\,\mathrm{d}s+C_{\varepsilon}\int_0^{t}\|X_n\|_{\mathbb{L}^2}^2\,\mathrm{d}s.
\end{split}
\end{equation}
And by the Sobolev embedding $\mathbb{H}^2\hookrightarrow L^{\infty}$ we also have
\begin{equation}\label{sys666-15}
\begin{split}
&|M_2(t)|\leq C\int_0^{t}(1+\|\mathbf{u}_n\|_{\mathbb{L}^{\infty}}^2+\|\mathbf{u}\|_{\mathbb{L}^{\infty}}^2)\|X_n\|_{\mathbb{L}^2}^2\,\mathrm{d}s\\
&\leq C\int_0^{t}(1+\|\mathbf{u}_n\|_{\mathbb{H}^{2}}^2+\|\mathbf{u}\|_{\mathbb{H}^{2}}^2)\|X_n\|_{\mathbb{L}^2}^2\,\mathrm{d}s.
\end{split}
\end{equation}
By the GN inequality, the following inequality is valid for $d=1,2,3$.
\begin{equation}\label{addd1}
\begin{split}
\|f\|_{\mathbb{L}^{\infty}}\leq C\|f\|_{\mathbb{H}^{1}}^{\frac{1}{2}}\|f\|_{\mathbb{H}^{2}}^{\frac{1}{2}},~f\in\mathbb{H}^2.
\end{split}
\end{equation}
Thus we have
\begin{equation}\label{sys666-16}
\begin{split}
&|M_3(t)|\leq C\int_0^{t}(1+\|\mathbf{u}_n\|_{\mathbb{L}^{\infty}}^2+\|\mathbf{u}\|_{\mathbb{L}^{\infty}}^2)\|X_n\|_{\mathbb{L}^2}\|\Delta X_n\|_{\mathbb{L}^2}\,\mathrm{d}s\\
&\leq \varepsilon\int_0^{t}\|\Delta X_n\|_{\mathbb{L}^2}^2\,\mathrm{d}s+C_{\varepsilon}\int_0^{t}(1+\|\mathbf{u}_n\|_{\mathbb{L}^{\infty}}^4+\|\mathbf{u}\|_{\mathbb{L}^{\infty}}^4)\|X_n\|_{\mathbb{L}^2}^2\,\mathrm{d}s\\
&\leq \varepsilon\int_0^{t}\|\Delta X_n\|_{\mathbb{L}^2}^2\,\mathrm{d}s+C_{\varepsilon}\int_0^{t}(1+\|\mathbf{u}_n\|_{\mathbb{H}^{1}}^2\|\mathbf{u}_n\|_{\mathbb{H}^{2}}^2+\|\mathbf{u}\|_{\mathbb{H}^{1}}^2\|\mathbf{u}\|_{\mathbb{H}^{2}}^2)\|X_n\|_{\mathbb{L}^2}^2\,\mathrm{d}s
\end{split}
\end{equation}
Using the fact $(f\times g,g)_{\mathbb{L}^2}=0$, we infer from \eqref{sys666-14} that
\begin{equation}\label{sys666-17}
\begin{split}
&|M_4(t)|\leq C\int_0^{t}\|\mathbf{u}_n\|_{\mathbb{L}^{\infty}}\|\nabla X_n\|_{\mathbb{L}^2}\|X_n\|_{\mathbb{L}^2}\,\mathrm{d}s\\
&\leq \varepsilon\int_0^{t}\|\Delta X_n\|_{\mathbb{L}^2}^2\,\mathrm{d}s+C_{\varepsilon}\int_0^{t}\|\mathbf{u}_n\|_{\mathbb{H}^{2}}^2\|X_n\|_{\mathbb{L}^2}^2\,\mathrm{d}s.
\end{split}
\end{equation}
By using Corollary \ref{--cor2-2-2}, we have
\begin{equation}\label{sys666-17-1}
\begin{split}
&|M_5(t)|\leq C\int_0^{t}\|X_n\|_{\mathbb{L}^2}^2\,\mathrm{d}s.
\end{split}
\end{equation}
Similarly, we have
\begin{equation}\label{sys666-18}
\begin{split}
&|M_6(t)|\leq C\int_0^t\|X_n\|_{\mathbb{L}^2}^2\left(\int_{\mathbb{B}}|l||\theta_n(s,l)-1|\nu(\mathrm{d}l)\right)\mathrm{d}s\\
&+C(1+\sup_{s\in[0,t]}\|\mathbf{u}(s)\|_{\mathbb{L}^2})\int_0^t\|X_n\|_{\mathbb{L}^2}\left(\int_{\mathbb{B}}|l|(|\theta_n(s,l)-1|+|\theta(s,l)-1|)\nu(\mathrm{d}l)\right)\mathrm{d}s\\
&\leq C\int_0^t\|X_n\|_{\mathbb{L}^2}^2\left(\int_{\mathbb{B}}|l||\theta_n(s,l)-1|\nu(\mathrm{d}l)\right)\mathrm{d}s\\
&+C(\mathbf{c}_0)\sup_{f\in S^K}\int_0^t\|X_n\|_{\mathbb{L}^2}\left(\int_{\mathbb{B}}|l||f(s,l)-1|\nu(\mathrm{d}l)\right)\mathrm{d}s,
\end{split}
\end{equation}
where $\mathbf{c}_0$ is defined by \eqref{sys666-31}. Now, plugging \eqref{sys666-14}-\eqref{sys666-18} into \eqref{sys666-12}, choosing $\varepsilon$ small enough and then using the elliptic regularity result, we infer that
\begin{equation}\label{sys666-19}
\begin{split}
&\|X_n(t)\|_{\mathbb{L}^2}^2+\int_0^{t}\|X_n\|_{\mathbb{H}^2}^2\,\mathrm{d}s\\
&\leq C(\mathbf{c}_0)\sup_{f\in S^K}\int_0^t\|X_n\|_{\mathbb{L}^2}\left(\int_{\mathbb{B}}|l||f(s,l)-1|\nu(\mathrm{d}l)\right)\mathrm{d}s\\
&+C\int_0^t\|X_n\|_{\mathbb{L}^2}^2\left(\int_{\mathbb{B}}|l||\theta_n(s,l)-1|\nu(\mathrm{d}l)+1+\|\mathbf{u}_n\|_{\mathbb{H}^{1}}^2\|\mathbf{u}_n\|_{\mathbb{H}^{2}}^2+\|\mathbf{u}\|_{\mathbb{H}^{1}}^2\|\mathbf{u}\|_{\mathbb{H}^{2}}^2+\|\mathbf{u}_n\|_{\mathbb{H}^{2}}^2+\|\mathbf{u}\|_{\mathbb{H}^{2}}^2\right)\,\mathrm{d}s\\
&:=C(\mathbf{c}_0)\sup_{f\in S^K}\int_0^t\|X_n\|_{\mathbb{L}^2}\left(\int_{\mathbb{B}}|l||f(s,l)-1|\nu(\mathrm{d}l)\right)\mathrm{d}s+C\int_0^t\|X_n\|_{\mathbb{L}^2}^2\mathbf{V}(s)\,\mathrm{d}s.
\end{split}
\end{equation}

Moreover, we have
\begin{equation}\label{sys666-21}
\begin{split}
&|N_1(t)|\leq \varepsilon\int_0^{t}\|\nabla\Delta X_n\|_{\mathbb{L}^2}^2\,\mathrm{d}s+C_{\varepsilon}\int_0^{t}\|\nabla X_n\|_{\mathbb{L}^2}^2\,\mathrm{d}s.
\end{split}
\end{equation}
Since $\mathbb{H}^2\hookrightarrow L^{\infty}$,
\begin{equation}\label{sys666-22}
\begin{split}
&|N_2(t)|\\
&\leq C\int_0^{t}(1+\|\mathbf{u}_n\|_{\mathbb{L}^{\infty}}^2)\|\nabla X_n\|_{\mathbb{L}^2}^2\,\mathrm{d}s+C\int_0^{t}(\|\mathbf{u}_n\|_{\mathbb{L}^{\infty}}+\|\mathbf{u}\|_{\mathbb{L}^{\infty}})\|\nabla \mathbf{u}\|_{\mathbb{L}^2}\|X_n\|_{\mathbb{L}^{\infty}}\|\nabla X_n\|_{\mathbb{L}^2}\,\mathrm{d}s\\
&\leq \varepsilon\int_0^{t}\|X_n\|_{\mathbb{L}^{\infty}}^2\,\mathrm{d}s+C_{\varepsilon}\int_0^{t}(1+\|\mathbf{u}_n\|_{\mathbb{L}^{\infty}}^2+\|\mathbf{u}_n\|_{\mathbb{L}^{\infty}}^2\|\nabla \mathbf{u}\|_{\mathbb{L}^2}^2+\|\mathbf{u}\|_{\mathbb{L}^{\infty}}^2\|\nabla \mathbf{u}\|_{\mathbb{L}^2}^2)\|\nabla X_n\|_{\mathbb{L}^2}^2\,\mathrm{d}s\\
&\leq \varepsilon\int_0^{t}\|X_n\|_{\mathbb{H}^{2}}^2\,\mathrm{d}s+C_{\varepsilon}\int_0^{t}(1+\|\mathbf{u}_n\|_{\mathbb{H}^{2}}^2+\|\mathbf{u}_n\|_{\mathbb{H}^{2}}^2\|\mathbf{u}\|_{\mathbb{H}^1}^2+\|\mathbf{u}\|_{\mathbb{H}^{2}}^2\|\mathbf{u}\|_{\mathbb{H}^1}^2)\|\nabla X_n\|_{\mathbb{L}^2}^2\,\mathrm{d}s
\end{split}
\end{equation}
Additionally, by $\mathbb{H}^1\hookrightarrow \mathbb{L}^{6}$ as well as \eqref{addd1}
\begin{equation}\label{sys666-23}
\begin{split}
&|N_3(t)|\\
&\leq C\int_0^{t}\|\mathbf{u}_n\|_{\mathbb{L}^{\infty}}^2\|\nabla X_n\|_{\mathbb{L}^2}\|\nabla\Delta X_n\|_{\mathbb{L}^2}\,\mathrm{d}s+C\int_0^{t}(\|\mathbf{u}_n\|_{\mathbb{L}^{3}}+\|\mathbf{u}\|_{\mathbb{L}^{3}})\|\nabla \mathbf{u}\|_{\mathbb{L}^{\infty}}\|X_n\|_{\mathbb{L}^{6}}\|\nabla\Delta X_n\|_{\mathbb{L}^2}\,\mathrm{d}s\\
&\leq\varepsilon\int_0^{t}\|\nabla\Delta X_n\|_{\mathbb{L}^2}^2\,\mathrm{d}s+C_{\varepsilon}\int_0^{t}\|\mathbf{u}_n\|_{\mathbb{L}^{\infty}}^4\|\nabla X_n\|_{\mathbb{L}^2}^2\,\mathrm{d}s+C_{\varepsilon}\int_0^{t}(\|\mathbf{u}_n\|_{\mathbb{H}^{1}}^2\|\mathbf{u}\|_{\mathbb{H}^3}^2+\|\mathbf{u}\|_{\mathbb{H}^{1}}^2\|\mathbf{u}\|_{\mathbb{H}^3}^2)\|X_n\|_{\mathbb{H}^1}^2\,\mathrm{d}s\\
&\leq\varepsilon\int_0^{t}\|\nabla\Delta X_n\|_{\mathbb{L}^2}^2\,\mathrm{d}s+C_{\varepsilon}\int_0^{t}(\|\mathbf{u}_n\|_{\mathbb{H}^{1}}^2\|\mathbf{u}_n\|_{\mathbb{H}^{2}}^2+\|\mathbf{u}_n\|_{\mathbb{H}^{1}}^2\|\mathbf{u}\|_{\mathbb{H}^3}^2+\|\mathbf{u}\|_{\mathbb{H}^{1}}^2\|\mathbf{u}\|_{\mathbb{H}^3}^2)\|X_n\|_{\mathbb{H}^1}^2\,\mathrm{d}s.
\end{split}
\end{equation}
Moreover, by using \eqref{sys666-21}, we have
\begin{equation}\label{sys666-24}
\begin{split}
&|N_4(t)|\leq C\int_0^{t}\|\Delta\mathbf{u}\|_{\mathbb{L}^{3}}\|X_n\|_{\mathbb{L}^6}\|\Delta X_n\|_{\mathbb{L}^2}\,\mathrm{d}s\\
&\leq \varepsilon\int_0^{t}\|\Delta X_n\|_{\mathbb{L}^2}^2\,\mathrm{d}s+C_{\varepsilon}\int_0^{t}\|\mathbf{u}\|_{\mathbb{H}^{3}}^2\|X_n\|_{\mathbb{H}^1}^2\,\mathrm{d}s\\
&\leq\varepsilon\int_0^{t}\|\nabla\Delta X_n\|_{\mathbb{L}^2}^2\,\mathrm{d}s+C_{\varepsilon}\int_0^{t}(1+\|\mathbf{u}\|_{\mathbb{H}^{3}}^2)\|X_n\|_{\mathbb{H}^1}^2\,\mathrm{d}s.
\end{split}
\end{equation}
By using Corollary \ref{--cor2-2-2}, we have
\begin{equation}\label{sys666-24-1}
\begin{split}
&|N_5(t)|\leq C\int_0^{t}\|X_n\|_{\mathbb{H}^1}^2\,\mathrm{d}s.
\end{split}
\end{equation}
And we also have
\begin{equation}\label{sys666-25}
\begin{split}
&|N_6(t)|\leq C\int_0^t\|\nabla X_n\|_{\mathbb{L}^2}^2\left(\int_{\mathbb{B}}|l||\theta_n(s,l)-1|\nu(\mathrm{d}l)\right)\mathrm{d}s\\
&+C(1+\sup_{s\in[0,t]}\|\nabla\mathbf{u}(s)\|_{\mathbb{L}^2})\int_0^t\|\nabla X_n\|_{\mathbb{L}^2}\left(\int_{\mathbb{B}}|l|(|\theta_n(s,l)-1|+|\theta(s,l)-1|)\nu(\mathrm{d}l)\right)\mathrm{d}s\\
&\leq C\int_0^t\|\nabla X_n\|_{\mathbb{L}^2}^2\left(\int_{\mathbb{B}}|l||\theta_n(s,l)-1|\nu(\mathrm{d}l)\right)\mathrm{d}s\\
&+C(\mathbf{c}_0)\sup_{f\in S^K}\int_0^t\|\nabla X_n\|_{\mathbb{L}^2}\left(\int_{\mathbb{B}}|l||f(s,l)-1|\nu(\mathrm{d}l)\right)\mathrm{d}s.
\end{split}
\end{equation}
Plugging \eqref{sys666-21}-\eqref{sys666-25} into \eqref{sys666-13} and choosing $\varepsilon$ small enough, we infer that
\begin{equation}\label{sys666-26}
\begin{split}
&\|\nabla X_n(t)\|_{\mathbb{L}^2}^2+\int_0^{t}\|\nabla\Delta X_n\|_{\mathbb{L}^2}^2\,\mathrm{d}s-\frac{1}{2}\int_0^{t}\|X_n\|_{\mathbb{H}^{2}}^2\,\mathrm{d}s\\
&\leq C(\mathbf{c}_0)\sup_{f\in S^K}\int_0^t\|\nabla X_n\|_{\mathbb{L}^2}\left(\int_{\mathbb{B}}|l||f(s,l)-1|\nu(\mathrm{d}l)\right)\mathrm{d}s+C\int_0^t\|X_n\|_{\mathbb{H}^1}^2\Biggl(\int_{\mathbb{B}}|l||\theta_n(s,l)-1|\nu(\mathrm{d}l)\\
&+1+\|\mathbf{u}_n\|_{\mathbb{H}^{2}}^2+\|\mathbf{u}_n\|_{\mathbb{H}^{2}}^2\|\mathbf{u}\|_{\mathbb{H}^1}^2+\|\mathbf{u}\|_{\mathbb{H}^{3}}^2+\|\mathbf{u}_n\|_{\mathbb{H}^{1}}^2\|\mathbf{u}_n\|_{\mathbb{H}^{2}}^2+\|\mathbf{u}_n\|_{\mathbb{H}^{1}}^2\|\mathbf{u}\|_{\mathbb{H}^3}^2+\|\mathbf{u}\|_{\mathbb{H}^{1}}^2\|\mathbf{u}\|_{\mathbb{H}^3}^2\Biggl)\,\mathrm{d}s\\
&:=C(\mathbf{c}_0)\sup_{f\in S^K}\int_0^t\|\nabla X_n\|_{\mathbb{L}^2}\left(\int_{\mathbb{B}}|l||f(s,l)-1|\nu(\mathrm{d}l)\right)\mathrm{d}s+C\int_0^t\|X_n\|_{\mathbb{H}^1}^2\mathbf{V}_1(s)\,\mathrm{d}s.
\end{split}
\end{equation}
Now combining \eqref{sys666-19} and \eqref{sys666-26} and using the elliptic regularity result, we infer that
\begin{equation}\label{sys666-27}
\begin{split}
&\|X_n(t)\|_{\mathbb{H}^1}^2+\int_0^{t}\|X_n\|_{\mathbb{H}^3}^2\,\mathrm{d}s\\
&\leq C(\mathbf{c}_0)\sup_{f\in S^K}\int_0^t\|X_n\|_{\mathbb{H}^1}\left(\int_{\mathbb{B}}|l||f(s,l)-1|\nu(\mathrm{d}l)\right)\mathrm{d}s+C\int_0^t\|X_n\|_{\mathbb{H}^1}^2(\mathbf{V}(s)+\mathbf{V}_1(s))\,\mathrm{d}s.
\end{split}
\end{equation}
Noting that
\begin{equation*}
\begin{split}
&\int_0^T\mathbf{V}(t)+\mathbf{V}_1(t)\,\mathrm{d}t<\infty,
\end{split}
\end{equation*}
thus by applying the Gronwall lemma to \eqref{sys666-27}, we infer that
\begin{equation}\label{sys666-28}
\begin{split}
&\sup_{t\in[0,T]}\|X_n(t)\|_{\mathbb{H}^1}^2+\int_0^{T}\|X_n\|_{\mathbb{H}^3}^2\,\mathrm{d}s\\
&\leq C(\mathbf{c}_0,T)\sup_{f\in S^K}\int_0^T\|X_n\|_{\mathbb{H}^1}\left(\int_{\mathbb{B}}|l||f(s,l)-1|\nu(\mathrm{d}l)\right)\mathrm{d}s.
\end{split}
\end{equation}
Finally, by using the result \eqref{sys666-10}, we derive from \eqref{sys666-28} that
\begin{equation*}\label{sys666-29}
\begin{split}
&\lim_{n\rightarrow\infty}\left(\sup_{t\in[0,T]}\|X_n(t)\|_{\mathbb{H}^1}^2+\int_0^{T}\|X_n\|_{\mathbb{H}^3}^2\,\mathrm{d}s\right)=0.
\end{split}
\end{equation*}
The proof is thus complete.\end{proof}

\subsubsection{Verification of Condition 2}\label{subsubsec6-4-1} Let us recall the statement of Condition 2. Let $\{\varepsilon_n\}_{n\in\mathbb{N}}$ be a $(0,1]$-valued sequence converging to $0$. For all $K\in\mathbb{N}$, let $\varphi_{\varepsilon_n},~\varphi\in S^K$ be such that $\varphi_{\varepsilon_n}$ converges in law to $\varphi$ as $\varepsilon_n\rightarrow 0$. For simplicity, we denote
\begin{equation*}\label{sys66666-1}
\begin{split}
&Y_n:=\mathcal{J}^{\varepsilon_n}(\varepsilon_n\eta^{\varepsilon_n^{-1}\varphi_{\varepsilon_n}}),~y_n:=\mathcal{J}^{0}(\varphi_{\varepsilon_n}).
\end{split}
\end{equation*}
By Lemma \ref{the66-1}, $Y_n$ has the same law with the unique solution $\mathbf{u}^{\varepsilon,\varphi}$ of the stochastic control equation \eqref{sys61-2}. In particular, there exists a constant $C>0$ independent of $n$ such that
\begin{equation}\label{adddd}
\begin{split}
&\sup_{n\in\mathbb{N}}\mathbb{E}\left(\sup_{t\in[0,T]}\|Y_n(t)\|_{\mathbb{H}^1}^2+\int_0^{T}\|Y_n(t)\|_{\mathbb{H}^3}^2\,\mathrm{d}t\right)\leq C.
\end{split}
\end{equation}
Let $n,~R\in\mathbb{N}$. Let us define a stopping time $\tau^R_n$ as follows:
\begin{equation*}\label{sys66666-2}
\begin{split}
&\tau^R_n:=\inf\left\{t:\sup_{s\in[0,t]}\|Y_n(s)\|_{\mathbb{H}^1}+\int_0^t\|Y_n(s)\|_{\mathbb{H}^3}^2\,\mathrm{d}s+\sup_{s\in[0,t]}\|y_n(s)\|_{\mathbb{H}^1}+\int_0^t\|y_n(s)\|_{\mathbb{H}^3}^2\,\mathrm{d}s>R\right\}\wedge T.
\end{split}
\end{equation*}
For convenience, we denote $\tau^R_n$ by $\tau_n$. Now we state the following result.
\begin{lemma}\label{lem666-1*}  Let $\tau_n$ be defined as above, then we have
\begin{equation}\label{sys66666-3}
\begin{split}
\lim_{n\rightarrow\infty}\mathbb{E}\left(\sup_{t\in[0,T\wedge\tau_n]}\|Y_n(t)-y_n(t)\|_{\mathbb{H}^1}^2+\int_0^{T\wedge\tau_n}\|Y_n(t)-y_n(t)\|_{\mathbb{H}^3}^2\,\mathrm{d}t\right)=0.
\end{split}
\end{equation}
\end{lemma}
\begin{proof}[\emph{\textbf{Proof}}] We use some notations defined in \eqref{def111}, then $Y_n$ is a solution to the equation
\begin{equation*}\label{sys66666-4}
\begin{split}
&\mathrm{d}Y_n(t)=\sum_{j=1}^5F^j(Y_n(t))\mathrm{d}t+\mathbf{b}(Y_n(t))\,\mathrm{d}t+\varepsilon_n\int_{\mathbb{B}}\mathbf{G}(l,Y_n(t))\tilde{\eta}^{\varepsilon^{-1}_n\varphi_{\varepsilon_n}}(\mathrm{d}t,\mathrm{d}l)\\
&+\int_{\mathbb{B}}\mathbf{G}(l,Y_n(t))(\varphi_{\varepsilon_n}(t,l)-1)\nu(\mathrm{d}l)\mathrm{d}t,
\end{split}
\end{equation*}
with $Y_n(0)=\mathbf{u}_0$. Similarly, $y_n$ is a solution to the equation
\begin{equation*}\label{sys66666-5}
\begin{split}
&\mathrm{d}y_n(t)=\sum_{j=1}^5F^j(y_n(t))\mathrm{d}t+\mathbf{b}(y_n(t))\,\mathrm{d}t+\int_{\mathbb{B}}\mathbf{G}(l,y_n(t))(\varphi_{\varepsilon_n}(t,l)-1)\nu(\mathrm{d}l)\mathrm{d}t,
\end{split}
\end{equation*}
with $y_n(0)=\mathbf{u}_0$. Let $Z_n(t):=Y_n(t)-y_n(t)$, then for $n\in\mathbb{N}$
\begin{equation}\label{sys66666-6}
\begin{split}
&\mathrm{d}Z_n(t)=\sum_{j=1}^5[F^j(Y_n(t))-F^j(y_n(t))]\mathrm{d}t+[\mathbf{b}(Y_n(t))-\mathbf{b}(y_n(t))]\,\mathrm{d}t\\
&+\int_{\mathbb{B}}[\mathbf{G}(l,Y_n(t))-\mathbf{G}(l,y_n(t))](\varphi_{\varepsilon_n}(t,l)-1)\nu(\mathrm{d}l)\mathrm{d}t+\varepsilon_n\int_{\mathbb{B}}\mathbf{G}(l,Y_n(t))\tilde{\eta}^{\varepsilon^{-1}_n\varphi_{\varepsilon_n}}(\mathrm{d}t,\mathrm{d}l)
\end{split}
\end{equation}
with $Z_n(0)=0$. Applying It\^{o}'s formula to $\|Z_n(t)\|_{\mathbb{L}^2}^2$ and $\|\nabla Z_n(t)\|_{\mathbb{L}^2}^2$ respectively, and integrating by parts, we infer that
\begin{equation}\label{sys66666-7}
\begin{split}
&\|Z_n(t)\|_{\mathbb{L}^2}^2+2\int_0^{t}\|\Delta Z_n\|_{\mathbb{L}^2}^2\,\mathrm{d}s\\
&=2\int_0^{t}\|\nabla Z_n\|_{\mathbb{L}^2}^2\,\mathrm{d}s+4\int_0^{t}\left(Z_n-|Y_n|^2Z_n+y_n\left(|y_n|^2-|Y_n|^2\right),Z_n\right)_{\mathbb{L}^2}\,\mathrm{d}s\\
&+4\int_0^{t}\left(|Y_n|^2Y_n-|y_n|^2y_n,\Delta Z_n\right)_{\mathbb{L}^2}\,\mathrm{d}s+2\int_0^{t}\left(Y_n\times\nabla Y_n-y_n\times\nabla y_n,\nabla Z_n\right)_{\mathbb{L}^2}\,\mathrm{d}s\\
&+2\int_0^{t}\left(\mathbf{b}(Y_n)-\mathbf{b}(y_n),Z_n\right)_{\mathbb{L}^2}\,\mathrm{d}s+2\int_0^t\int_{\mathbb{B}}(\varphi_{\varepsilon_n}(s,l)-1)\left(\mathbf{G}(l,Y_n)-\mathbf{G}(l,y_n),Z_n\right)_{\mathbb{L}^2}\nu(\mathrm{d}l)\mathrm{d}s\\
&+2\varepsilon_n\int_0^t\int_{\mathbb{B}}\left(\mathbf{G}(l,Y_n),Z_n\right)_{\mathbb{L}^2}\tilde{\eta}^{\varepsilon^{-1}_n\varphi_{\varepsilon_n}}(\mathrm{d}s,\mathrm{d}l)+\varepsilon_n^2\int_0^t\int_{\mathbb{B}}\|\mathbf{G}(l,Y_n)\|_{\mathbb{L}^2}^2\eta^{\varepsilon^{-1}_n\varphi_{\varepsilon_n}}(\mathrm{d}s,\mathrm{d}l)\\
&:=\sum_{j=1}^{8}\mathbf{I}_j(t),
\end{split}
\end{equation}
and
\begin{equation}\label{sys66666-8}
\begin{split}
&\|\nabla Z_n(t)\|_{\mathbb{L}^2}^2+2\int_0^{t}\|\nabla\Delta Z_n\|_{\mathbb{L}^2}^2\,\mathrm{d}s\\
&=2\int_0^{t}\|\Delta Z_n\|_{\mathbb{L}^2}^2\,\mathrm{d}s+4\int_0^{t}\left(\nabla(Z_n-|Y_n|^2Z_n+y_n\left(|y_n|^2-|Y_n|^2\right)),\nabla Z_n\right)_{\mathbb{L}^2}\,\mathrm{d}s\\
&+4\int_0^{t}\left(\nabla(|Y_n|^2Y_n-|y_n|^2y_n),\nabla\Delta Z_n\right)_{\mathbb{L}^2}\,\mathrm{d}s+2\int_0^{t}\left(Y_n\times\Delta Y_n-y_n\times\Delta y_n,\Delta Z_n\right)_{\mathbb{L}^2}\,\mathrm{d}s\\
&+2\int_0^{t}\left(\nabla(\mathbf{b}(Y_n)-\mathbf{b}(y_n)),\nabla Z_n\right)_{\mathbb{L}^2}\,\mathrm{d}s+2\int_0^t\int_{\mathbb{B}}(\varphi_{\varepsilon_n}(s,l)-1)\left(\nabla(\mathbf{G}(l,Y_n)-\mathbf{G}(l,y_n)),\nabla Z_n\right)_{\mathbb{L}^2}\nu(\mathrm{d}l)\mathrm{d}s\\
&+2\varepsilon_n\int_0^t\int_{\mathbb{B}}\left(\nabla\mathbf{G}(l,Y_n),\nabla Z_n\right)_{\mathbb{L}^2}\tilde{\eta}^{\varepsilon^{-1}_n\varphi_{\varepsilon_n}}(\mathrm{d}s,\mathrm{d}l)+\varepsilon_n^2\int_0^t\int_{\mathbb{B}}\|\nabla\mathbf{G}(l,Y_n)\|_{\mathbb{L}^2}^2\eta^{\varepsilon^{-1}_n\varphi_{\varepsilon_n}}(\mathrm{d}s,\mathrm{d}l)\\
&:=\sum_{j=1}^{8}\mathbf{J}_j(t).
\end{split}
\end{equation}

Using the estimates \eqref{sys666-14}-\eqref{sys666-17-1} from the proof of Lemma \ref{lem666-1}, it is not difficult to obtain that for some small $\varepsilon^*>0$
\begin{equation}\label{sys66666-9}
\begin{split}
&\sum_{j=1}^5|\mathbf{I}_j(t)|\leq \varepsilon^*\int_0^t\|\Delta Z_n\|_{\mathbb{L}^2}^2\\
&+C_{\varepsilon^*}\int_0^t\|Z_n\|_{\mathbb{L}^2}^2\left(1+\|Y_n\|_{\mathbb{H}^{1}}^2\|Y_n\|_{\mathbb{H}^{2}}^2+\|y_n\|_{\mathbb{H}^{1}}^2\|y_n\|_{\mathbb{H}^{2}}^2+\|Y_n\|_{\mathbb{H}^{2}}^2+\|y_n\|_{\mathbb{H}^{2}}^2\right)\,\mathrm{d}s.
\end{split}
\end{equation}
Using the linear property of $\mathbf{G}$, we have
\begin{equation}\label{sys66666-10}
\begin{split}
&|\mathbf{I}_6(t)|\leq C\int_0^t\|Z_n\|_{\mathbb{L}^2}^2\left(\int_{\mathbb{B}}|l||\varphi_{\varepsilon_n}(s,l)-1|\nu(\mathrm{d}l)\right)\,\mathrm{d}s.
\end{split}
\end{equation}
Plugging \eqref{sys66666-9} and \eqref{sys66666-10} into \eqref{sys66666-7} and choosing $\varepsilon^*$ small enough, we have
\begin{equation}\label{sys66666-11}
\begin{split}
&\|Z_n(t)\|_{\mathbb{L}^2}^2+\int_0^{t}\|Z_n\|_{\mathbb{H}^2}^2\,\mathrm{d}s\\
&\leq C\int_0^t\|Z_n\|_{\mathbb{L}^2}^2\left(\int_{\mathbb{B}}|l||\varphi_{\varepsilon_n}(s,l)-1|\nu(\mathrm{d}l)+1+\|Y_n\|_{\mathbb{H}^{1}}^2\|Y_n\|_{\mathbb{H}^{2}}^2+\|y_n\|_{\mathbb{H}^{1}}^2\|y_n\|_{\mathbb{H}^{2}}^2+\|Y_n\|_{\mathbb{H}^{2}}^2+\|y_n\|_{\mathbb{H}^{2}}^2\right)\,\mathrm{d}s\\
&+\sum_{j=7}^{8}\mathbf{I}_j(t)\\
&:=C\int_0^t\|Z_n\|_{\mathbb{L}^2}^2\mathbf{W}(s)\,\mathrm{d}s+\sum_{j=7}^{8}\mathbf{I}_j(t).
\end{split}
\end{equation}

Similarly, using the estimates \eqref{sys666-21}-\eqref{sys666-24-1} from the proof of Lemma \ref{lem666-1}, we infer that for some small $\varepsilon^*>0$
\begin{equation}\label{sys66666-12}
\begin{split}
&\sum_{j=1}^5|\mathbf{J}_j(t)|\leq \varepsilon^*\int_0^t\|Z_n\|_{\mathbb{H}^2}^2\,\mathrm{d}s+\varepsilon^*\int_0^t\|\nabla\Delta Z_n\|_{\mathbb{L}^2}^2\,\mathrm{d}s\\
&+C_{\varepsilon^*}\int_0^t\|Z_n\|_{\mathbb{H}^1}^2\Bigl(1+\|Y_n\|_{\mathbb{H}^{2}}^2+\|Y_n\|_{\mathbb{H}^{2}}^2\|y_n\|_{\mathbb{H}^1}^2+\|y_n\|_{\mathbb{H}^{3}}^2+\|Y_n\|_{\mathbb{H}^{1}}^2\|Y_n\|_{\mathbb{H}^{2}}^2\\
&+\|Y_n\|_{\mathbb{H}^{1}}^2\|y_n\|_{\mathbb{H}^3}^2+\|y_n\|_{\mathbb{H}^{1}}^2\|y_n\|_{\mathbb{H}^3}^2\Bigl)\,\mathrm{d}s.
\end{split}
\end{equation}
According to Corollary \ref{--cor2-2-2}, we have
\begin{equation}\label{sys66666-13}
\begin{split}
&|\mathbf{J}_6(t)|\leq C\int_0^t\|Z_n\|_{\mathbb{H}^2}^2\left(\int_{\mathbb{B}}|l||\varphi_{\varepsilon_n}(s,l)-1|\nu(\mathrm{d}l)\right)\,\mathrm{d}s.
\end{split}
\end{equation}
Plugging \eqref{sys66666-12} and \eqref{sys66666-13} into \eqref{sys66666-8} and choosing $\varepsilon^*$ small enough, we have
\begin{equation}\label{sys66666-14}
\begin{split}
&\|\nabla Z_n(t)\|_{\mathbb{L}^2}^2+\int_0^{t}\|\nabla\Delta Z_n\|_{\mathbb{L}^2}^2\,\mathrm{d}s-\frac{1}{2}\int_0^{t}\| Z_n\|_{\mathbb{H}^2}^2\,\mathrm{d}s\\
&\leq C\int_0^t\|Z_n\|_{\mathbb{H}^1}^2\Bigl(\int_{\mathbb{B}}|l||\varphi_{\varepsilon_n}(s,l)-1|\nu(\mathrm{d}l)+1+\|Y_n\|_{\mathbb{H}^{2}}^2+\|Y_n\|_{\mathbb{H}^{2}}^2\|y_n\|_{\mathbb{H}^1}^2+\|y_n\|_{\mathbb{H}^{3}}^2\\
&+\|Y_n\|_{\mathbb{H}^{1}}^2\|Y_n\|_{\mathbb{H}^{2}}^2+\|Y_n\|_{\mathbb{H}^{1}}^2\|y_n\|_{\mathbb{H}^3}^2+\|y_n\|_{\mathbb{H}^{1}}^2\|y_n\|_{\mathbb{H}^3}^2\Bigl)\,\mathrm{d}s+\sum_{j=7}^{8}\mathbf{J}_j(t)\\
&:=C\int_0^t\|Z_n\|_{\mathbb{H}^1}^2\mathbf{W}_1(s)\,\mathrm{d}s+\sum_{j=7}^{8}\mathbf{J}_j(t).
\end{split}
\end{equation}

Moreover combining \eqref{sys66666-11} and \eqref{sys66666-14} and using the elliptic regularity result, we infer that
\begin{equation*}\label{sys66666-15}
\begin{split}
&\sup_{t\in[0,T\wedge\tau_n]}\|Z_n(t)\|_{\mathbb{H}^1}^2+\int_0^{T\wedge\tau_n}\|Z_n(s)\|_{\mathbb{H}^3}^2\,\mathrm{d}s\\
&\leq C\int_0^{T\wedge\tau_n}\|Z_n(s)\|_{\mathbb{H}^1}^2(\mathbf{W}(s)+\mathbf{W}_1(s))\,\mathrm{d}s+\sum_{j=7}^{8}\sup_{t\in[0,T\wedge\tau_n]}\mathbf{I}_j(t)+\sum_{j=7}^{8}\sup_{t\in[0,T\wedge\tau_n]}\mathbf{J}_j(t).
\end{split}
\end{equation*}
Noting that
\begin{equation*}\label{sys66666-15-1}
\begin{split}
&\int_0^{T\wedge\tau_n}\mathbf{W}(s)+\mathbf{W}_1(s)\,\mathrm{d}s\leq C_R<\infty,
\end{split}
\end{equation*}
we can use the Gronwall lemma to infer that
\begin{equation}\label{sys66666-15-2}
\begin{split}
&\sup_{t\in[0,T\wedge\tau_n]}\|Z_n(t)\|_{\mathbb{H}^1}^2+\int_0^{T\wedge\tau_n}\|Z_n(s)\|_{\mathbb{H}^3}^2\,\mathrm{d}s\\
&\leq\left(\sum_{j=7}^{8}\sup_{t\in[0,T\wedge\tau_n]}|\mathbf{I}_j(t)|+\sum_{j=7}^{8}\sup_{t\in[0,T\wedge\tau_n]}|\mathbf{J}_j(t)|\right)e^{\int_0^{T\wedge\tau_n}\mathbf{W}(s)+\mathbf{W}_1(s)\,\mathrm{d}s}\\
&\leq C_R\left(\sum_{j=7}^{8}\sup_{t\in[0,T\wedge\tau_n]}|\mathbf{I}_j(t)|+\sum_{j=7}^{8}\sup_{t\in[0,T\wedge\tau_n]}|\mathbf{J}_j(t)|\right).
\end{split}
\end{equation}

The BDG inequality and the linear growth property of $\mathbf{G}$ imply that
\begin{equation}\label{sys66666-16}
\begin{split}
&\mathbb{E}\sup_{t\in[0,T\wedge\tau_n]}|\mathbf{I}_7(t)|\leq C\varepsilon_n\mathbb{E}\left(\int_0^{T\wedge\tau_n}\int_{\mathbb{B}}\varepsilon_n^{-1}\varphi_{\varepsilon_n}(s,l)\|\mathbf{G}(l,Y_n)\|_{\mathbb{L}^2}^2\|Z_n\|_{\mathbb{L}^2}^2\nu(\mathrm{d}l)\mathrm{d}s\right)^{\frac{1}{2}}\\
&\leq C\varepsilon_n^{\frac{1}{2}}\mathbb{E}\left(\sup_{s\in[0,T\wedge\tau_n]}\left(1+\|Y_n(s)\|_{\mathbb{L}^2}^2\right)\int_0^{T\wedge\tau_n}\|Z_n\|_{\mathbb{L}^2}^2\int_{\mathbb{B}}|l|^2\varphi_{\varepsilon_n}(s,l)\nu(\mathrm{d}l)\mathrm{d}s\right)^{\frac{1}{2}}\\
&\leq C_R\varepsilon_n^{\frac{1}{2}}+C\varepsilon_n^{\frac{1}{2}}\mathbb{E}\left(\int_0^{T\wedge\tau_n}\|Z_n\|_{\mathbb{L}^2}^2\int_{\mathbb{B}}|l|^2\varphi_{\varepsilon_n}(s,l)\nu(\mathrm{d}l)\mathrm{d}s\right).
\end{split}
\end{equation}
Additionally we infer that
\begin{equation}\label{sys66666-17}
\begin{split}
&\mathbb{E}\sup_{t\in[0,T\wedge\tau_n]}|\mathbf{I}_8(t)|\leq \varepsilon_n^2\mathbb{E}\left(\sup_{t\in[0,T\wedge\tau_n]}\left|\int_0^t\int_{\mathbb{B}}\|\mathbf{G}(l,Y_n)\|_{\mathbb{L}^2}^2\tilde{\eta}^{\varepsilon^{-1}_n\varphi_{\varepsilon_n}}(\mathrm{d}s,\mathrm{d}l)\right|\right)\\
&+\varepsilon_n^2\mathbb{E}\left(\sup_{t\in[0,T\wedge\tau_n]}\left|\int_0^t\int_{\mathbb{B}}\|\mathbf{G}(l,Y_n)\|_{\mathbb{L}^2}^2\nu^{\varepsilon^{-1}_n\varphi_{\varepsilon_n}}(\mathrm{d}l)\mathrm{d}s\right|\right)\\
&\leq C\varepsilon_n^2\mathbb{E}\left(\int_0^{T\wedge\tau_n}\int_{\mathbb{B}}\varepsilon_n^{-1}\varphi_{\varepsilon_n}(s,l)\|\mathbf{G}(l,Y_n)\|_{\mathbb{L}^2}^4\nu(\mathrm{d}l)\mathrm{d}s\right)^{\frac{1}{2}}\\
&+\varepsilon_n\mathbb{E}\left(\int_0^{T\wedge\tau_n}\int_{\mathbb{B}}\varphi_{\varepsilon_n}(s,l)\|\mathbf{G}(l,Y_n)\|_{\mathbb{L}^2}^2\nu(\mathrm{d}l)\mathrm{d}s\right)\\
&\leq C_R\varepsilon_n^{\frac{3}{2}}\int_0^{T\wedge\tau_n}\int_{\mathbb{B}}|l|^4\varphi_{\varepsilon_n}(s,l)\nu(\mathrm{d}l)\mathrm{d}s+C_R\varepsilon_n\int_0^{T\wedge\tau_n}\int_{\mathbb{B}}|l|^2\varphi_{\varepsilon_n}(s,l)\nu(\mathrm{d}l)\mathrm{d}s\\
&\leq C_R(\varepsilon_n^{\frac{3}{2}}+\varepsilon_n).
\end{split}
\end{equation}
Here for the above estimate, we have use the fact (see Lemma 3.4 in \cite{budhiraja2013large}) that
\begin{equation*}\label{sys66666-18}
\begin{split}
&\sup_{f\in S^N}\int_0^{T}\int_{\mathbb{B}}(|l|^2+|l|^4)f(s,l)\nu(\mathrm{d}l)\mathrm{d}s<\infty.
\end{split}
\end{equation*}
Similarly, we have
\begin{equation}\label{sys66666-19}
\begin{split}
&\mathbb{E}\sup_{t\in[0,T\wedge\tau_n]}|\mathbf{J}_7(t)|\leq C\varepsilon_n\mathbb{E}\left(\int_0^{T\wedge\tau_n}\int_{\mathbb{B}}\varepsilon_n^{-1}\varphi_{\varepsilon_n}(s,l)\|\mathbf{G}(l,Y_n)\|_{\mathbb{H}^1}^2\|\nabla Z_n\|_{\mathbb{L}^2}^2\nu(\mathrm{d}l)\mathrm{d}s\right)^{\frac{1}{2}}\\
&\leq C_R\varepsilon_n^{\frac{1}{2}}+C\varepsilon_n^{\frac{1}{2}}\mathbb{E}\left(\int_0^{T\wedge\tau_n}\|Z_n\|_{\mathbb{H}^1}^2\int_{\mathbb{B}}|l|^2\varphi_{\varepsilon_n}(s,l)\nu(\mathrm{d}l)\mathrm{d}s\right),
\end{split}
\end{equation}
and
\begin{equation}\label{sys66666-20}
\begin{split}
&\mathbb{E}\sup_{t\in[0,T\wedge\tau_n]}|\mathbf{J}_8(t)|\leq \varepsilon_n^2\mathbb{E}\left(\sup_{t\in[0,T\wedge\tau_n]}\left|\int_0^t\int_{\mathbb{B}}\|\nabla\mathbf{G}(l,Y_n)\|_{\mathbb{L}^2}^2\tilde{\eta}^{\varepsilon^{-1}_n\varphi_{\varepsilon_n}}(\mathrm{d}s,\mathrm{d}l)\right|\right)\\
&+\varepsilon_n^2\mathbb{E}\left(\sup_{t\in[0,T\wedge\tau_n]}\left|\int_0^t\int_{\mathbb{B}}\|\nabla\mathbf{G}(l,Y_n)\|_{\mathbb{L}^2}^2\nu^{\varepsilon^{-1}_n\varphi_{\varepsilon_n}}(\mathrm{d}l)\mathrm{d}s\right|\right)\\
&\leq C_R(\varepsilon_n^{\frac{3}{2}}+\varepsilon_n).
\end{split}
\end{equation}
Now plugging \eqref{sys66666-16}-\eqref{sys66666-20} into \eqref{sys66666-15-2}, we infer that
\begin{equation*}\label{sys66666-21}
\begin{split}
&\mathbb{E}\sup_{t\in[0,T\wedge\tau_n]}\|Z_n(t)\|_{\mathbb{H}^1}^2+\mathbb{E}\left(\int_0^{T\wedge\tau_n}\|Z_n(s)\|_{\mathbb{H}^3}^2\,\mathrm{d}s\right)\leq C_R(\varepsilon_n^{\frac{3}{2}}+\varepsilon_n+\varepsilon_n^{\frac{1}{2}})\\
&+C_R\varepsilon_n^{\frac{1}{2}}\int_0^{T\wedge\tau_n}\mathbb{E}\left(\sup_{r\in[0,s\wedge\tau_n]}\|Z_n(r)\|_{\mathbb{H}^1}^2\right)\int_{\mathbb{B}}|l|^2\varphi_{\varepsilon_n}(s,l)\nu(\mathrm{d}l)\,\mathrm{d}s.
\end{split}
\end{equation*}
Applying the Gronwall lemma again, we infer that
\begin{equation*}\label{sys66666-22}
\begin{split}
&\mathbb{E}\sup_{t\in[0,T\wedge\tau_n]}\|Z_n(t)\|_{\mathbb{H}^1}^2+\mathbb{E}\left(\int_0^{T\wedge\tau_n}\|Z_n(s)\|_{\mathbb{H}^3}^2\,\mathrm{d}s\right)\\
&\leq C_R(\varepsilon_n^{\frac{3}{2}}+\varepsilon_n+\varepsilon_n^{\frac{1}{2}})e^{C_R\varepsilon_n^{\frac{1}{2}}}\rightarrow0~\textrm{as}~n\rightarrow\infty.
\end{split}
\end{equation*}
This completes the proof.\end{proof}

\begin{lemma}\label{lem666-2} $
\mathcal{J}^{\varepsilon_n}(\varepsilon_n\eta^{\varepsilon_n^{-1}\varphi_{\varepsilon_n}})-\mathcal{J}^{0}(\varphi_{\varepsilon_n})~\textrm{converges to}~0~\textrm{in probability}.
$
\end{lemma}
\begin{proof}[\emph{\textbf{Proof}}] Using the result of Lemma \ref{lem666-1*} and inequality \eqref{adddd} and applying the Chebyshev inequality, for $\delta>0$ and $\varepsilon>0$ we have
\begin{equation*}\label{sys6677-2}
\begin{split}
&\mathbb{P}\left(\sup_{t\in[0,T]}\|Y_n(t)-y_n(t)\|_{\mathbb{H}^1}^2+\int_0^{T}\|Y_n(t)-y_n(t)\|_{\mathbb{H}^3}^2\,\mathrm{d}t\geq \delta\right)\\
&\leq\mathbb{P}\left(\left\{\sup_{t\in[0,T]}\|Y_n(t)-y_n(t)\|_{\mathbb{H}^1}^2+\int_0^{T}\|Y_n(t)-y_n(t)\|_{\mathbb{H}^3}^2\,\mathrm{d}t\geq \delta\right\}\cap\{\tau_n=T\}\right)\\
&+\mathbb{P}\left(\left\{\sup_{t\in[0,T]}\|Y_n(t)-y_n(t)\|_{\mathbb{H}^1}^2+\int_0^{T}\|Y_n(t)-y_n(t)\|_{\mathbb{H}^3}^2\,\mathrm{d}t\geq \delta\right\}\cap\{\tau_n<T\}\right)\\
&\leq\frac{\mathbb{E}\left(\sup_{t\in[0,T\wedge\tau_n]}\|Y_n(t)-y_n(t)\|_{\mathbb{H}^1}^2+\int_0^{T\wedge\tau_n}\|Y_n(t)-y_n(t)\|_{\mathbb{H}^3}^2\,\mathrm{d}t\right)}{\delta}\\
&+\mathbb{P}\left(\sup_{t\in[0,T]}\|Y_n(t)\|_{\mathbb{H}^1}^2+\int_0^{T}\|Y_n(t)\|_{\mathbb{H}^3}^2\,\mathrm{d}t\geq R\right)\\
&+\mathbb{P}\left(\sup_{t\in[0,T]}\|y_n(t)\|_{\mathbb{H}^1}^2+\int_0^{T}\|y_n(t)\|_{\mathbb{H}^3}^2\,\mathrm{d}t\geq R\right)\\
&\leq\frac{\mathbb{E}\left(\sup_{t\in[0,T\wedge\tau_n]}\|Y_n(t)-y_n(t)\|_{\mathbb{H}^1}^2+\int_0^{T\wedge\tau_n}\|Y_n(t)-y_n(t)\|_{\mathbb{H}^3}^2\,\mathrm{d}t\right)}{\delta}+\frac{C}{R}\\
&\leq\varepsilon.
\end{split}
\end{equation*}
Here for the above estimate, we have use the fact that
\begin{equation*}
\begin{split}
&\frac{\mathbb{E}\left(\sup_{t\in[0,T\wedge\tau_n]}\|Y_n(t)-y_n(t)\|_{\mathbb{H}^1}^2+\int_0^{T\wedge\tau_n}\|Y_n(t)-y_n(t)\|_{\mathbb{H}^3}^2\,\mathrm{d}t\right)}{\delta}\leq\frac{\varepsilon}{2},
\end{split}
\end{equation*}
and $\frac{C}{R}\leq\frac{\varepsilon}{2}$ if $N_0$ and $R$ big enough. The proof is thus complete.\end{proof}

Condition 2 is a consequence of the following Lemma.
\begin{lemma}\label{lem666-3}
$
\mathcal{J}^{\varepsilon_n}(\varepsilon_n\eta^{\varepsilon_n^{-1}\varphi_{\varepsilon_n}})~\textrm{converges in law to}~\mathcal{J}^0(\varphi)~\textrm{in}~\mathcal{Z}_T.
$
\end{lemma}
\begin{proof}[\emph{\textbf{Proof}}] According to  Lemma \ref{lem666-1}, we have
$
\mathcal{J}^0(\varphi_{\varepsilon_n})~\textrm{converges strongly to}~\mathcal{J}^0(\varphi)~\textrm{in}~\mathcal{Z}_T,
$
which implies that $\mathscr{L}(\mathcal{J}^0(\varphi_{\varepsilon_n}))$ converges weakly to $\mathscr{L}(\mathcal{J}^0(\varphi))$. Since $S^N$ is a separable metric space, by the Skorokhod theorem, there exists a new probability space $(\tilde{\Omega},\tilde{\mathcal{F}},\tilde{\mathbb{P}})$ and on this space, there exist random variables $\tilde{\varphi}_{\varepsilon_n}$ and $\tilde{\varphi}$, which have the same laws as $\varphi_{\varepsilon_n}$ and $\varphi$ respectively with $\tilde{\varphi}_{\varepsilon_n}\rightarrow\tilde{\varphi}$ in $S^N$, $\tilde{\mathbb{P}}$-a.s. Then
$
\mathcal{J}^0(\tilde{\varphi}_{\varepsilon_n})~\textrm{pointwise converges to}~\mathcal{J}^0(\tilde{\varphi})~\textrm{in}~\mathcal{Z}_T.
$
Additionally, from the second part of Lemma \ref{lem666-1}, we have
\begin{equation*}
\begin{split}
&\mathscr{L}(\mathcal{J}^0(\tilde{\varphi}_{\varepsilon_n}))=\mathscr{L}(\mathcal{J}^0(\varphi_{\varepsilon_n}))~\textrm{and}~\mathscr{L}(\mathcal{J}^0(\tilde{\varphi}))=\mathscr{L}(\mathcal{J}^0(\varphi)).
\end{split}
\end{equation*}
Since by Lemma \ref{lem666-2}, we can choose a subsequence (keeping the same notation) such that
\begin{equation*}
\begin{split}
\mathcal{J}^{\varepsilon_n}(\varepsilon_n\eta^{\varepsilon_n^{-1}\tilde{\varphi}_{\varepsilon_n}})-\mathcal{J}^{0}(\tilde{\varphi}_{\varepsilon_n})~\textrm{converges to}~0,~\tilde{\mathbb{P}}\textrm{-a.s}.
\end{split}
\end{equation*}
Thus for any bounded and globally Lipschitz continuous function $K:\mathcal{Z}_T\rightarrow\mathbb{R}$ we have
\begin{equation*}
\begin{split}
&\left|\int_{\mathcal{Z}_T}K(x)\mathrm{d}\mathscr{L}(\mathcal{J}^{\varepsilon_n}(\varepsilon_n\eta^{\varepsilon_n^{-1}\varphi_{\varepsilon_n}}))-\int_{\mathcal{Z}_T}K(x)\mathrm{d}\mathscr{L}(\mathcal{J}^0(\varphi))\right|\\
&=\left|\int_{\mathcal{Z}_T}K(x)\mathrm{d}\mathscr{L}(\mathcal{J}^{\varepsilon_n}(\varepsilon_n\eta^{\varepsilon_n^{-1}\tilde{\varphi}_{\varepsilon_n}}))-\int_{\mathcal{Z}_T}K(x)\mathrm{d}\mathscr{L}(\mathcal{J}^0(\tilde{\varphi}))\right|\\
&\leq\int_{\tilde{\Omega}}\left|K(\mathcal{J}^{\varepsilon_n}(\varepsilon_n\eta^{\varepsilon_n^{-1}\tilde{\varphi}_{\varepsilon_n}}))-K(\mathcal{J}^{0}(\tilde{\varphi}_{\varepsilon_n}))\right|\mathrm{d}\tilde{\mathbb{P}}+\int_{\tilde{\Omega}}\left|K(\mathcal{J}^{0}(\tilde{\varphi}_{\varepsilon_n}))-K(\mathcal{J}^{0}(\tilde{\varphi}))\right|\mathrm{d}\tilde{\mathbb{P}}\\
&\rightarrow0~\textrm{as}~n\rightarrow\infty.
\end{split}
\end{equation*}
Here the above result comes from the Dominated convergence theorem. \end{proof}

\section{Appendix}
Let $(\mathbb{S},\varrho)$ be a complete and separable metric space. Let $(X_n)_{n\in\mathbb{N}}$ be a sequence of c\`{a}dl\`{a}g $\mathbb{F}$-adapted $\mathbb{S}$-valued processes.
\begin{definition}\label{def4-1}(\cite{metivier1988stochastic}) Let $f\in\mathbb{ D}([0,T];\mathbb{S})$ and let $\delta>0$ be given. A modulus of $f$ is defined by
\begin{equation*}
\begin{split}
w_{[0,T],\mathbb{S}}(f,\delta):=\inf_{\Pi_{\delta}}\max_{t_i\in\omega}\sup_{t_i\leq s<t\leq t_{i+1}\leq T}\varrho(f(t),f(s)),
\end{split}
\end{equation*}
where $\Pi_{\delta}$ is the set of all increasing sequences $\omega=\{0=t_0<t_1<...<t_n=T\}$ with the property
$
t_{i+1}-t_i\geq\delta,~i=0,~1,...,~n-1.
$
\end{definition}

Analogous to the Arzel\`{a}-Ascoli theorem for the space of continuous functions, we introduce the following criterion for relative compactness of a subset of the space $\mathbb{D}([0,T];\mathbb{S})$.
\begin{lemma}\label{lem4-1} (\cite{joffe1986weak,metivier1988stochastic}) A set $\mathcal{K}\subset \mathbb{D}([0,T];\mathbb{S})$ has compact closure if and only if it satisfies the following conditions:
\begin{enumerate}
\item[(1)] there exists a dense subset $I\subset[0,T]$ such that for every $t\in I$ the set $f(t),~f\in\mathcal{ K}$ has compact closure in $\mathbb{S}$;
\item[(2)] $\lim_{\delta\rightarrow0}\sup_{f\in\mathcal{K}}w_{[0,T],\mathbb{S}}(f,\delta)=0$.
\end{enumerate}
\end{lemma}

The following Lemma provides a useful consequence of the Aldous condition.
\begin{lemma}\label{lem4-2} (\cite{motyl2013stochastic}) Assume that $(X_n)_{n\in\mathbb{N}}$ satisfies the Aldous condition . Let $\mathscr{L}(X_n)$ be the law of $X_n$ on $\mathbb{D}([0,T];\mathbb{S})$. Then for every $\varepsilon>0$ there exists a subset $A_{\varepsilon}\subset \mathbb{D}([0,T];\mathbb{S})$ such that $\sup_{n\in\mathbb{N}}\mathscr{L}(X_n)(A_{\varepsilon})\geq1-\varepsilon$, and $\lim_{\delta\rightarrow0}\sup_{f\in A_{\varepsilon}}w_{[0,T],\mathbb{S}}(f,\delta)=0$.
\end{lemma}

Let $\mathbb{B}^1:=\{f\in \mathbb{H}^1:\|f\|_{\mathbb{H}^1}\leq r^1\}$. Let $\mathbb{B}_w^1$ denotes the ball $\mathbb{B}^1$ endowed with the weak topology. It is clear that $\mathbb{B}_w^1$ is metrizable \cite{brezis1983analyse}. Let us consider the following space
\begin{equation*}
\begin{split}
\mathbb{D}([0,T];\mathbb{B}^1_w):=\{f\in \mathbb{D}([0,T];\mathbb{H}_w^1):\sup_{t\in[0,T]}\|f(t)\|_{\mathbb{H}^1}\leq r^1\}.
\end{split}
\end{equation*}
A criterion for convergence of a sequence in $\mathbb{D}([0,T];\mathbb{B}^1_w)$ is as follows.
\begin{lemma}\label{adlem4-1} (\cite{motyl2013stochastic,brzezniak2019weak}) Let $f_n:[0,T]\rightarrow \mathbb{H}^1,~n\in\mathbb{N}$, be functions such that
\begin{enumerate}
\item[(1)] $\sup_{n\in\mathbb{N}}\sup_{t\in[0,T]}\|f_n(t)\|_{\mathbb{H}^1}\leq r^1$;
\item[(2)] $f_n\rightarrow f$ in $\mathbb{D}([0,T];(\mathbb{H}^{\beta_1})^*)$,
\end{enumerate}
then $f_n,~f\in\mathbb{D}([0,T];\mathbb{B}_w^1)$ and $f_n\rightarrow f$ in $\mathbb{D}([0,T];\mathbb{B}_w^1)$ as $n\rightarrow\infty$.
\end{lemma}

\section*{Acknowledgements}
This work was partially supported by the National Natural Science Foundation of China (Grant No. 12231008), and the National Key Research and Development Program of China (Grant No.  2023YFC2206100).

\bibliographystyle{plain}%
\bibliography{SLLBar}

\end{document}